\documentclass[11pt,reqno]{amsart}
\usepackage{amsmath}
\usepackage{amssymb}
\usepackage{color}
\usepackage{eucal}
\usepackage{color}
\usepackage{tikz}
\usepackage{tikz-cd}
\usepackage[all,arc,poly]{xy}
\usetikzlibrary{calc}
\usepackage{changes}
\usepackage{calligra}
\DeclareMathAlphabet{\mathcalligra}{T1}{calligra}{m}{n} 
\usetikzlibrary{matrix,arrows,decorations.markings}
\usetikzlibrary{decorations.pathmorphing}
\makeatletter
\@namedef{subjclassname@2020}{%
  \textup{2020} Mathematics Subject Classification}
\makeatother
\newcommand{\circt}{{{\mathop{\hbox{$\bigcirc$\kern-7.5pt\raise-0.5pt\hbox{$t$}\kern3.8pt}}}}}

\newcommand{\circu}{{{\mathop{\hbox{$\bigcirc$\kern-8.5pt\raise0.5pt\hbox{$u$}\kern2.8pt}}}}}

\newcommand{\circv}{{{\mathop{\hbox{$\bigcirc$\kern-8.5pt\raise0.5pt\hbox{$v$}\kern2.8pt}}}}}

\newcommand{\circw}{{{\mathop{\hbox{$\bigcirc$\kern-9.5pt\raise0.5pt\hbox{$w$}\kern1.5pt}}}}}

\newcommand{\circs}{{{\mathop{\hbox{$\bigcirc$\kern-8.5pt\raise0.5pt\hbox{$s$}\kern3.5pt}}}}}

\numberwithin{equation}{section}
\theoremstyle{plain}
\newtheorem{Thm}{Theorem}[section]
\newtheorem{Prop}[Thm]{Proposition}
\newtheorem{Lemma}[Thm]{Lemma}
\newtheorem{Cor}[Thm]{Corollary}

\theoremstyle{definition}
\newtheorem{Def}[Thm]{Definition}
\newtheorem{Cond}[Thm]{Condition}
\newtheorem{Problem}[Thm]{Problem}
\newtheorem{Rmk}[Thm]{Remark}
\newtheorem{Prov}[Thm]{Proviso}
\DeclareSymbolFont{rsfscript}{OMS}{rsfs}{m}{n}
\DeclareSymbolFontAlphabet{\mathrsfs}{rsfscript}
\newcommand{\co}{\operatorname{co}}
\newcommand{\inv}{^{-1}}
\newcommand{\til}[1]{\widetilde {#1}}
\newcommand{\ol}[1]{\overline{#1}}
\newcommand{\rbl}{{\xrightarrow{\hspace*{1cm}}}}
\newcommand{\lbl}{{\xleftarrow{\hspace*{1cm}}}}
\newcommand{\nc}{\newcommand}
\nc{\cT}{\mathrsfs T}
\nc{\bP}{\mathbb{P}}
\nc{\bO}{\mathbb{O}}
\nc{\mC}{\mathcal{C}}
\nc{\mQ}{\mathcal{Q}}
\nc{\mE}{\mathcal{E}}
\nc{\mF}{\mathcal{F}}
\nc{\mL}{\mathcal L}
\nc{\mB}{\mathcal B}
\nc{\mK}{\mathcal K}
\nc{\mH}{\mathcal H}
\nc{\mG}{\mathcal G}
\nc{\mA}{\mathcal A}
\nc{\mM}{\mathcal M}
\nc{\mU}{\mathcal U}
\nc{\mX}{\mathcal X}
\nc{\mY}{\mathcal Y}
\nc{\mZ}{\mathcal Z}
\nc{\mD}{\mathcal D}
\nc{\sfCE}{\mathsf{CE}}
\nc{\sfCL}{\mathsf{CL}}
\nc{\nothing}{\rule{0em}{1ex}}
\nc{\nt}{\nothing}
\nc{\highnothing}{\rule{0em}{3ex}}
\nc{\hnt}{\highnothing}
\nc{\look}{\marginpar{$\bullet$}}
\nc{\ssc}{\scriptscriptstyle}
\nc{\pss}[1]{\nt^{\ssc \subsetneq}\! #1}
\nc{\red}[1]{\textcolor{red}{#1}}
\nc{\blue}[1]{\textcolor{blue}{#1}}
\nc{\black}[1]{\textcolor{black}{#1}}
\newenvironment{romanenumerate}
{\begin{list}{(\roman{enumi})}{\usecounter{enumi}
\setlength{\labelwidth}{2cm}
\setlength{\itemindent}{0pt}
\setlength{\itemsep}{0.5\itemsep}
\setlength{\topsep}{\itemsep}
\setlength{\parsep}{0pt}
}}{\end{list}}

\begin{document}

\title[{Finite approximation of free groups I}]{{Finite approximation of free groups I:\\ the $F$-inverse cover problem}}

\author{K.~Auinger, J.~Bitterlich and M.~Otto}
\address{Fakult\"at f\"ur Mathematik, Universit\"at Wien, Oskar-Morgenstern-Platz 1, A-1090 Wien, Austria}
\email{karl.auinger@univie.ac.at}
\address{   }
\email{bitt.j@protonmail.com}
\address{Department of Mathematics,
Technische Universit\"at Darmstadt,
Schlossgartenstrasse 7,
D-64289 Darmstadt,
Germany}
\email{otto@mathematik.tu-darmstadt.de}

\begin{abstract}
For a finite connected graph $\mE$ with
edge set $E$, a finite $E$-generated group $G$ is constructed such that the set of relations 
$p=1$ satisfied by  $G$ (with $p$ a word over $E\cup E^{-1}$)
is closed under deletion of
generators (i.e.~edges); 
as a consequence,  every element $g\in G$ admits a unique minimal set
$\mathrm{C}(g)$ of edges (the \emph{content} of $g$) needed to
represent $g$ as a word over
$\mathrm{C}(g)\cup\mathrm{C}(g)^{-1}$. The crucial property of the
group $G$  is that
connectivity in the graph $\mE$ is
reflected in $G$ in the following sense: if a word $p$ forms a path $u\longrightarrow v$ in $\mE$ then there exists a $G$-equivalent word $q$ which also forms a path $u\longrightarrow v$ and uses only edges
from their common content; in particular, the content of the corresponding group
element $[p]_G=[q]_G$ spans a connected subgraph of $\mE$ containing
the vertices $u$ and $v$. 
As the free group generated by $E$ obviously has these properties, the construction provides another instance of how certain
features of free groups can be ``approximated'' or ``simulated''
in finite groups.
As an application it is shown that every finite inverse monoid admits a finite $F$-inverse cover. This solves a long-standing problem of Henckell and Rhodes.
\end{abstract}
\maketitle

\section{Introduction}\label{sec:intro} In the influential paper~\cite{HR}, Henckell
and Rhodes stated a series of conjectures and two problems. The paper
was concerned with the celebrated question whether every finite block
group $M$ (a monoid in which every von Neumann regular element
admits a unique inverse) is a quotient of a submonoid of  the power monoid
$\mathfrak{P}(G)$ of some finite group $G$.
Henckell and Rhodes presented an
affirmative answer to the question modulo some conjecture, namely
about the structure of {pointlike sets};  a  subset $X$ of a finite
monoid $M$ is \emph{pointlike} (with respect to groups) if and only if
in \textsl{every} subdirect product $T\subseteq M\times G$ of $M$ with
any  finite group $G$ there exists an element $g\in G$ with $X\times
\{g\}\subseteq T$ (that is, all elements of $X$ \textsl{relate} to
some point $g\in G$). 
{The questions raised by Henckell and Rhodes in~\cite{HR} concerned the algorithmic recognisability of certain subsets of $M$ and relations on $M$ for a given finite monoid $M$. These subsets and relations  are defined by use of the collection of all subdirect products $T\subseteq M\times G$ of $M$ with arbitrary finite groups $G$.}

Shortly after, all stated conjectures and one of the two problems
(about \emph{liftable tuples})
were verified respectively solved by Ash in his celebrated paper
\cite{Ash}. Roughly speaking, Ash proved  that {in the situations mentioned, and even beyond those,} the  collection of all subdirect products $T\subseteq M\times G$ of $M$ with finite groups $G$ has the same ``computational power'' as a particularly chosen ``canonical'' subdirect product $\tau\subseteq M\times F$ of $M$ with some free group $F$. This is a strong form of approximation in finite groups of the free group $F$. The algorithmic recognisability of the aforementioned subsets and relations of $M$ is an immediate consequence. The importance of {Ash's} paper went beyond its
immediate task as in the following years interesting and deep
connections with the profinite topology of the free group
\cite{RibesZalesskii} and
 model theory~\cite{HerwigLascar} have been revealed and studied
\cite{almeidadelgado1, almeidadelgado2}. 

Yet  the second 
problem stated, which was called in~\cite{HR} a ``stronger form of the
pointlike conjecture for inverse monoids'',
was not solved in  Ash's paper and has since then attracted considerable 
attention~\cite{Lawsonglasgow, Lawson, sze1, sze2, ASz, szaszen, sza, billh}. It asked:
\begin{Problem}\label{problem:Finverse} Does every finite inverse monoid admit a finite $F$-inverse cover?
\end{Problem}
An inverse monoid $S$ 
is \emph{$F$-inverse} if every congruence class of
the least group congruence $\sigma$ of $S$
admits a greatest element (with respect to the natural partial order)
and an inverse monoid 
$S$ is a \emph{cover} of an inverse monoid
$M$ if there exists a surjective, idempotent separating homomorphism
$S \to M$.

{When reading the paper~\cite{ASz} by Szendrei and the first author, the second author understood that a result by the third
author~\cite{otto1, otto2} about the existence of certain finite
groupoids can be used to give an affirmative answer 
to Problem~\ref{problem:Finverse}. He presented this solution in his dissertation~\cite{bitterlichdiss} and  his
paper~\cite{bitterlich}.}
Later, some  flaws were discovered in~\cite{otto1, otto2}
which, however,  have been fixed in the meantime~\cite{otto3}. The
intention of the present paper is to give a complete and self-contained
presentation of the solution to Problem~\ref{problem:Finverse} (up to
classical results on inverse monoids), which is based on the ideas and
proofs of~\cite{otto3} but  is in a sense tailored for what is needed
in the present context and presented in a language which
(hopefully) makes it more accessible to the semigroup community. 

{Since an infinite $F$-inverse cover can be constructed for every inverse monoid $M$ by use of a free group $F$, the task for finite $M$ is,
to replace $F$ by a suitable finite group $H$. The group $H$ needs to have a sufficiently high combinatorial complexity in order to ``simulate''
the required behaviour of the free group $F$
with respect to the monoid $M$. Hence the task is to \textsl{approximate} the free group $F$ sufficiently well by a finite group $H$. What this exactly means in the present context is one of the essentials of the paper.}

{The paper is organised as follows: Section~\ref{sec:inverse monoids}
collects prerequisites from inverse monoids, graphs and a proof that
the existence of certain finite groups yields
a positive solution of Problem~\ref{problem:Finverse}. }
Section~\ref{sec:machinery} introduces the main graph-theoretic
tools while Section~\ref{sec:2results} presents two crucial technical
results. Finally, in
Section~\ref{sec:groupoids} we obtain the required group in a construction which intends to ``reflect the geometry'' of a given finite graph {$\mE$ and thereby prove the main result of the paper (Lemma~\ref{thm:main theorem})}.

\section{Inverse monoids}\label{sec:inverse monoids}

\subsection{Preliminaries} A monoid $M$ is \emph{inverse} if every
element $x\in M$
admits a unique element $x^{-1}$, called the \emph{inverse} of $x$, satisfying $xx^{-1}x=x$ and $x^{-1}xx^{-1}=x^{-1}$. This gives rise to a unary operation ${}^{-1}\colon M\to M$ and an inverse monoid may equivalently  be defined as an algebraic structure $(M;\cdot,{}^{-1},1)$ with $\cdot$ an associative binary operation, $1$ a neutral element with respect to $\cdot$ and a unary operation ${}^{-1}$ satisfying the laws 
\[{(x^{-1})}^{-1}=x,\ (xy)^{-1}=y^{-1}x^{-1},\ xx^{-1}x=x \mbox{ and }xx^{-1}yy^{-1}=yy^{-1}xx^{-1}.\] 
In particular, the class of all inverse monoids forms a variety of
algebraic structures (in the sense of universal algebra), the variety
of all groups $(G;\cdot,{}^{-1},1)$ being a subvariety. 
By the Wagner--Preston Theorem~\cite[Chapter 1, Theorem 1]{Lawson}, inverse monoids may as well be characterised as monoids of partial bijections on a set, closed under composition of partial mappings and inversion. Therefore, while groups  model symmetries of mathematical structures,  inverse monoids (or semigroups) model partial symmetries, that is, symmetries between substructures of mathematical structures.

From basic
facts of universal algebra it follows that every inverse monoid $M$
admits a least congruence such that the corresponding quotient
structure is a group. This congruence is usually denoted $\sigma$ and
it can be characterised as the least congruence on $M$
that identifies
all idempotents of $M$ with each other. Another way to characterise
this congruence is this:
two elements $x,y\in M$ are $\sigma$-related if and only if $xe=ye$ for some idempotent $e$ of $M$ (and this is equivalent to  $fx=fy$ for some idempotent $f$ of $M$).

Every inverse monoid $M$ is equipped with a partial order $\le$, the
\emph{natural order},  defined by $x\le y$ if and only if $x=ye$ for
some idempotent $e$ of $M$ (this is equivalent to $x=fy$ for some
idempotent $f$ of $M$).
In particular, $e\le 1$ for every idempotent $e$ of $M$. 
If an inverse monoid $M$ is represented as a monoid of partial
bijections, then
the idempotents of $M$ are exactly the restrictions of the identity function and for $x,y\in M$ we have $x\le y$ if and only if $x\subseteq y$, that is, $x$ is a restriction of $y$.
The order is compatible with the binary
operation  and inversion of $M$ where the latter means that $x\le y$ implies $x^{-1}\le y^{-1}$. 
In terms of the natural order, the congruence
$\sigma$ can be characterised as the least congruence for which the
natural order on the quotient is the identity relation, and, likewise
as the least congruence
that identifies every pair of $\le$-comparable elements. This leads to yet another description of $\sigma$: two elements $x$ and $y$ are $\sigma$-related if and only if they admit a common lower bound with respect to $\le$. For further information on inverse monoids the reader is referred to the monographs by Petrich~\cite{Petrich} and Lawson~\cite{Lawson}.

An inverse monoid $S$
is \emph{$F$-inverse} if every $\sigma$-class
of $S$ 
possesses a greatest element with respect to $\le$. For recent
developments concerning the systematic study of $F$-inverse monoids
and their relevance in various contexts the reader is referred to~\cite{AKS} and the literature cited there. An $F$-inverse monoid
$S$ is an \emph{$F$-inverse cover} of the inverse monoid $M$ if there
exists a surjective idempotent separating homomorphism
$S \to M$. 
As mentioned in the introduction, it has been an outstanding open problem whether every finite inverse monoid $M$ admits a finite $F$-inverse cover.

\subsection{$A$-generated inverse monoids}\label{sec:Agenerated}
Throughout, for any non-empty set $X$ (of letters, of edges, etc.) we
let $X\inv:=\{x\inv \colon x\in X\}$ be a disjoint copy of $X$
consisting of formal inverses of the elements of $X$, and set 
$\til{X}:=X\cup X\inv$.
The mapping $x\mapsto x\inv$ is extended to an
involution of $\til{X}$ by setting ${(x\inv )}\inv=x$, for all $x\in
X$. We let $\til{X}^*$ be the free monoid
over $\til{X}$, which, subject to $(x_1\cdots x_n)\inv=x_n\inv\cdots
x_1\inv$ (where $x_i\in \til{X}$), is the \emph{free involutory
  monoid} over $X$.
The elements of $\til{X}^*$ are called \emph{words over}
$\til{X}$, and we let $1$ denote the \emph{empty word}.
A word $w\in \til X^*$ is \emph{reduced} if it does not contain any
factor of the form $xx^{-1}$ for $x\in \til  X$.
Repeated deletion of such factors in a word $w$, until no one is
present any more, leads to \emph{the reduced form} $\mathrm{red}(w)$ of $w$.

We fix a non-empty set $A$ (called alphabet in this context).
An inverse monoid $M$ together with a (not necessarily injective)
mapping $i_M\colon A\to M$ (called \emph{assignment function}) is an
\emph{$A$-generated} inverse monoid if $M$ is generated by $i_M(A)$
as an inverse monoid, that is, generated
with respect to the operations $1,\cdot,{}^{-1}$. For every congruence
$\rho$ of an $A$-generated inverse monoid $M$, the quotient $M/\rho$
is  $A$-generated with respect to the map $i_{M/\rho}=\pi_\rho\circ
i_M$ where $\pi_\rho$ is the projection $M\to M/\rho$.
A \emph{morphism $\psi$ from the $A$-generated inverse monoid $M$ to the $A$-generated inverse monoid $N$}
is a homomorphism $M\to N$ {respecting generators from $A$}, that is,
satisfying $i_N=\psi\circ i_M$.
If it exists, such a morphism is unique and surjective and is called 
\emph{canonical morphism},
denoted $\psi\colon M\twoheadrightarrow N$. 
{On a more formal level, an $A$-generated inverse monoid is an algebraic structure of the form $(M;\cdot,\inv,1,A)$ where every symbol $a\in A$ is interpreted in $M$ as a constant (that is, as a nullary operation) via the assignment function $i_M$. Canonical morphisms of $A$-generated inverse monoids then are just homomorphisms of algebraic structures in the signature $\{\cdot,\inv,1\}\cup A$.}
{If $M\twoheadrightarrow N$ then $M$ is an
\emph{expansion of $N$}. In our usage, the term ``expansion'' just concerns the relationship between two individual $A$-generated inverse monoids $M$ and $N$. This somehow deviates from the widespread use of that term standing for a functor on certain categories of monoids. }
The special case of $A$-generated groups will play a significant r\^ole in this paper. 

As already mentioned, the assignment function is not necessarily injective, and, what is more, some generators may even be sent to the identity element of $M$.
This is not a deficiency, but rather is adequate
in our context,
since we want the quotient of an $A$-generated structure to be again
$A$-generated.
In particular  $M/\sigma$, the quotient of an $A$-generated inverse monoid $M$
modulo the least group congruence $\sigma$, is an $A$-generated group.

The assignment function $i_M$ is usually not explicitly mentioned; it uniquely extends 
to a homomorphism $[\ \ ]_M\colon \tilde{A}^*\to M$ (of involutory monoids). For every word $p\in \tilde{A}^*$, $[p]_M$  is the \emph{value of  $p$ in $M$} or simply the \emph{$M$-value of $p$}. For two words $p,q\in \til{A}^*$, the $A$-generated inverse monoid $M$ \emph{satisfies the relation} $p=q$ if $[p]_M=[q]_M$, in which case the words $p$ and $q$ are \emph{$M$-equivalent}, while $M$ \emph{avoids the relation} $p=q$ if $[p]_M\ne[q]_M$.

Using the concept of ``$A$-generatedness'' we see that every inverse monoid admits an $F$-inverse cover. Indeed, let $M$ be an inverse monoid; choose a set $A$ with assignment function $i_M\colon A\to M$ so that $M$ becomes $A$-generated and  let $F$ be the free $A$-generated group. Then the subdirect product
\begin{equation}\label{eq:infinite F-inverse cover}
S:=\{([w]_F,[w]_M)\colon w\in \til{A}^*\}\subseteq F\times M
\end{equation}
is  an $F$-inverse cover of $M$.
This is well known and it is easy to see.
{Indeed, the congruence $\sigma$ on $S$ can be described by 
\[([u]_F,[u]_M)\mathrel{\sigma}([v]_F,[v]_M)\mbox{ if and only if
  }[u]_F=[v]_F. \]
Furthermore, for two words $u,v\in \til A^*$ for which $u$ is 
obtained from $v$ by (successive) deletion of some factors of the
form $aa\inv$ ($a\in \til A$) we have 
$[v]_M\le [u]_M$ in any $A$-generated inverse monoid $M$.
Consequently, for a given word $w\in \til A^*$ the maximum element of
the $\sigma$-class of $([w]_F,[w]_M)$ is the element
$([\mathrm{red}(w)]_F,[\mathrm{red}(w)]_M)=
([w]_F,[\mathrm{red}(w)]_M)$.
}
However, the inverse monoid $S$ is infinite, no matter what $M$ is. The Henckell--Rhodes problem then asks if in case of a finite inverse monoid $M$ the infinite free group $F$ in \eqref{eq:infinite F-inverse cover} may  be replaced  by some finite ($A$-generated) group $H$ serving the same purpose. An affirmative answer to this question will be established in Theorem~\ref{thm:groupoid-premorphism}.

\subsection{Graphs}
In this paper, we consider {the Serre definition~\cite{Serre}} of {graph}
structures, admitting multiple directed edges between pairs of
vertices and 
including directed loops at individual vertices. In the literature, such
structures are often called \emph{{multidigraphs,} directed multigraphs} or
\emph{quivers}.
The following formalisation is convenient for our purposes. 
A \emph{graph} $\mE$ is a \textsl{two-sorted structure} $(V, K;\alpha,\omega,{}^{-1})$ with $V$ its set of \emph{vertices}, $K$ its set of \emph{edges} (disjoint from $V$),  with \emph{incidence functions} $\alpha\colon K\to V$ and $\omega\colon K\to V$, selecting, for each edge $e$ the \emph{initial} vertex $\alpha e$ and the \emph{terminal} vertex $\omega e$,
and \emph{involution} ${}\inv\colon K\to K$ satisfying $\alpha e=\omega e\inv$, $\omega e=\alpha e\inv$ and $e\ne e^{-1}$ for every edge $e\in K$. Instead of \textsl{initial/terminal vertex} the terms \textsl{source/target} are also used in the literature. One should think of an edge $e$ with $\alpha e=u$ and $\omega e=v$ in ``geometric'' terms as $e\colon \underset{u}{\bullet}\!\overset{}{\rbl}\underset{v}{\bullet}$ and its inverse $e\inv\colon \underset{u}{\bullet}\!\overset{}{\lbl}\underset{v}{\bullet}$ as ``the same edge but traversed in the opposite direction''. A graph $(V, K;\alpha,\omega,{}^{-1})$ is \emph{oriented} if the edge set $K$ is partitioned as $K=E\cup E^{-1}=\til{E}$ such that every ${}^{-1}$-orbit contains exactly one element of $E$ and one of $E^{-1}$; the edges in $E$ are the \emph{positive} or \emph{positively oriented} edges, those in $E^{-1}$ the \emph{negative} or \emph{negatively oriented} ones. An oriented graph $\mE$ {with set of positive edges $E$} will be denoted as $\mE=(V,\til{E};\alpha,\omega,{}^{-1})$.

A \emph{subgraph} of the graph $\mE$ is a
substructure that is induced over a pair $(V',K')$ of subsets $V'\subseteq V$ and $K'\subseteq K$
both of which are
closed under the operations $\alpha$ and ${}\inv$ (and therefore also
under $\omega$). In particular, every pair $(S,T)$ of subsets $S\subseteq V$ and $T\subseteq  K$
generates a unique
minimal subgraph $\langle (S,T)\rangle$ of $\mE$ containing $(S,T)$,
which is the \emph{subgraph of $\mE$ spanned by $(S,T)$}. An \emph{automorphism}  of a graph $\mE=(V,K;\alpha,\omega,{}^{-1})$ is a pair of maps $\varphi=(\varphi_V,\varphi_K)$ with $\varphi_V\colon V\to V$, $\varphi_K\colon K\to K$ being bijections satisfying for all $e\in K$:
\[\alpha\varphi_K(e)=\varphi_V(\alpha e),\ \omega\varphi_K(e)=\varphi_V(\omega e),\ \varphi_K(e^{-1})=(\varphi_K(e))^{-1}.\]
We note that the second equality is a consequence of the first and third. In the oriented case we require in addition that $\varphi_{\til{E}}(E)=E$ and (therefore also) $\varphi_{\til{E}}(E^{-1})=E^{-1}$.
A benefit from our definition of a graph as a
two-sorted 
functional rather than a relational structure is that there is no
distinction between weak and induced subgraphs and that
concepts like homomorphism, congruence and quotient are easier to handle.

Let $A$ be a finite set; a \emph{labelling} of the graph
$\mathcal{E}=(V, K;\alpha,\omega,{}^{-1})$ \emph{by the alphabet
  $A$} (an $A$-\emph{labelling}, for short) is a mapping $\ell\colon
K\to\til{A}$ respecting the involution: $\ell(e^{-1})=\ell(e)^{-1}$
for all $e\in K$. The labelling $\ell\colon K\to \til{A}$ gives rise
to an {orientation} of $\mE$: setting $E:=\{e\in K \colon \ell(e)\in A\}$
({positive} edges) and  $E^{-1}:=\{e\in K\colon \ell(e)\in A^{-1}\}$
({negative} edges),
it follows that $E\cap E^{-1}=\varnothing$ and we get $K=\til{E}$.

We consider $A$-labelled graphs as structures $(V,
K;\alpha,\omega,{}^{-1},\ell,A)$ in their own right. By a
\emph{subgraph of an $A$-labelled graph} we mean just a subgraph
with the induced labelling.
Morphisms of $A$-labelled graphs are naturally defined as follows. 
Let $\mK=(V, K;\alpha,\omega,{}^{-1},\ell,A)$ and $\mL=(W,
L;\alpha,\omega,{}^{-1},\ell,A)$ be $A$-labelled graphs. \emph{A
  morphism  $\varphi\colon \mK\to \mL$ of $A$-labelled graphs} is a
pair of mappings $\varphi=(\varphi_1,\varphi_2)$, $\varphi_1\colon V\to W$, $\varphi_2\colon K\to L$,
both compatible with the operations $\alpha$ and ${}^{-1}$
(and therefore also $\omega$)
as well as with the labelling. Throughout the paper, in the situation of a morphism $\varphi=(\varphi_1,\varphi_2)\colon \mK\to \mL$ we shall write, for every vertex $v$ [every edge $e$] of $\mK$, $\varphi(v)$ instead of $\varphi_1(v)$ [$\varphi(e)$ instead of $\varphi_2(e)$].

A \emph{congruence $\Theta$ on the $A$-labelled graph $\mK=(V, K;\alpha,\omega,{}^{-1},\ell,A)$} is a pair $(\Theta_V,\Theta_K)$ with $\Theta_V$ an equivalence relation on $V$, $\Theta_K$ an equivalence relation on $K$, compatible with the operations $\alpha$ and ${}^{-1}$ (therefore also $\omega$) and respecting $\ell$, that is:
\[e\mathrel{\Theta_K}f\Longrightarrow \alpha e\mathrel{\Theta_V}\alpha f,\  \omega e\mathrel{\Theta_V}\omega f,\ e^{-1}\mathrel{\Theta_K}f^{-1}\mbox{ for all }e,f\in K\] and
\[e\mathrel{\Theta_K}f \Longrightarrow \ell(e)=\ell(f)\mbox{ for all }e,f\in K.\]
 The definition of the quotient graph $\mathcal{K}/\Theta$ for a congruence $\Theta$ is obvious, and we have the usual  Homomorphism Theorem. As for images under morphisms, the congruence class $v\Theta_V$ of a vertex $v$ [the congruence class $e\Theta_K$ of an edge $e$] will be denoted by $v\Theta$ [by $e\Theta$].

A \emph{non-empty path} $\pi$ in $\mE$ is a sequence $\pi=e_1e_2\cdots e_n$ ($n\ge 1$) of \emph{consecutive} edges (that is $\omega e_i=\alpha e_{i+1}$ for all $1\le i< n$); we set $\alpha\pi:=\alpha e_1$ and $\omega \pi=\omega e_n$ (denoting the initial and terminal vertices of the path $\pi$); the \emph{inverse path} $\pi\inv$ is the path $\pi\inv:=e_n\inv\cdots e_1\inv$; it has initial vertex $\alpha\pi\inv=\omega\pi$ and terminal vertex $\omega\pi\inv=\alpha\pi$. A path $\pi$ is \emph{closed} or a \emph{cycle} if $\alpha\pi=\omega\pi$. We also consider, for each vertex $v$, the \emph{empty path at $v$}, denoted $\varepsilon_v$ for which we set $\alpha\varepsilon_v=v=\omega \varepsilon_v$ and $\varepsilon_v\inv=\varepsilon_v$ (it is convenient to identify $\varepsilon_v$ with the vertex $v$ itself). We say that $\pi$ is a path \emph{from $u=\alpha \pi$ to $v=\omega\pi$}, and we  will also say that $u$ and $v$ are \emph{connected by $\pi$} (and likewise by $\pi\inv$). A graph is connected if any two vertices can be connected by some path.    The subgraph $\langle \pi\rangle$ spanned by the non-empty path $\pi=e_1\cdots e_n$ is the graph spanned by the edges of $\pi$, that is, by the pair $(\varnothing,\{e_1,\dots,e_n\})$; it coincides with $\langle \pi\inv\rangle$; the graph spanned by an empty path $\varepsilon_v$ simply is $\{v\}$ (one vertex, no edge). For a path $e_1\cdots e_k$ in an $A$-labelled graph $\mathcal{E}$, its label  is $\ell(e_1\cdots e_k):=\ell(e_1)\cdots\ell(e_k)$ which is a word in $\til{A}^*$.

\subsection{{Cayley graphs of $A$-generated groups} }
Given an $A$-generated group $Q$ we define the \emph{Cayley graph}
$\mQ$ of $Q$ as follows; 
as an $A$-labelled graph, this graph $\mQ$ depends on the underlying
assignment function $i_Q$:
\begin{itemize}
\item[--] the set of vertices of $\mQ$ is $Q$,
\item[--] the set of edges of $\mQ$ is $Q\times \til{A}$,
and, for $g\in Q,\ a\in \til{A}$, the incidence functions, involution
and labelling are
defined according to 
\[
  \begin{array}{l@{\;\;:=\;\;}l}
    \alpha(g,a) & g,
    \\
    \hnt
    \omega(g,a) &g[a]_Q,
    \\
    \hnt
    (g,a)\inv \!\!&(g[a]_Q,a\inv),
    \\
    \hnt
    \ell(g,a) & a.
  \end{array}
\]
\end{itemize}               
The edge $(g,a)$ should be thought of as
$\underset{g}{\bullet} \overset{a}{\rbl}
\!\underset{ga}{\bullet}$, its
inverse as
$\underset{g}{\bullet}\overset{a\inv}{\lbl}\!\underset{ga}{\bullet}$,
where $ga$ stands for $g[a]_Q$.  We note that $Q$ acts on $\mQ$ by
left multiplication as a group of automorphisms via
\[g\longmapsto {}^hg:=hg \quad \mbox{ and } \;\;
(g,a) \longmapsto {}^h(g,a) := (hg,a)
\]
for all $g,h\in Q$ and $(g,a)\in Q\times\til{A}$, where $h$ is an element of the acting group $Q$, $g$ a vertex of $\mQ$ and $(g,a)$ an edge of $\mQ$.

\subsection{$F$-inverse covers}
{
For a given finite $A$-generated inverse monoid $M$ we intend to construct a finite $A$-generated group $H$ for which the $A$-generated subdirect product
\begin{equation}\label{eq:subdirect product}
S:=\{([w]_H,[w]_M)\colon w\in \til A^*\}\subseteq H\times M
\end{equation}
is an $F$-inverse cover of $M$.
We start with a finite $A$-generated group $Q$ such that, for all $w\in \til A^*$,
$[w]_Q=1_Q$ implies that $[w]_M$ is an idempotent of $M$.
Such a group $Q$ can be found by representing $M$ as an inverse monoid
of partial bijections on a finite set $X$ and extending the partial
mappings $[a]_M$ ($a\in A$) to total permutations $\hat a$ on $X$ or on
some finite superset $Y\supseteq X$ and taking $Q:=\langle \hat a\colon a\in A\rangle$, the group generated by the permutations $\hat a$ ($a\in A$). In semigroup theoretic terms this
just means that the $A$-generated subdirect product
\[\{([w]_Q,[w]_M)\colon w\in \til A^*\} \subseteq Q\times M\] is an \emph{$E$-unitary} cover of $M$.}

{The following lemma will be crucial.
It is well known {and} readers 
familiar with the Margolis--Meakin-expansion $M(Q)$ of a  group $Q$~\cite{MM} will recognise that this lemma essentially proves the 
universal property of $M(Q)$. We present a proof in order to keep the paper
more self-contained; it is a modified version of the proof of Lemma~4.6 in~\cite{AKS}. We fix some notation: for an $A$-generated group $Q$
with Cayley graph $\mQ$, $q\in Q$ and a word $w\in \til A^*$ let
$\pi_q^\mQ(w)$ [resp.~$\pi_1^\mQ(w)$] denote the path in $\mQ$
labelled $w$ that starts at $q$ [resp.~$1_Q$].
\begin{Lemma}\label{lem:MMexpansion}Let $M$ be an $A$-generated
inverse monoid and $Q$ be an $A$-generated group
such that, for all $w\in \til A^*$,
$[w]_Q=1_Q$ implies that $[w]_M$ is an idempotent of $M$.
Then, for any words $u,v\in \til A^*$ for which $[u]_Q=[v]_Q$ and $\langle \pi_1^\mQ(u)\rangle\subseteq \langle \pi_1^\mQ(v)\rangle$ the inequality $[u]_M\ge[v]_M$ holds in $M$.
\end{Lemma}
\begin{proof} The proof is by induction on the length $|u|$ of the
word $u$. If $|u|=0$, that is, if $u=1$ is the empty word, then
$[u]_Q=[v]_Q$ implies that $[v]_Q=[u]_Q=1_Q$,
whence $[v]_M$ is an idempotent so that $[v]_M\le 1_M=[u]_M$.
Let $|u|=1$, that is, $u=a$ is a letter in $\til A$. 
The assumptions $\langle \pi_1^\mQ(a)\rangle \subseteq \langle\pi_1^\mQ(v)\rangle$ and $[a]_Q=[v]_Q$ imply that either (i) $v=v_1av_2$ with $[v_1]_Q=[v_2]_Q=1_Q$, or (ii) $v=v_1a\inv v_2$ and $[v_1]_Q=[a]_Q=[v_2]_Q$. In case (i), $[v_1]_M, [v_2]_M\le 1_M$, whence $[v]_M=[v_1av_2]_M=[v_1]_M[a]_M[v_2]_M\le[a]_M$. In case (ii), $[v_1a\inv]_Q=1_Q=[a\inv v_2]_Q$, whence $[v_1a\inv]_M,[a\inv v_2]_M\le 1_M$ so that
\[[v]_M=[v_1a\inv aa\inv v_2]_M=[v_1a\inv]_M[a]_M[a\inv v_2]_M\le [a]_M.\]
So let $|u|>1$ and let $[v]_Q=[u]_Q$ and $\langle
\pi_1^\mQ(u)\rangle\subseteq \langle \pi_1^\mQ(v)\rangle$, and assume
that the statement of the lemma is true for all words $u'$ with
$|u'|<|u|$ and arbitrary $v$. Choose some factorisation $u=u_1u_2$
with $|u_1|, |u_2|<|u|$. Let $q:=[u_1]_Q$; the assumption $\langle
\pi_1^\mQ(u)\rangle\subseteq \langle \pi_1^\mQ(v)\rangle$ implies that
$q$ is a vertex of $\langle \pi_1^\mQ(v)\rangle$, i.e.\ the path
$\pi_1^\mQ(v)$ meets the vertex $q$. Let $v=v_1v_2$ be a corresponding
factorisation. That is, the terminal vertex of $\pi_1^\mQ(v_1)$ is
$q$. Then $[vv\inv v_1]_Q=[v_1]_Q=[u_1]_Q=q$ and clearly $\langle
\pi_1^\mQ(u_1)\rangle \subseteq \langle \pi_1^\mQ(vv\inv v_1)\rangle$,
whence $[vv\inv v_1]_M\le [u_1]_M$
by the inductive hypothesis. Similarly,
$[v_2v\inv v]_Q=[v_2]_Q=[u_2]_Q$ and \[\langle\pi_q^\mQ(u_2)\rangle\subseteq \langle \pi_1^\mQ(u)\rangle\subseteq\langle \pi_1^\mQ(v)\rangle = \langle\pi_q^\mQ(v_2v\inv v)\rangle,\]
whence also $\langle\pi_1^\mQ(u_2)\rangle\subseteq\langle \pi_1^\mQ(v_2v\inv v)\rangle$ (here we apply the automorphism $x\mapsto{}^{q\inv}\!\!x$ of $\mQ$). The inductive assumption implies $[v_2v\inv v]_M\le [u_2]_M$. Altogether, 
\[[v]_M=[vv\inv v_1]_M[v_2v\inv v]_M\le [u_1]_M[u_2]_M=[u]_M. \]
\end{proof}	}

For a given
finite $A$-generated inverse monoid $M$ and a finite
$A$-generated group $Q$ as above,
we now seek to
provide a finite expansion $H$ of $Q$,
for which the subdirect product \eqref{eq:subdirect product}
is a finite $F$-inverse cover of $M$.
First we isolate an important property of groups generated by an
alphabet.

\begin{Def}[$X$-generated group with content function]
\label{def:groupwithcontentfctn}
Let $X$ be any alphabet; an $X$-generated group $R$ \emph{has a content function $\mathrm{C}$} if for every element $g\in R$ there is a unique $\subseteq$-minimal subset $\mathrm{C}(g)$ of $X$ such that $g$ is represented as a product of elements of $\mathrm{C}(g)$ and their inverses.
\end{Def}

We need to define one further property, which will be crucial towards
the construction of the desired group $H$. 

\begin{Def}[group reflecting the structure of a Cayley graph]\label{def:conditioncayley}
Let $Q$ be an $A$-generated group with Cayley graph $\mQ$, 
where $E:=Q\times A$ is the set of positive edges of $\mQ$.
An $E$-generated group $G$ \emph{reflects the structure of $\mQ$} if the following hold.
\begin{enumerate}
\item The action of $Q$ on $E$ by left multiplication extends to an action of $Q$ on $G$ by automorphisms on the left (denoted $(g,\xi)\mapsto {}^g\xi$ for $g\in Q$ and $\xi\in G$).
\item $G$ has a content function $\mathrm{C}$ 
such that, for any 
word $p\in \til{E}^*$ which forms a path $g\longrightarrow h$ in $\mQ$, the following hold:
\begin{enumerate}
\item if $\mathrm{C}([p]_G)=\varnothing$, that is if $[p]_G=1$, then $g=h$,
\item if $\mathrm{C}([p]_G)\ne\varnothing$, that is if  $[p]_G\ne 1$,
then there exists a word $q\in\til{E}^*$ which also forms a path
$g\longrightarrow h$ in $\mQ$
and such that $[p]_G=[q]_G$ and $q$ uses only edges of the content $\mathrm{C}([p]_G)$ of $[p]_G$ (and their inverses). In particular, the content $\mathrm{C}([p]_G)$ spans a connected subgraph of $\mQ$   containing $g$ and $h$.
\end{enumerate}
\end{enumerate}
\end{Def}

Next let  $Q$ be an $A$-generated group and, for $E=Q\times A$, let
$G$ be {a finite} $E$-generated group reflecting the structure of the
Cayley graph $\mQ$ of $Q$. {The existence of such a group $G$ is guaranteed by Lemma~\ref{thm:main theorem}, whose proof will be 
completed in Section~\ref{sec:groupoids}.}
Since $Q$ acts on $G$ by automorphisms on
the left, we can form the semidirect product
$G\rtimes Q$, which consists of the set $G\times Q$ endowed with the
binary operation 
\[(\gamma,g)(\eta,h):=(\gamma\cdot{}^{g}\eta,gh),\]  inversion
\[(\gamma,g)^{-1}:=({}^{g^{-1}}\gamma^{-1},g^{-1})\] and identity element $(1_G,1_Q)$.
Consider the following $A$-generated subgroup $H$ of $ G\rtimes Q$: 
\begin{equation}\label{eq:expansion of Q}
H:=\langle ([(1_Q,a)]_G,[a]_Q)\colon a\in A\rangle\subseteq G\rtimes Q.
\end{equation} 
Readers familiar with the Margolis--Meakin-expansion $M(Q)$ \cite{MM}  will notice that the group $H$, in a sense, approximates $M(Q)$. The type of construction used for the group $H$ occurs frequently in (semi)group theory, see e.g. Elston~\cite{elston} or Almeida~\cite[Section 10]{almeida:book}; that it can be useful for the construction of $F$-inverse covers is discussed in \cite{ASz}. 
For a word $p\in \til{A}^*$, the value of $p$ in $H$ is
\begin{equation}\label{eq:value in H}
[p]_H=([\pi_1^\mQ(p)]_G,[p]_Q)
\end{equation}
 where, again,  $\pi_1^\mQ(p)$ is the unique path in $\mQ$ starting at $1_Q$ and being labelled $p$, interpreted as a word over  $\til{E}$. 
 This is easily seen by induction on the length $|p|$ of $p$. 
 In particular, $H$ is an expansion of $Q$ with canonical morphism $([\pi_1^\mQ(p)]_G,[p]_Q)\mapsto [p]_Q$.
We can now formulate the main theorem of this section, and derive it essentially based on Lemma~\ref{thm:main theorem}.
{\begin{Thm}\label{thm:groupoid-premorphism}
Let $M$ be a finite $A$-generated inverse monoid,
$Q$ a finite $A$-generated group such that, for all $w\in \til A^*$,
$[w]_Q=1_Q$ implies that $[w]_M$ is an idempotent of $M$.
For $E=Q\times A$, let $G$ be a 
finite $E$-generated
 group (with content function $\mathrm{C}$) which reflects the structure of the Cayley graph $\mQ$ of $Q$ (Definition~\ref{def:conditioncayley}), and let $H$ be the group defined by \eqref{eq:expansion of Q}. 
Then the subdirect product
\[S:=\{([w]_H,[w]_M)\colon w\in \til A^*\}\]
is a finite $F$-inverse cover of $M$.
\end{Thm}}
{\begin{proof}
It is clear that $S$ is finite. The natural order
$\le$ on $S$ is given by $(g,m)\le (h,n)$ if and only if $g=h$ and
$m\le n$. The canonical morphism $S\twoheadrightarrow M$,
$(g,m)\mapsto m$ is idempotent separating. We need to show that
every $\sigma$-class $(g,m)\sigma$ has a greatest element. Note that
$(g,m)\sigma=\{(h,n)\in S\colon h=g\}$. Let $w\in \til A^*$ be a word
such that $[w]_H=g$. Then $[w]_H=([\pi_1^\mQ(w)]_G,[w]_Q)$. Since $G$
reflects the Cayley graph $\mQ$ of $Q$ there exists a word $\pi\in \til E^*$ 
such that $[\pi]_G=[\pi_1^\mQ(w)]_G$, $\pi$ contains only edges from the
content $\mathrm{C}([\pi_1^\mQ(w)]_G)$ (and their inverses)  and $\pi$
forms a path $1_Q\longrightarrow [w]_Q$ in $\mQ$. The path $\pi$ is
induced by some word $u\in \til A^*$, that is, $\pi=\pi_1^\mQ(u)$. As
the terminal vertex of $\pi_1^\mQ(u)$ is $[u]_Q$ we have
$[w]_Q=[u]_Q$, together with $[\pi_1^\mQ(w)]_G=[\pi_1^\mQ(u)]_G$
therefore also $g=[w]_H=[u]_H$. 
While $\pi$ and $u$ are not necessarily uniquely determined by $g$, we note that $\langle \pi_1^\mQ(u)\rangle$ is the graph spanned by $\mathrm{C}([\pi_1^\mQ(w)]_G)$ and therefore  $\langle \pi_1^\mQ(u)\rangle$ is uniquely determined by $g=([\pi_1^\mQ(w)]_G,[w]_Q)$. In addition, for \textsl{every} $v\in \til A^*$ for which $[v]_H=g$ we have \[\langle\pi_1^\mQ(u)\rangle=\langle\mathrm{C}([\pi_1^\mQ(w)]_G)\rangle=\langle\mathrm{C}([\pi_1^\mQ(v)]_G)\rangle,\] which implies $\langle\pi_1^\mQ(u)\rangle\subseteq \langle\pi_1^\mQ(v)\rangle$ for every such $v$. It follows from Lemma~\ref{lem:MMexpansion} that $[v]_M\le[u]_M$ and therefore also
$(g,[v]_M)\le(g,[u]_M)$ for every such $v$. This just says that $(g,[u]_M)$
 is indeed the greatest element of $(g,m)\sigma$.
\end{proof}}
The $F$-inverse cover $S$ of $M$ in
  Theorem~\ref{thm:groupoid-premorphism} is $A$-generated \textsl{as
    an inverse monoid}. In a sense, this is a strong form of
  $F$-inverse cover, as the original definition does not require the
  cover to be $A$-generated as an inverse monoid. For any finite
  $A$-generated inverse monoid $M$ and any finite $A$-generated group
  $K$ expanding an
  $A$-generated group $H$ as in Theorem~\ref{thm:groupoid-premorphism},  there exists  an $F$-inverse
  cover $U\subseteq K\times M$ of $M$ with maximum group quotient $K$,
  albeit not necessarily one that is  $A$-generated as an inverse monoid.

This can be seen as follows. Let $M$ and $H$ be as in
  Theorem~\ref{thm:groupoid-premorphism} and let the $A$-generated
  group $K$ be an expansion of $H$. Then the subdirect product
\begin{equation}\label{eq: F-cover}
	U:=\{([u]_K,[v]_M)\in K\times M\colon u,v\in \til A^*, [u]_H=[v]_H\}
\end{equation}
is an $F$-inverse cover of $M$ with maximum group quotient
$K$. Indeed, let $([u]_K,[v]_M)\in U$, then $[u]_H=[v]_H$. By the
proof of Theorem~\ref{thm:groupoid-premorphism} there exists $w\in
\til A^*$ with $[w]_H=[u]_H$ and such that  $[w]_M\ge [x]_M$ for
\textsl{every} $x\in \til A^*$ for which $[x]_H=[w]_H=[u]_H$.
Since $[w]_H=[u]_H$ we have $([u]_K,[w]_M)\in U$; in addition, $([u]_K,[w]_M)\ge([u]_K,[x]_M)$ for all $x\in \til A^*$ for which $[x]_H=[u]_H$. It follows that $([u]_K,[w]_M)$ is the greatest element of the $\sigma$-class $([u]_K,[v]_M)\sigma$. It should be  emphasised that the $F$-inverse cover $U$ of $M$ cannot be guaranteed to be $A$-generated as an inverse monoid, and the $A$-generated subdirect product of $M$ and $K$ contained in $U$ cannot be
guaranteed to be an $F$-inverse monoid.

\subsection{Excursion: symmetries} 
An alternative perspective on $A$-generated inverse monoids, maybe
broader in a model-theoretic sense, would view them as two-sorted
structures in which the $A$-labelling of the generators is integrated
in an explicit fashion. This leads to the two-sorted structure
\[
\mathfrak{M}:= 
(M,A;\iota, \cdot,\inv,1),
\]
where the set $A$ of labels forms a second sort along with first sort $M$, and 
$\iota\colon A \rightarrow M$ encodes the explicit $A$-labelling
as a function, so that $\iota(A)$ becomes a subset of $M$.
The $A$-generated inverse monoid as a structure $(M;\cdot,\inv,1,A)$ --- with the algebraic signature $\{\cdot,\inv,1\}$ enriched by the set $A$ of constant symbols  interpreted as generators of the inverse monoid $(M;\cdot,\inv,1)$ ---
is clearly rigid since fixing the generators fixes everything.
The two-sorted structure $\mathfrak{M}$, on the other hand, allows us 
to analyse internal symmetries induced by permutations of $A$: 
these precisely are the automorphisms of the two-sorted structure $\mathfrak{M}$,
namely pairs of compatible permutations of the sets $M$ and $A$ that
commute with $\iota \colon A \rightarrow M$ and the algebraic
operations on $M$. As $\iota(A) \subseteq M$ is a generating set
in $(M; \cdot,\inv,1)$, 
compatibility with the algebraic operations
now implies that any such automorphism is uniquely determined by its action on $A$.

{As our construction of the $F$-inverse cover $S$ of $M$ according to
Theorem~\ref{thm:groupoid-premorphism} deals with $A$-generated
inverse monoids, the question arises to which extent this construction
may also be symmetry-preserving -- in the sense of commuting with
permutations of the generator set $A$ that induce automorphisms of $\mathfrak{M}$.
And indeed, our construction can be made fully symmetry-preserving overall. 
Theorem~\ref{thm:groupoid-premorphism}  rests on the chain
\[
  M \;\leadsto\;  Q \;\leadsto\;  H  \;\leadsto\;  S:= H \times_{\!\!\ssc A} M,
\]
where $H \times_{\!\!\ssc A} M$ denotes the $A$-generated subdirect
product of the $A$-generated group $H$ and the $A$-generated inverse
monoid $M$. The step from the $A$-generated group $Q$ to the
$A$-generated group $H$, which is the technically challenging 
construction behind Lemma~\ref{thm:main theorem}, is really based on the
Cayley graph $\mQ$ of $Q$; from this Cayley graph $\mQ$ we first obtain 
a $(Q\times A)$-generated group $G$, which reflects the structure of $\mQ$ in the sense of
Definition~\ref{def:conditioncayley}, and finally $H$ as the
group defined in~\eqref{eq:expansion of Q}.
The Cayley graph $\mQ$, as a graph that is edge-labelled by~$A$,  
lifts every symmetry of the $A$-generated group $Q$ in a
canonical manner. The symmetry-preserving
nature of the core construction in the passage from this Cayley graph 
$\mQ$ to $G$
is captured in Proposition~\ref{prop:auto extends} and to be further discussed in
Section~\ref{sec:groupoids}.
That $H$ as defined in~\eqref{eq:expansion of Q} then carries all shared symmetries of $G$ and
$\mQ$ is straightforward, and similarly for $S = H \times_{\!\!\ssc A} M$
in relation to $H$ and $M$. In fact, all of these
steps between $Q$ and $S$, including the core passage from $\mQ$
to $G$, can be seen to be symmetry-preserving (in terms of two-sorted
presentations with explicit domains for the generator or label sets), simply
because they can all be cast as explicit definitions (in more
model-theoretic terms: as interpretations, albeit of a higher-order nature) of the target structures
over the input structures, which cannot possibly violate isomorphism invariance.}

Therefore the one step that curiously risks breaking symmetries, lies in
the passage from $M$ to $Q$.
The straightforward recipe indicated above is to obtain a representation
of $M$ as an inverse monoid of partial bijections over some set $X$,
which can then be extended to permutations of $X$ to yield a
permutation group $Q$ that relates to $M$ as required for Lemma~\ref{lem:MMexpansion}.
But this method crucially involves free choices of a representation over
a suitable set $X$ and of extensions from partial to global bijections over that
$X$; and these choices can break symmetries.
\footnote{E.g.\ two partial bijections that are identical over $X$ may be extended to permutations of different orders.} 
This seeming obstacle can be overcome though, if we use for instance the
canonical representation of the $A$-generated inverse monoid $M$
by partial bijections over the set $M$ itself, and 
then consider all possible extensions of the partial
mappings $[a]_M \colon M \rightarrow M$ in parallel:
for every possible extension of the family $([a]_M)_{a\in A}$
by total permutations $[ \hat a]_M \supseteq [a]_M$  of $M$, we may
use, for the set~$X$, the disjoint union of copies of $M$,
one for each choice of extensions, and the bijection induced by 
the instances of $[\hat a]_M$ in each copy of $M$ as the generator
set for $Q$ in the symmetric group over~$X$.
This is but one of several variants for the step from $M$ to $Q$ that 
are explicitly definable over $\mathfrak{M}$ and therefore symmetry-preserving.

\subsection{Excursion: pointlike conjecture for inverse monoids
  versus $F$-inverse cover problem}  What can we say about the gap
between these two problems?
Recall that for an $A$-generated inverse monoid $M$, an $A$-generated group $H$ is a witness for the pointlike pairs of $M$ if
\begin{equation}\label{eq:pointlike pairs1}
\forall \ u_1,u_2\in \til{A}^*\ \exists\ v\in
\til{A}^*\colon[u_1]_H\ne [u_2]_H \;\mbox{ or } \; 
[u_1]_M, [u_2]_M\le [v]_M.
\end{equation}
 On the other hand, the $A$-generated subdirect product \[\{([w]_H,[w]_M)\colon w\in \til{A}^* \}\subseteq H\times M\] is an $F$-inverse cover of $M$ provided that
\begin{equation}\label{eq:Fcover}
  \forall \ u_1,u_2\in \til{A}^*\ \exists\ v\in \til{A}^*\colon
  [u_1]_H\ne [u_2]_H \mbox{ or }
  \left\{
    \begin{array}{l@{}}
      [u_1]_M, [u_2]_M\le [v]_M
      \\
      \nt\!\!\mbox{and}
      \\
       {[u_1]_H=[v]_H=[u_2]_H.}
    \end{array}
  \right.
\end{equation}

As shown in~\cite{AS}, the expansion $H=Q^{\mathbf{Ab}_p}$ of an $A$-generated group $Q$
witnesses
the pointlike sets of the inverse monoid $M$ ($Q$ in relation to $M$ as in Lemma~\ref{lem:MMexpansion}) and
therefore the pointlike conjecture for inverse
monoids is verified; in particular $H=Q^{\mathbf{Ab}_p}$ satisfies condition \eqref{eq:pointlike pairs1}. For any prime $p$,  the so-called \emph{universal $p$-expansion} $Q^{\mathbf{Ab}_p}$ of $Q$
is the largest $A$-generated expansion $R\twoheadrightarrow Q$ whose kernel is
an elementary Abelian $p$-group.
This expansion can be  obtained by the construction in~\eqref{eq:expansion of Q}, 
except that the $E$-group
$G$ used there is replaced with the free $E$-generated Abelian group
of exponent $p$ (which is the $|E|$-fold direct product of cyclic
groups of order $p$), in fact a very transparent group. Sufficient for
the verification of the pointlike conjecture is an $E$-group which
reflects the structure of the Cayley graph $\mQ$ of $Q$ in a very weak
sense: the graph spanned by the content  of a word over $\til{E}$
which forms a path $u\longrightarrow v$ requires only a
\textsl{connected component} containing $u$ and $v$.  The enormous
effort we require in the remainder of the paper to construct an expansion
$H$ of $Q$ that satisfies the seemingly innocent, additional condition
$[u_1]_H=[v]_H=[u_2]_H$ in \eqref{eq:Fcover}, indicates that 
the gap between the pointlike problem for inverse monoids and the
$F$-inverse cover problem may indeed be huge.

As already mentioned, Henckell and Rhodes considered Problem~\ref{problem:Finverse} as a ``stronger form'' of the pointlike conjecture for inverse monoids. On the other hand, in the last sentence of their paper they wrote: ``We do not necessarily believe [the $F$-inverse cover problem] has an affirmative answer.'' So, in contrast to what is often reported, Henckell and Rhodes did not really conjecture that every finite inverse monoid does admit a finite $F$-inverse cover, but rather seem to have been undecided about this question. In fact, they seem to have had some feeling that the $F$-inverse cover problem might be hard.

\subsection{{The main result}} 
In order to prove Theorem~\ref{thm:groupoid-premorphism}  {it is sufficient to construct}, for any finite $A$-generated group $Q$ and $E=Q\times A$ a finite $E$-generated group $G$ which reflects the structure of the Cayley graph $\mQ$ of $Q$ according to Definition~\ref{def:conditioncayley}. 
The existence of such a  group $G$  is guaranteed by the following
more general lemma,
which is the main result of the paper. For item (1) recall that every automorphism of an oriented graph induces a permutation of its set of positive edges.

\begin{Lemma}[main lemma]\label{thm:main theorem}
For every finite connected oriented graph $\mE=(V,\til{E};\alpha,\omega,{}^{-1})$ there exists a finite $E$-generated group $G$ which has the following properties:
\begin{enumerate}
\item Every permutation of $E$ induced by an  automorphism of $\mE$ extends to an automorphism of $G$.
\item The set of relations $p=1$ satisfied by $G$ (with $p\in \til{E}^*$) is closed under the deletion of generators and thus $G$ has a content function $\mathrm{C}$ (Proposition~\ref{prop:retractablehasconctent}).
\item For any word $p\in \til{E}^*$ which forms a path $u\longrightarrow v$ in  $\mE$ (with $u$ and $v$ not necessarily distinct vertices of $\mE$) the following hold:
\begin{enumerate}
\item if $\mathrm{C}([p]_G)=\varnothing$ then $u=v$,
\item if $\mathrm{C}([p]_G)\ne\varnothing$ then there exists a word $q\in \til{E}^*$ with $[p]_G=[q]_G$ such that $q$ also forms a path $u\longrightarrow v$ 
in $\mE$ and $q$ only uses edges from the content $\mathrm{C}([p]_G)$ (and their inverses). In particular, $\mathrm{C}([p]_G)$ spans a connected subgraph of $\mE$ containing  $u$ and $v$.
\end{enumerate} 
\end{enumerate}
\end{Lemma}

\begin{Rmk}\rm The free group generated by $E$ obviously enjoys properties (1)--(3) of Lemma~\ref{thm:main theorem}. Hence, the main result of the paper is another instance of when the behaviour of a free group can be ``simulated'' or ``approximated'' by a finite group~\cite{almeidadelgado1,almeidadelgado2,Ash,HerwigLascar, otto0}, {in contrast to~\cite{Coulbois} where such an approximation is not possible.} 
\end{Rmk}
The remainder of the paper is devoted to proving
Lemma~\ref{thm:main theorem}.
This requires quite a bit of work. It will be accomplished in
Section~\ref{sec:groupoids}.   In order to achieve this goal we
introduce several graph-theoretic constructions which will be
presented in Sections~\ref{sec:machinery} and~\ref{sec:2results}. 
{The results in those 
three sections are of a more general nature, 
may be of independent interest and 
will be of particular use  in the follow-up paper~\cite{ABO}.}

\section{Tools}\label{sec:machinery}
In this section we introduce some graph-theoretic
constructions, which later will enable the construction of a group $G$
as mentioned above. The group itself will be
realised as a permutation group defined by its \textsl{action graph}. It is
a well-established approach to construct finite $A$-generated  groups
which avoid certain unwanted relations, to proceed as described in the
following.
First  encode the relations in a finite $A$-labelled directed graph
${\mathcal X}$  --- the set of unwanted relations will be infinite in
most cases, but must in some sense be regular (recognisable by a
finite automaton). If necessary take a  quotient ${\mathcal X}/{\equiv}$
of ${\mathcal X}$ which guarantees that the edge labels from $A$
induce partial permutations on the vertex set. Finally form some
\textsl{completion} $\ol{{\mathcal X}/{\equiv}}$ of ${\mathcal X}/{\equiv}$,
through extending the 
partial permutations to total permutations of the vertex set of
${\mathcal X}/{\equiv}$ or of some finite superset.
The letters $a\in A$ then act as permutations on the finite set of
vertices of $\ol{{\mathcal X}/{\equiv}}$ and one gets a finite
permutation group that avoids the unwanted relations. 

The simplest example of this procedure is the construction of a finite $A$-generated
group which avoids a single relation $p=1$ for a given reduced word $p\in \til A^*$
--- this provides a transparent and elegant proof that every free group is residually
finite. A slightly more general
application is the Biggs construction~\cite{Biggs} providing a finite group that avoids
all relations $p=1$ for  all reduced words $p$ of length
up to a given bound~$n$ --- this has been used for the construction of
finite regular graphs of large girth.
A meanwhile classical and more advanced application of this approach is Stallings' proof of Hall's
Theorem that every finitely generated subgroup of a free group $F$ is
closed in the profinite topology of $F$~\cite{stallings}.
Here a finite $A$-generated group is constructed that avoids all the
(infinitely many) relations of the form $h=p$ where $h$ runs through all elements
of a finitely generated subgroup $H$ of the free $A$-generated group $F$
and $p$ is a fixed element of $F\setminus H$.
Many more examples can be found in~\cite{KapMas, auinbors} and elsewhere. In his
paper~\cite{Ash} Ash definitely developed some mastership of arguments of this kind.
Independently, the third author has suggested
a considerable refinement of this approach~\cite{otto0}. He proposed a
construction which is inductive on the subsets of the generating set
$A$ in the sense that the $k$th group $G_k$ satisfies/avoids all
relations $p=1$ in at most $k$ letters that should be
satisfied/avoided by the final group $G$. In the step $G_k\leadsto
G_{k+1}$ not only new relations $p=1$ in more than $k$ letters are
added which are to be avoided (by adding components to the  graph
which defines $G_k$) but, at the same time, the relations in at most
$k$ letters must be preserved.
The motivation for this approach has come from some relevant
applications to  hypergraph coverings and
finite model theory~\cite{otto0}. The constructions in this section and the results of the next section are of this flavour and are taken from the third author's~\cite{otto3}.
\subsection{$E$-graphs and $E$-groups} We slightly change
perspective: since the edges of the graph $\mE$ of
Lemma~\ref{thm:main theorem}
are the letters of the
labelling alphabet  we now denote the labelling alphabet by $E$.
An $E$-labelled graph is an
\emph{$E$-graph} if every vertex $u$
has, for every label $a\in \til{E}$, at most one edge
with initial vertex $u$ and label $a$. In the literature, such graphs occur under a variety of different names, such as \emph{folded graph}~\cite{KapMas} or \emph{inverse automaton}~\cite{AS,auinbors}, to mention just two. In an $E$-graph $\mathcal{K}$,
for every word $p\in \til{E}^*$ and every vertex $u$ there is at most
one path $\pi=\pi_u^\mathcal{K}(p)$ with initial vertex $\alpha\pi=u$ and
label $\ell(\pi)=p$. For a path $\pi$ in $\mathcal{K}$ with initial
vertex $u$, terminal vertex  $v$ and label $p\in \til{A}^*$ (for
$A\subseteq E$) we write $u\overset{p}{\longrightarrow}v$ 
and call $\pi$ an $A$-\emph{path}
$u\longrightarrow v$; the vertices $u$ and $v$ are
\emph{$A$-connected} in $\mK$. The \emph{$A$-component} of a vertex
$v$ of the $E$-graph $\mathcal{K}$, denoted $v\mathcal{K}[A]$, is the
subgraph of $\mathcal{K}$ spanned by  all paths in $\mathcal{K}$
having initial vertex $v$ and whose labels are in $\til{A}^*$. A
labelled graph $\mathcal{K}$ is called \emph{complete} or  a
\emph{group action graph} (also called  \emph{permutation
  automaton}) if every vertex $u$
has, for every label $a\in \til{E}$ exactly one edge $f$ with initial vertex $\alpha f=u$ and label $\ell(f)=a$; in this case, for every word $p\in \til{E}^*$ and every vertex $u$ there exists exactly one path $\pi=\pi_u^\mathcal{K}(p)$ starting at $u$ and having label $p$. We set $u\cdot p:=\omega(\pi_u^\mathcal{K}(p))$, the terminal vertex of the path starting at $u$ and being labelled $p$; then,  for every $p\in \til{E}^*$, the mapping 
$[p] \colon V\to V$, $u\mapsto u\cdot p$ is a permutation of the
vertex set $V$ of $\mathcal{K}$. 
Thus the involutory monoid $\til{E}^*$ acts on $V$ by permutations on the right.
The permutation group
\begin{equation}\label{eq:transition group}
\mathrsfs{T}(\mathcal{K}):=\{[p]\colon p\in \til{E}^*\}
\end{equation}
 obtained this way, is called the \emph{transition group} $\mathrsfs{T}(\mathcal{K})$ of the graph $\mathcal{K}$.
This transition group $\mathrsfs{T}(\mathcal{K})$ is an $E$-generated
group (\emph{$E$-group} for short) in a natural way, the letter $e\in \til{E}$ induces the
permutation $[e]$ which maps every vertex $u$ to the terminal vertex
$\omega\pi_u(e)$ of the edge $\pi_u(e)$ which is the unique edge with
initial vertex $u$ and label $e$. Note that this edge may be a loop
edge for every vertex $u$ (so $[e]$ might be the identity element of
$\mathrsfs{T}(\mathcal{K})$). Moreover, it may happen that distinct
letters $e\ne f\in \til{E}$ induce the same permutation.

A crucial fact concerning the transition group
$G=\mathrsfs{T}(\mathcal{K})$ is the following:  for every connected
component $\mathcal{C}$ of $\mathcal{K}$ and every vertex $u$ of
$\mathcal{C}$ there is a unique surjective graph morphism
$\varphi_u\colon \mathcal{G}\twoheadrightarrow \mathcal{C}$ 
from the Cayley graph $\mathcal{G}$ of $G$ onto $\mathcal{C}$
for which $\varphi_u(1)=u$;
we call $\varphi_u$ the \emph{canonical morphism $\mathcal{G}\twoheadrightarrow\mathcal{C}$  with respect to} $u$; occasionally we shall leave the vertex $u$ undetermined and shall speak of \textsl{some} canonical morphism $\mathcal{G}\twoheadrightarrow\mathcal{C}$. {The existence of these canonical morphisms will be frequently assumed without further mention.} For easy reference we give a name to this phenomenon.

\begin{Def}\label{def:covering relation}\rm The Cayley graph $\mG$ of an $E$-group $G$ \emph{covers} a complete, connected $E$-graph $\mC$ if there is a canonical morphism $\varphi\colon \mG\twoheadrightarrow \mC$.
\end{Def}

An $E$-graph $(V,K;\alpha,\omega,{}^{-1},\ell, E)$ is \emph{weakly complete} if, 
for every letter $a\in
\til{E}$, the partial permutation on $V$ induced by $a$ is a
permutation on its domain; in other words, provided that the graph is
finite, the subgraph spanned by all edges with label $a$ is a disjoint
union of cycle 
graphs ($a$-cycles).
For every weakly complete graph $\mathcal{K}$ we denote by
$\overline{\mathcal{K}}$ its \emph{trivial completion}, that is, the
complete graph obtained by adding, for every $a\in \til{E}$, a loop
edge with  label $a$ to every vertex not already contained in an
$a$-cycle of $\mathcal{K}$.

\subsection{$k$-retractable groups, content function and $k$-stable expansions} \label{sec:retractable,stable,content} For
$a\in E$ and $p\in \til{E}^*$ let $p_{a\to 1}$ be the word obtained
from $p$ by deletion of all occurrences of $a$ and $a^{-1}$ in
$p$. Let $G$ be an $E$-group; for every $A\subseteq E$ let $G[A]$ be
the $A$-generated subgroup of $G$.

\begin{Def}\label{def:retractable group}\rm
An $E$-group $G$ is
\emph{retractable} if, for all words $p,q\in \til{E}^*$ and every
letter $a\in E$ the following holds:
\footnote{It suffices to restrict this postulate to the case $q=1$.}
\[ [p]_G=[q]_G\Longrightarrow [p_{a\to 1}]_G=[q_{a\to 1}]_G.\]
Moreover,  $G$ is \emph{$A$-retractable} if $G[A]$ is retractable (as an $A$-group), and, for $k\le|E|$, $G$ is \emph{$k$-retractable} if $G$ is $A$-retractable for every $A\subseteq E$ with $|A|=k$. 
\end{Def}
Of course, $k$-retractability implies $l$-retractability for all $l\le k$, and every group is $1$-retractable.
Retractability of an $E$-group $G$ means that for every subset $A\subseteq E$ the mapping
\[E\to E\cup\{1\},\ a\mapsto\begin{cases} a\mbox{ if }a\in A\\ 1\mbox{ if }a\notin A\end{cases}\] extends to an endomorphism $\psi_A$ of $G$, which in fact is a retract endomorphism onto $G[A]$ (the image of $\psi_A$ is $G[A]$ and its restriction to $G[A]$ is the identity mapping). For an $E$-group
$G$ and $A\subseteq E$ we denote the Cayley graph of $G[A]$,
considered as an $A$-graph, by $\mathcal{G}[A]$; this graph is weakly
complete as an $E$-graph and, as above,  we denote its trivial completion by
$\overline{\mathcal{G}[A]}$.
In light of the connection
with retract endomorphisms
we see the following.

\begin{Prop}\label{prop:retractable groups} An $E$-group $G$ is retractable if and only if its Cayley graph $\mG$ covers $\overline{\mG[A]}$ for every $A\subseteq E$.
\end{Prop}

\begin{proof} Suppose that $G$ is retractable and $A\subseteq E$. The
  retract endomorphism $\psi_A$ is a canonical morphism  $\psi_A\colon
  G\twoheadrightarrow G[A]$ if $G[A]$ is considered as an $E$-group
  with all $e\in E\setminus A$ being identity generators. Its Cayley
  graph with respect to $E$ coincides with $\overline{\mG[A]}$. It
  follows that there is a canonical graph morphism $\mG
  \twoheadrightarrow \ol{\mG[A]}$, that is, $\mG$ covers $\overline{\mG[A]}$. 

Suppose conversely that for every $A\subseteq E$ there is a canonical
graph morphism $\mG \twoheadrightarrow \ol{\mG[A]}$. We note that this
morphism must be injective when restricted to $\mG[A]$ (considered as
a subgraph of $\mG$).  Let $p\in \til{E^*}$,  $a\in E$ and suppose
that $[p]_G=1$. Then $p$ labels a closed path $\pi_1^{\mG}(p)$
at $1$ in $\mG$. Let $B=E\setminus\{a\}$. The canonical morphism
$\mG\twoheadrightarrow \ol{\mG[B]}$ maps the path $\pi_1^{\mG}(p)$ to
the path $\pi_1^{\ol{\mG[B]}}(p)$ which is also closed. The paths
$\pi_1^{\ol{\mG[B]}}(p)$ and $\pi_1^{\ol{\mG[B]}}(p_{a\to 1})$
traverse the same edges except loop edges labelled $a^{\pm 1}$, and
therefore visit the same vertices.
So $\pi_1^{\ol{\mG[B]}}(p_{a\to 1})$ is also closed, and as it runs entirely in $\mG[B]$, it follows that $[p_{a\to 1}]_{G[B]}=1$ and therefore $[p_{a\to 1}]_G=1$.
\end{proof}

For a word $p\in \til{E}^*$ the \emph{content} $\co(p)$ is the set of all letters $a\in E$ for which $a$ or $a^{-1}$ occurs in $p$.
The importance of retractable $E$-groups for our purpose comes from
the fact that such $E$-groups admit a \emph{content function}
 (Definition~\ref{def:groupwithcontentfctn}). 
Indeed, assume that $G$ is retractable. Then, for $p,q\in \til{E}^*$ and $a\in E$ the equality $[p]_G=[q]_G$ implies $[p_{a\to 1}]_G=[q_{a\to 1}]_G$. Suppose now that $a\in \co(p)$ but $a\notin \co(q)$. Then
 the words $q$ and $q_{a\to 1}$ are identical. Hence $[p]_G=[q]_G$ implies
 \[
   [p_{a\to 1}]_G=[q_{a\to 1}]_G=[q]_G=[p]_G.
 \]
In this way, we may delete (without changing its  value $[p]_G$) every letter in a word $p$ which does not occur in every other representation $q$ of the group element $[p]_G$. 
This leads to the following definition.

\begin{Def}\label {def:G-content}
  Let $G$ be a retractable $E$-group and $g\in G$.
  The \emph{content} $\mathrm{C}(g)$ of $g$ is
\[\mathrm{C}(g):=\bigcap\big\{\co(q)\colon q\in \til{E}^*, [q]_G=g\big\}.\]
For a word $p\in \til{E}^*$ the $G$\emph{-content} of $p$ is the content $\mathrm{C}([p]_G)$.
\end{Def}

The terminology is justified as $\mathrm{C} \colon g \mapsto \mathrm{C}(g)$ 
 clearly is a content function in the sense of
 Definition~\ref{def:groupwithcontentfctn}. 
 So we have shown the following.

\begin{Prop}\label{prop:retractablehasconctent} Every retractable group  has a content function.
\end{Prop}

In case $G$ is retractable, for any two subsets $A,B\subseteq E$ we have
\begin{equation}\label {eq:2acyclic}
G[A]\cap G[B]=G[A\cap B].
\end{equation}
Groups satisfying  this condition for all $A,B\subseteq E$ have been called
\emph{$2$-acyclic} by the third author in
\cite{otto0,otto3}: condition~(\ref{eq:2acyclic}) rules out 
patterns as on the left-hand side of Figure~\ref{fig:2/3-acyclic} where
$g$ would belong to $G[A]\cap G[B]$ but not to $G[A\cap B]$,
and the cosets $G[A]$ and $G[B]$ form
a non-trivial $2$-cycle.  In other words, 
condition~\eqref{eq:2acyclic} implies that the intersection of
two cosets $gG[A]$ and $hG[B]$ in $G$ is either empty or
is a coset of the form $kG[A\cap B]$. In terms of connectivity
in the Cayley graph $\mathcal{G}$ of $G$
this means that, if two vertices $u$ and $v$ are connected by an
$A$-path as well as by a $B$-path, then there is even an $(A\cap
B)$-path $u\longrightarrow v$;
this point of view will be frequently used in the paper.

But indeed, retractable groups also avoid patterns as on the
right-hand side of  Figure~\ref{fig:2/3-acyclic}. In the terminology
of~\cite{otto0,otto3}, they are even \emph{$3$-acyclic}. This means
that, for all $A,B,C \subseteq E$ and all 
$g,h,k \in G$ the following holds: 
\begin{equation}\label{eq:3acyclic}
  \begin{array}{@{}l@{}}
     gG[A]=hG[A],\ hG[B]=kG[B] \mbox{ and }\ kG[C]=gG[C]
    \\
\Longrightarrow \;\;\hnt
hG[A\cap B]\cap kG[B\cap C]\cap gG[C\cap A]\ne\varnothing,
    \end{array}
\end{equation}
as we shall see in passing, in connection with the proof 
of Lemma~\ref{lem:3-acyclic} below.

\begin{figure}[ht]
\begin{tikzpicture}[scale=0.7]
\draw plot [smooth cycle] coordinates {(-2.5,1.5) (0,2.5) (2,1.5)( 2.5,3) (0,3.5) (-2,3)};
\draw plot [smooth cycle] coordinates {(-2.3,2) (-2,2.5) (0,1) (2,2.5)( 2.3,1.3)
  (1,0)(0,-.3) (-1,0) (-2,1)};
\filldraw (-1.75,2) circle (2.5pt); 
\draw(-1.75,2) node[left]{$1$};
\filldraw(1.75,2) circle (2.5pt);
\draw(1.75,1.9) node[right]{$g$};
\draw(0,.5) node {$G[B]$};
\draw(0,3) node {$G[A]$};
\draw(0,-1) {};
\end{tikzpicture}
\qquad\qquad \begin{tikzpicture}[scale=0.6]
\filldraw(-3,-1.3)circle(2.5pt);
\draw(-3,-1)node[below left]{$g$};
\filldraw(3,-1)circle(2.5pt);
\draw(3,-1)node[right]{$k$};
\filldraw(1,3)circle(2.5pt);
\draw(1,3)node[below]{$h$};
\draw plot [smooth cycle] coordinates {(0.5,3) (1.5,3)(2.9,1.3) (4,-1.2) (2,-1)(0.7,1)};
\draw plot [smooth cycle] coordinates {(-4,-2)(0,-0.3)(1.5,3.5)(-1.5,2)};
\draw plot [smooth cycle] coordinates {(-4,-1) (0,-0.3)(2,-0.4)(3.5,-.7)(4,-1.5)(1,-1.8)(-3,-2)};
\draw(0,-1.8)node[below]{$k{G}[C]$};
\draw(3.4,2.3)node{$h{G}[B]$};
\draw(-2.1,2.6)node{$g{G}[A]$};
\end{tikzpicture}
\caption{Coset $2$- and $3$-cycles}
\label{fig:2/3-acyclic}
\end{figure}

\begin{Rmk}
\label{rem:2and3acyc}
Retractable $E$-groups are $2$- and $3$-acyclic in the sense of
satisfying conditions~\ref{eq:2acyclic} and~\ref{eq:3acyclic},
meaning that their Cayley graphs do not admit 
connectivity patterns of cosets as in Figure~\ref{fig:2/3-acyclic}.
\end{Rmk}

\begin{Def}\label{def:stable} For $A\subseteq E$, an expansion $H\twoheadrightarrow G$ of $E$-groups
is $A$-\emph{stable} if the canonical morphism is injective when
restricted to $H[A]$; it is \emph{$k$-stable} (for $k<|E|$)
if it is $A$-stable for every $k$-element subset $A$ of $E$.
\end{Def}  
We arrive at our first basic
construction. Here and in the following
we use $\sqcup$ and
$\bigsqcup$ to denote the disjoint union of graphs; recall the
definition of the transition group of a complete graph
\eqref{eq:transition group}.

\begin{Thm}\label{thm:first basic} Let $\mathcal{X}$ be a complete $E$-graph, $1\le k<|E|$ and suppose that the transition group $G=\mathrsfs{T}(\mathcal{X})$ is $k$-retractable. Then the transition group
\[H:=\mathrsfs{T}\left(\mathcal{X}\,\sqcup\,\bigsqcup \big\{\overline{\mathcal{G}[C]}\colon C\subseteq E, |C|=k\big\}\right)\] 
is $(k+1)$-retractable and is a $k$-stable expansion of $G$. Moreover, every $k$-stable expansion of $H$ is also $(k+1)$-retractable.
\end{Thm}
\begin{proof} We first show that $H$ is a $k$-stable expansion of
  $G$. So, let $p\in \til{E}^*$ be a word {with $\vert{\co(p)}\vert\le k$} and suppose that $[p]_G=1$. We need to show that
  $[p]_H=1$. In order to do so it is sufficient to show that, for
  every vertex $v$ in
  $\mathcal{X}\sqcup\bigsqcup_{|C|=k}\overline{\mathcal{G}[C]}$ the
  path $\pi_v(p)$ which starts at $v$ and has label $p$ is a
  cycle.
  This is obvious for every $v\in \mathcal{X}$ and $v\in
  \overline{\mathcal{G}[A]}$ when $A$ is a set of $k$ letters for
  which $p\in \til{A}^*$. So, let $B\subseteq E$ with $|B|=k$ and
  suppose that $p\notin \til{B}^*$, which means that at least one
  element of the content of $p$ does not belong to $B$,
  and let $v$ be a vertex of $\overline{\mathcal{G}[B]}$. Let $p'$ be
  the word obtained from $p$ by deletion of all letters from
  $\co(p)\setminus B$.
  Since $G$ is $k$-retractable, we have $[p']_G=1$ and hence also $[p']_{G[B]}=1$ since $p'$  contains only letters from $B$.  It follows that the path $\pi_v^{{\mathcal{G}[B]}}(p')$ is closed and hence so is $\pi_v^{\overline{\mathcal{G}[B]}}(p')$. Since the paths $\pi_v^{\overline{\mathcal{G}[B]}}(p)$ and $\pi_v^{\overline{\mathcal{G}[B]}}(p')$ meet exactly the same vertices --- the two paths differ only in loop edges labelled by letters from $\co(p)\setminus B$ --- it follows that $\pi_v^{\overline{\mathcal{G}[B]}}(p)$ is also closed. Altogether, $[p]_H=1$ and the expansion $H\twoheadrightarrow G$ is $k$-stable.

Let $K\twoheadrightarrow H$ be a $k$-stable expansion; then the
expansion $K\twoheadrightarrow G$ is also $k$-stable. We show that $K$
is $(k+1)$-retractable,
which then also applies to $K=H$.
So let  $A\subseteq E$ with $|A|=k+1$; according to
Proposition~\ref{prop:retractable groups} it suffices to show that for
every subset $B\subsetneq A$ there is a canonical morphism  {
$\mK[A]\twoheadrightarrow \ol{\mK[B]}^A$, where $\ol{\mK[B]}^A$ denotes the trivial completion 
of $\mK[B]$ as an $A$-graph, that is, loop edges labelled
by letters form $A\setminus B$ (and their inverses) are added to all
vertices of $\mK[B]$.} From the definition of $H$ and the assumption on
$K$ it follows that there is a canonical morphism
$\mK\twoheadrightarrow \overline{\mG[B]}$. {Indeed, if $|B|=k$ then by definition of $K$ and $H$, $\mK\twoheadrightarrow \mH\twoheadrightarrow \overline{\mG[B]}$ since $\overline{\mG[B]}$ is a component in the graph defining $H$ as a transition group. If $|B|<k$ we may choose a set $C$ with $B\subseteq C\subseteq A$ and $|C|=k$. Again there is a canonical morphism $\mK\twoheadrightarrow \ol{\mG[C]}$. Since $G[C]$ is retractable there is a canonical morphism $\mG[C]\twoheadrightarrow{\ol{\mG[B]}^C}$, where ${\ol{\mG[B]}^C}$ is the trivial completion of $\mG[B]$ as a $C$-graph. By adding loop edges for all labels not in $C$ to all vertices, the morphism $\mG[C]\twoheadrightarrow \ol{\mG[B]}^C$ can be extended to a morphism $\ol{\mG[C]}\twoheadrightarrow \ol{\mG[B]}$. Composition with the morphism $\mK\twoheadrightarrow \ol{\mG[C]}$ then yields the desired morphism $\mK\twoheadrightarrow \ol{\mG[B]}$.}
But $K\twoheadrightarrow G$ is $k$-stable, hence $K[B]\cong G[B]$ and therefore also $\ol{\mK[B]}\cong \ol{\mG[B]}$. It follows that the restriction of the morphism $\mK\twoheadrightarrow \ol{\mG[B]}\cong \ol{\mK[B]}$ to $\mK[A]$ provides the required morphism. 
\end{proof}

The principal idea of the paper is to construct a series of $E$-generated permutation groups
\begin{equation}\label{eq:series of G}
G_1 \twoheadleftarrow G_2\twoheadleftarrow\cdots\twoheadleftarrow G_{|E|}=:G
\end{equation}
defined by an ascending sequence $\mX_1\subseteq \mX_2\subseteq
\cdots\subseteq \mX_{|E|}$ of complete $E$-graphs such that
$G_k=\mathrsfs{T}(\mX_k)$ is $k$-retractable and
$G_{k+1}\twoheadrightarrow G_k$ is $k$-stable for every $k$.
The crucial property of this sequence
in relation to the given $E$-graph $\mathcal{E}$
is the following:

\bigskip\hfill
\begin{minipage}{.97\textwidth}
For every word $p\in \til{E}^*$ on $k+1$ letters which forms a path $u\overset{p}{\longrightarrow} v$ in $\mathcal{E}$ and every letter 
$a\in A:=\co(p)$  either there is a word $q$ in the letters $A
\setminus\{a\}$ such that $[p]_{G_{k+1}}=[q]_{G_{k+1}}$ and $q$ also forms a path $u\overset{q}{\longrightarrow}v$ in $\mathcal{E}$, or otherwise (if no such $q$ exists)   there is  a component in $\mX_{k+1}\setminus \mX_k$ which guarantees that $G_{k+1}$ avoids the relation $p=p_{a\to 1}$, so that $a$  belongs to the content of
$[p]_{G_{k+1}[A]}$ and therefore to the content of $[p]_G$.
\end{minipage}

\bigskip
The graph-theoretic constructions to be introduced in the following
are designed to serve this purpose.
In order to guarantee that $G_{k+1}\twoheadrightarrow G_k$ is
$k$-stable, the new components of $\mX_{k+1}$ are constructed in a way
so that their
$B$-components for $k$-element subsets $B$ of $E$ have already
occurred as subgraphs of $\mX_k$. 
This turns out to be a  challenging task. It crucially involves
$E$-graphs whose $A$-components 
for $(k+1)$-element subsets $A$ are designed so that their transition
groups avoid certain new relations over~$A$
but preserve all relations over $B$ for every $B\subsetneq A$. The
latter is guaranteed, as already mentioned, by the fact that all
$B$-components of the new components in $\mX_{k+1}$ have been
encountered already as subgraphs at earlier stages of the construction.

\subsection{Two crucial constructions: clusters and coset extensions}
\label{subsec:clustercosetextn}
We introduce two crucial constructions involving Cayley graphs. Let $G$ be an $E$-group; for $A\subseteq E$ and $g\in G$, $g\mG[A]$ has the obvious meaning: it denotes the $A$-component of the vertex $g$ of $\mG$ and is isomorphic (as an $A$-graph) with $\mG[A]$ --- we shall call such graphs \emph{$A$-coset graphs} or simply \emph{coset graphs} if the set of labels is understood. In the following subsections we shall construct new (bigger) graphs by gluing together disjoint copies of various coset graphs for different subsets $A\subseteq E$. In this context, the notation $v\mG[A]$, where $v$ is some vertex of a graph, means that the $A$-component of $v$ in the graph in question is isomorphic with the full $A$-coset graph $\mG[A]$.

\begin{Prov}
\label{prov:Aretractable}
For the remainder of the section
(\S~\ref{subsec:clustercosetextn}.1--3)
all $E$-groups $G$ are assumed to be $A$-retractable, i.e.\ $G[A]$ is retractable for the
(arbitrary but fixed) subset $A \subseteq E$ under consideration.  
\end{Prov}

In Sections~\ref{subsubsection:clusters}
and~\ref{subsubsection:cosetextns} we discuss families of
\emph{clusters} and \emph{coset extensions} 
whose $A$-components, as subgraphs of the $E$-graphs $\mX_k$,
provide the essential information for the setup of the
expansions in the series~(\ref{eq:series of G});
as discussed above, we need to account for their $B$-components for $B \subsetneq A \subseteq E$.

\subsubsection{Clusters}
\label{subsubsection:clusters}
Let $G$ be an $E$-group, $A\subseteq E$ and
assume that, as stated in Proviso~\ref{prov:Aretractable}, 
$G[A]$ is retractable. For every set $\mathbb{P}$ of proper subsets of $A$, the graph
\[\mathsf{CL}(G[A],\mathbb{P}):=\bigcup_{B\in \bP} \mathcal{G}[B]\subseteq \mG[A]\]
is an \emph{$A$-cluster}. Note that $\mathsf{CL}(G[A],\mathbb{P})$ is the subgraph of $\mathcal{G}[A]$ which is spanned by all $B$-paths in $\mathcal{G}[A]$ starting at $1$, for $B\in\mathbb{P}$.
The \emph{core} of the cluster is the subgraph formed by the intersection $\bigcap_{B\in \bP}\mathcal{G}[B]$, and by retractability of $G[A]$, 
$\bigcap_{B\in \bP}\mathcal{G}[B]=\mathcal{G}[\bigcap_{B\in \bP} B]$.
This core is always non-empty but may consist of the vertex $1$ only; the subgraphs $\mathcal{G}[B]$, for $B\in \mathbb{P}$, are the \emph{constituent cosets} of the cluster $\mathsf{CL}(G[A],\mathbb{P})$.
Included in the definition of an $A$-cluster is, for $\bP=\{B\}$, every graph $\mG[B]$ with $B\subsetneq A$.
The structure of $\mathsf{CL}(G[A],\mathbb{P})$ as an $A$-graph  actually  only depends on the collection of the ``small'' subgroups $G[B]$, $B\in\mathbb{P}$ rather than on the entire group $G[A]$: indeed the cluster can be assembled from the constituents $\mathcal{G}[B]$ by forming their disjoint union and factoring by the congruence which identifies an element (vertex or edge) of some $\mathcal{G}[B]$ and some $\mathcal{G}[C]$ if and only if these elements coincide in $\mathcal{G}[B\cap C]$ {(recall that retractability of $G[A]$ implies that $G[B\cap C]=G[B]\cap G[C]$)}.
More precisely, let $\varphi\colon \bigsqcup_{B\in \bP}\mG[B]\to \mG$
be the morphism which maps every coset graph $\mG[B]$ to itself, considered as a subgraph of $\mG$. Let $\Theta$ be the mentioned congruence on $\bigsqcup_{B\in\bP} \mG[B]$. Then the kernel $\ker \varphi$ of $\varphi$ (that is, the equivalence relation induced by $\varphi$ on its domain) contains $\Theta$; 
retractability of $G[A]$  even implies the equality $\ker\varphi=\Theta$. From the Homomorphism Theorem we get
\begin{equation}\label{eq: cluster}
\sfCL(G[A],\bP)\cong \mathrm{im}(\varphi)\cong \left.\bigsqcup_{B\in \bP} \mG[B]\right.\big/\, \Theta.    
\end{equation}

A consequence of this fact is the next lemma which will be of
essential use in  the proof of Proposition~\ref{prop:k-stabilityH_k--->}.

\begin{Lemma}\label{lem:small cosets 1}  Let $G\twoheadrightarrow H$ be a $(k-1)$-stable expansion between k-retractable $E$-groups $G$ and $H$. Then for any $A\subseteq E$ with $|A|=k$ and any set $\bP$ of proper subsets of $A$, the labelled graphs $\sfCL(G[A],\bP)$ and $\sfCL(H[A],\bP)$ are isomorphic.
\end{Lemma}
\begin{proof} This follows from the above discussion since
  $(k-1)$-stability implies that  $G[C]\cong H[C]$ for all $C\in \bP$ as $|C|<|A| = k$.
\end{proof}
We next analyse the structure of $B$-components of $A$-clusters for
$B\subsetneq  A$.
Let $\mathbb{P}=\{A_1,\dots,A_k\}$ be a set of proper subsets of $A$ and let $B\subsetneq A$; let $v\in G[A]$ and $v\mathcal{G}[B]$ be the $B$-component of $v$ in $\mathcal{G}[A]$.  Then
\[\mathsf{CL}(G[A],\mathbb{P})\cap v\mathcal{G}[B] =\bigcup_{i=1}^k(\mathcal{G}[A_i]\cap v\mathcal{G}[B]).\]
The intersection $\mathcal{G}[A_i]\cap v\mathcal{G}[B]$ is either empty or a $(B\cap A_i)$-coset $v_i\mathcal{G}[B\cap A_i]$ for some (any) $v_i\in \mathcal{G}[A_i]\cap v\mathcal{G}[B]$. {In order to describe the structure of the intersection $\mathsf{CL}(G[A],\mathbb{P})\cap v\mG[B]$ of a cluster $\mathsf{CL}(G[A],\mathbb{P})$ with a coset graph $v\mG[B]$ we may ignore the constituent cosets $\mG[A_i]$ of $\mathsf{CL}(G[A],\mathbb{P})$ having empty intersection with $v\mG[B]$. Hence we} may assume that $\mathcal{G}[A_i]\cap v\mathcal{G}[B]\ne \varnothing$ for every $i$.
\begin{Lemma}\label{lem:3-acyclic} If $\mathcal{G}[A_i]\cap v\mathcal{G}[B]\ne\varnothing$ for $i=1,\dots,k$ then
\[\mathcal{G}[A_1]\cap\cdots\cap\mathcal{G}[A_k]\cap v\mathcal{G}[B]\ne\varnothing.\]
\end{Lemma}
\begin{proof} Let $t\le k$ and assume that we have already proved that
\[\big(\bigcap_{i=1}^{t-1}\mathcal{G}[A_i]\big) \cap v\mathcal{G}[B]\ne\varnothing.\] In order to simplify the notation we set $C:=A_1\cap\cdots\cap A_{t-1}$ and $D:=A_t$; then $\bigcap_{i=1}^{t-1}\mathcal{G}[A_i] =\mG[C]$.
{The situation is depicted in Figure~\ref{fig:3acyclicinproofoflemma} and the reader is invited to 
consult this illustration for the following argument.
We need to exhibit an element in $\mathcal{G}[C\cap D]\cap v\mathcal{G}[B]$.}
So, let $u\in \big(\bigcap_{i=1}^{t-1}\mathcal{G}[A_i]\big) \cap v\mathcal{G}[B] =\mathcal{G}[C]\cap v\mathcal{G}[B]$ and $w\in \mathcal{G}[D]\cap v\mathcal{G}[B]$.
\begin{figure}[ht]
\begin{tikzpicture}[scale=0.6]
\filldraw(-3,-1)circle(2.5pt);
\draw(-2.9,-0.9)node[below left]{$1$};
\filldraw(0,0.5)circle(2.5pt);
\draw(0,0.5)node[left]{$x$};
\filldraw(3,-1)circle(2.5pt);
\draw(3,-0.9)node[right]{$w$};
\filldraw(1,3)circle(2.5pt);
\draw(1,3)node[above]{$u$};
\draw[-latex][thick](1,3)--(-2.9,-0.9);
\draw[-latex][thick](1,3)--(3,-0.9);
\draw[-latex][thick](1,3)--(0.1,0.6);
\draw[-latex][thick](3,-1)--(0.1,0.4);
\draw[-latex][thick](3,-1)--(-2.9,-1);
\draw(-1,1.3)node{$p$};
\draw(0.8,1.3)node{$p_1$};
\draw(0,-0.9)node[below]{$q$};
\draw(2.3,1.3)node{$r$};
\draw(1.7,0)node{$q_1$};
\draw plot [smooth cycle] coordinates {(0.5,3.5) (1.5,3.4)(2.7,1.3) (4,-1) (2,-1)(-0.8,0)};
\draw plot [smooth cycle] coordinates {(-4,-1.5)(0,-0.3)(1.5,3.5)(-5,3)};
\draw plot [smooth cycle] coordinates {(-4,-1) (0,1)(4,-1)(1,-1.8)(-3,-2)};
\draw(0,-1.8)node[below]{$\mathcal{G}[D]$};
\draw(3.3,2)node{$v\mathcal{G}[B]$};
\draw(-2.2,2.9)node{$\mathcal{G}[C]$};
\end{tikzpicture}
\caption{Configuration as in the proof of Lemma~\ref{lem:3-acyclic}}
 \label{fig:3acyclicinproofoflemma}
\end{figure}
Let $p\in \til{C}^*$ be such that $[p]_G=u^{-1}$ and $q\in \til{D}^*$ be such that $[q]_G=w^{-1}$; moreover, let $r\in\til{B}^*$ be a word which labels a path $u\longrightarrow w$ running entirely in $v\mathcal{G}[B]$ (recall that all this happens in $\mathcal{G}[A]$). Let $p_1$ and $q_1$ be, respectively, the words obtained from $p$ and $q$ by deletion of all letters not in $B$. Since $[pq^{-1}]_G=[r]_G$ and $r\in \til{B}^*$, we have $[p_1q_1^{-1}]_G=[r]_G$, by retractability.
Let $x:=u\cdot p_1=w\cdot q_1$. Then $p^{-1}p_1$ labels a path
$1\longrightarrow x$ and so does $ q^{-1}q_1$. Since $p^{-1}p_1\in
\til{C}^*$ and $q^{-1}q_1\in
\til{D}^*$,  it follows that $x\in \mathcal{G}[C]\cap \mathcal{G}[D]=\mathcal{G}[C\cap D]$. From $x=u\cdot p_1$ and $p_1\in \til{B}^*$ it follows that
$x\in u\mathcal{G}[B]=v\mathcal{G}[B]$,
altogether
$x\in \mathcal{G}[C\cap D]\cap v\mathcal{G}[B]$.
\end{proof}

The proof of Lemma~\ref{lem:3-acyclic} implicitly shows that
retractable groups
are $3$-acyclic in the sense of condition~(\ref{eq:3acyclic}),
as stated in Remark~\ref{rem:2and3acyc}.
(Compare Figure~\ref{fig:2/3-acyclic} for coset patterns that are 
ruled out in the Cayley graph $\mathcal{G}$ of any 
$E$-group $G$ that is retractable; here now, the cosets in question,  
$1 \mathcal{G}[C]$, $1\mathcal{G}[{D}]$ and $v\mathcal{G}[B]$,
have $x$ in their intersection,
as indicated in Figure~\ref{fig:3acyclicinproofoflemma}.)

\medskip
In the situation of the proof of Lemma~\ref{lem:3-acyclic}   we consider the automorphism of $\mathcal{G}$ induced by left
multiplication by $x^{-1}$ for some $x\in G[A_1]\cap\cdots\cap
G[A_k]\cap vG[B]$. {Then $\mG[A_i]= x\inv\mG[A_i]$ for all $i$, and $x\inv v\mG[B]=\mG[B]$, so that 
\begin{equation}\label{eq:intersection coset cluster}
\begin{array}{r@{\;\;=\;\;}ll}    
x\inv\big(\mathsf{CL}(G[A],\mathbb{P})\cap v\mathcal{G}[B]\big) &
\bigcup_{i=1}^k(\mathcal{G}[A_i]\cap \mathcal{G}[B])
\\
\hnt
& \bigcup_{i=1}^k\mathcal{G}[A_i\cap B] \; = \; \mathsf{CL}(G[B],\mathbb{O})
\end{array}
\end{equation}
where
$\mathbb{O}=\{B\cap A_i\colon A_i\in \mathbb{P}\}$} (some of the sets
$B\cap A_i$ may be empty), which perhaps degenerates to a full $B$-coset. 
This allows us to characterise the $B$-components of $A$-clusters for $B\subsetneq A$.
\begin{Cor}\label{cor:Bcomponentscluster} Let $\mathbb{P}$ be a set of proper subsets of $A$ and $B\subsetneq A$. Then every $B$-component of the cluster $\mathsf{CL}(G[A],\mathbb{P})$ is either a $B$-coset, that is, isomorphic with $\mathcal{G}[B]$, or isomorphic with the $B$-cluster $\mathsf{CL}(G[B],\mathbb{O})$ where $\mathbb{O}=\{C\cap B\colon C\in\mathbb{P}\}$ (some $C\cap B$ may be empty).
\end{Cor}
\begin{proof} The intersection $\mathsf{CL}(G[A],\mathbb{P})\cap v\mathcal{G}[B]$ is either the $B$-coset $v\mathcal{G}[B]$ itself (if it is contained in some constituent $\mathcal{G}[C]$ with $C\in \bP$) or otherwise is isomorphic with the $B$-cluster $\mathsf{CL}(G[B],\mathbb{O})$, as indicated {in~\eqref{eq:intersection coset cluster}}. Now let $v$ be a vertex of $\mathsf{CL}(G[A],\mathbb{P})$; then the $B$-component $\mathcal{B}$ of $v$ in $\mathsf{CL}(G[A],\mathbb{P})$ is certainly contained in $\mathsf{CL}(G[A],\mathbb{P})\cap v\mathcal{G}[B]$. Since the latter intersection is a $B$-cluster, it is connected and therefore $\mathcal{B}$ must coincide with this intersection.
\end{proof}
\begin{Cor}\label{cor:BcapC for clusters} Let $B,C\subsetneq A$; then the intersection $\mathcal{B}\cap\mathcal{C}$ of a $B$-component $\mathcal{B}$ with a $C$-component $\mathcal{C}$ of an $A$-cluster $\mathsf{CL}$ is either empty or a $B\cap C$-coset or a $(B\cap C)$-cluster.
\end{Cor}
\begin{proof}
{By Corollary~\ref{cor:Bcomponentscluster}}, $\mathcal{B}=\mathsf{CL}\cap v\mathcal{G}[B]$ and
$\mathcal{C}=\mathsf{CL}\cap w\mathcal{G}[C]$ for some cosets
$v\mathcal{G}[B]$ and $w\mathcal{G}[C]$.
The latter two have either empty intersection or their intersection is a $(B\cap C)$-coset $u\mathcal{G}[B\cap C]$ from which the claim follows.
\end{proof}

We will need a generalisation of clusters, which  we are going to
present next. Let {$G[A]$ be again} retractable (Proviso~\ref{prov:Aretractable}),
$\mathbb{P}$ be a set of proper subsets of $A$, $v$ be a vertex of $\mathsf{CL}(G[A],\mathbb{P})$ and $B\subsetneq A$. Under these assumptions we define
\[\mathsf{CL}(G[A],\mathbb{P})\circv \mathcal{G}[B]:= \bigcup_{C\in \mathbb{P}}\mathcal{G}[C]\cup v\mathcal{G}[B]\]
considered as a subgraph of $\mathcal{G}[A]$ and call the latter graph a 
\emph{$B$-augmented $A$-cluster} or, more specifically, the \emph{$B$-augmentation of $\sfCL(G[A],\bP)$  at $v$}. We have seen 
{in Corollary~\ref{cor:Bcomponentscluster} } that the intersection $\mathsf{CL}(G[A],\mathbb{P})\cap v\mathcal{G}[B]$ is a $B$-component of $\mathsf{CL}(G[A],\mathbb{P})$. 
It follows that the structure of the graph
$\mathsf{CL}(G[A],\mathbb{P})\circv \mathcal{G}[B]$ only depends on
the collection $\{G[C]\colon C\in \mathbb{P}\}$, the vertex $v$ and
$G[B]$ rather than on the entire group $G[A]$. Indeed, {as can be seen in \eqref{eq: cluster},} the structure of $\mathsf{CL}(G[A],\mathbb{P})$ depends only
on the graphs $\mathcal{G}[C]$ for $C\in \mathbb{P}${; furthermore, 
by Corollary~\ref{cor:Bcomponentscluster},} the
$B$-component of $v$ is a certain $B$-cluster $\mathcal{B}$, which is
isomorphic with a 
subgraph of $\mathcal{G}[B]$ via the
monomorphism $\iota\colon \mathcal{B}\to \mathcal{G}[B]$ determined
by $v\mapsto 1$. {(We may neglect the trivial case 
in Corollary~\ref{cor:Bcomponentscluster}, namely that $\mB=\mG[B]$:  in that case, the augmented cluster would coincide with the original one.)}
The augmented cluster $\mathsf{CL}(G[A],\mathbb{P})\circv
\mathcal{G}[B]$ can then be obtained as the disjoint union of
$\mathsf{CL}(G[A],\mathbb{P})$ and $\mathcal{G}[B]$ factored by the
congruence whose non-singleton classes are $\{x,\iota(x)\}$ for all
$x\in \mathcal{B}$ ($x$ an edge or a vertex).
As a consequence we obtain the following lemma, whose proof is
analogous to the proof of Lemma~\ref{lem:small cosets 1}; it will
similarly be used in the proof of
Proposition~\ref{prop:k-stabilityH_k--->}.  

\begin{Lemma}\label{lem:small cosets 2}  Let $G\twoheadrightarrow H$ be a $(k-1)$-stable expansion between k-retractable $E$-groups $G$ and $H$,
 $\varphi$ the associated canonical morphism. Then, for any $A\subseteq E$ with $|A|=k$, any set $\bP$ of proper subsets of $A$, any $B\subsetneq A$ and any vertex $u$ of $\sfCL(G[A],\bP)$ with $v:=\varphi(u)$ there is an isomorphism of labelled graphs
\[\sfCL(G[A],\bP)\circu\mG[B]\cong \sfCL(H[A],\bP)\circv\mH[B].\]
\end{Lemma}

As the last result in this subsection we need to clarify, for 
$B,C\subsetneq A$, the structure of $C$-components of $B$-augmented $A$-clusters. These turn out to be $(B\cap C)$-augmented
$C$-clusters. As noticed in Corollary~\ref{cor:Bcomponentscluster},
every $C$-component {of} an $A$-cluster is a $C$-cluster (or a $C$-coset).
\begin{Cor}\label{cor:CcomponentBexpandedcluster} Let $B,C\subsetneq A$ and let $G[A]$ be retractable; then every $C$-component of a $B$-augmented
$A$-cluster is a $(B\cap C)$-augmented
 $C$-cluster (which includes $C$-clusters as a special case).
\end{Cor}
\begin{proof}
Let the group $G$ and $A,B,C\subseteq E$  be as in the statement of the corollary. Let $\mathsf{CL}(G[A],\mathbb{P})\circv \mathcal{G}[B]$ be a $B$-augmentation
of the $A$-cluster $\mathsf{CL}(G[A],\mathbb{P})$ and let $u$ be a
vertex of this  cluster. If the $C$-component $\mathcal{C}$ of $u$ in
$\mathsf{CL}(G[A],\mathbb{P})$ has empty intersection with the
$B$-component $\mathcal{B}$ of $v$ in $\mathsf{CL}(G[A],\mathbb{P})$
then $\mathcal{C}$ coincides with the $C$-component of $u$ in the
augmented cluster and we are done as $\mC$ is a $C$-cluster (or a $C$-coset). Now assume that $\mathcal{C}\cap \mathcal{B}\ne\varnothing$ with $w$ a vertex in $\mathcal{C}\cap \mathcal{B}$. We know that $\mathcal{C}\cap\mathcal{B}$ is a $(C\cap B)$-cluster (Corollary~\ref{cor:BcapC for clusters}) or a $(C\cap B)$-coset and the $C$-component of $w$ within $v\mathcal{G}[B]=w\mathcal{G}[B]$ consists exactly of the coset $w\mathcal{G}[B\cap C]$. It follows that the $C$-component of $w$ in $\mathsf{CL}(G[A],\mathbb{P})\circv\mathcal{G}[B]$ coincides with $\mathcal{C}\cup w\mathcal{G}[B\cap C]=\mathcal{C}\circw\mathcal{G}[B\cap C]$ which is a $(B\cap C)$-augmentation
of the $C$-cluster $\mathcal{C}$.
\end{proof}

\subsubsection{Coset extensions}
\label{subsubsection:cosetextns}
This second construction{, coset extensions}, can be seen as a generalisation of clusters, but is more involved.
{It is a somewhat complex concept but it is perhaps \textsl{the} essential construction of the paper. Its definition will be developed over the next few pages and, in a sense, only culminates in the defining equation~\ref{eq:CEgeneral}; 
but readers should bear in mind that its justification, 
including the non-trivial verification of its well-definedness, 
crucially relies on preparations expounded in the following pages.}

Let us fix  an $E$-group $G$  and a set $A\subseteq E$ of size $|A|\ge
2$. We assume that $G$ is $A$-retractable,
according to Proviso~\ref{prov:Aretractable}.
Let $\mathcal{K}$  be a connected $A$-subgraph of the Cayley graph $\mathcal{G}$ of $G$. We recall that being an $A$-subgraph means that all labels of edges of $\mathcal{K}$ belong to $\til{A}$ (but not necessarily all such letters actually need to occur in $\mathcal{K}$).
For some set $B\subsetneq A$ let $\mathcal{B}=v\mathcal{K}[B]$ be some
$B$-component of $\mathcal{K}$; this graph is embedded in
$v\mathcal{G}[B]\cong \mathcal{G}[B]$. Moreover, for
$B_1,B_2\subsetneq B$ any $B_1$- and $B_2$-components $\mathcal{B}_1$
and $\mathcal{B}_2$ of $\mathcal{B}$ are also  embedded in
$v\mathcal{G}[B]$ via their embedding in $\mB$.

\begin{Def}[admissibility for coset extension] \label{def:admissible} 
  Let $G$  be an $E$-group,  $A\subseteq E$ with $|A|\ge 2$,
  and assume that
  $G$ is $A$-retractable (Proviso~\ref{prov:Aretractable}).
Let $\mathcal{K}$  be a connected $A$-subgraph of the Cayley graph $\mathcal{G}$ of $G$.  Consider all possible choices of subsets $B_1,B_2\subsetneq B\subsetneq A$, of $B$-components $\mB=v\mK[B]$ of $\mK$ and for each pair of vertices $v_1,v_2\in \mB$ all possible $B_1$- and $B_2$-components $\mB_1=v_1\mB[B_1]=v_1\mK[B_1]$ and $\mB_2=v_2\mB[B_2]=v_2\mK[B_2]$. Then $\mathcal{K}$ is \emph{admissible for $\pss{A}$-coset extension} (with respect to $G$) if
\begin{equation}\label{eq:freeness}
\mathcal{B}_1\cap \mathcal{B}_2=\varnothing \mbox{ in }\mathcal{B}\Longrightarrow v_1\mathcal{G}[B_1]\cap v_2\mathcal{G}[B_2]=\varnothing \mbox{ in } v\mathcal{G}[B]\subseteq \mG.
\end{equation}
\end{Def}

 In other words, the patterns depicted in
 Figure~\ref{fig:admissibility}  are forbidden in the context of a
 graph $\mK$ that is admissible for
$\pss{A}$-coset extension
(the right-hand side picture is for the case $B_1=B_2$).
\begin{figure}[ht]
\begin{tikzpicture}[xscale=0.4,yscale=0.6]
\draw (-7.3,-2)--(7,-2);
\draw (-6,-1.9) node [below left]{$\mathcal{K}$};
\draw[thick](-6,-2)--(6,-2);
\draw[very thick](-5,-2)--(-2,-2);
\draw[very thick] (2,-2)--(5,-2);
\draw(-2.8,-1.7)node{$\mathcal{B}_1$};
\draw(2.8,-1.7)node{$\mathcal{B}_2$};
\draw(0,-2.1)node[above]{$\mathcal{B}$};
\draw plot [smooth cycle] coordinates {(5,-2) (2,-2) (-2,2) (0,3)(4,3)};
\draw plot[smooth cycle]coordinates {(-6,-2) (0,-3) (6,-2)(6,3) (0,4)(-6,3)};
\draw plot [smooth cycle] coordinates {(-5,-2) (-2,-2) (2,1) (2,3)(-1,2)};
\draw(-4,-2.1)node[above]{$v_1$};
\draw(4,-2.1)node[above]{$v_2$};
\filldraw (-4,-2) circle (2.5pt);
\filldraw(4,-2)circle(2.5pt);
\draw(-1.8,-0.6)node{$v_1\mathcal{G}[B_1]$};
\draw(3,0)node{$v_2\mathcal{G}[B_2]$};
\draw(0,3.9)node[above]{$v\mathcal{G}[B]$};
\end{tikzpicture}
\begin{tikzpicture}[xscale=0.4,yscale=0.6]
\draw (-7.5,2)node{};
\draw (-7,-2)--(7,-2);
\draw (7.2,-1.9) node [below left]{$\mathcal{K}$};
\draw[thick](-6,-2)--(6,-2);
\draw[very thick](-5,-2)--(-2,-2);
\draw[very thick] (2,-2)--(5,-2);
\draw(-2.8,-1.7)node{$\mathcal{B}_1$};
\draw(2.8,-1.7)node{$\mathcal{B}_2$};
\draw(0,-2.1)node[above]{$\mathcal{B}$};
\draw plot[smooth cycle]coordinates {(-6,-2) (0,-3) (6,-2)(6,3) (0,4)(-6,3)};
\draw plot [smooth cycle] coordinates {(-5,-2) (-2,-2) (0,0) (2,-2)  (5,-2) (5,2) (-5,2)};
\draw(-4,-2.1)node[above]{$v_1$};
\draw(4,-2.1)node[above]{$v_2$};
\filldraw (-4,-2) circle (2.5pt);
\filldraw(4,-2)circle(2.5pt);
\draw(0,1.5)node{$v_1\mathcal{G}[B_1]=v_2\mathcal{G}[B_2]$};
\draw(0,3.9)node[above]{$v\mathcal{G}[B]$};
\end{tikzpicture}
\caption{Configurations violating admissibility (Definition~\ref{def:admissible})}\label{fig:admissibility}
\end{figure}
The condition formulated in Definition~\ref{def:admissible}
 corresponds to the notion of \emph{freeness} in~\cite{otto3}, here
 for the embedded graphs $\mathcal{B} = v\mK[B]$ in $v \mathcal{G}[B]$. 
  We note that, if $\mathcal{K}$ is admissible for
 $\pss{A}$-coset extension,
  then, for every $B\subseteq A$, every $B$-component
 $v\mathcal{K}[B]$
 is admissible for
 $\pss{B}$-coset extension.
 
Now let $\mathcal{K}$ be a subgraph of $\mathcal{G}$ that is
admissible for
$\pss{A}$-coset extension
and fix a set $B\subsetneq A$. Let $\mathcal{B}_1,\dots,\mathcal{B}_k$ be all the
$B$-components of $\mathcal{K}$. For every $i=1,\dots, k$ select a
vertex $v_i\in \mathcal{B}_i$. Then, in $\mathcal{G}$, the coset
$v_i\mathcal{G}[B]$ contains $\mathcal{B}_i$ as a subgraph. Let now
$\mathsf{CE}(G,\mathcal{K};B)$ be the graph obtained by 
extending each
component $\mathcal{B}_i$ in $\mK$ to the entire coset
$v_i\mathcal{G}[B]$. So $\mathsf{CE}(G,\mathcal{K};B)$ is
the graph obtained by attaching in $\mathcal{K}$ to each vertex $v_i$
a copy $v_i\mathcal{G}[B]$ of $\mathcal{G}[B]$ and then identifying
all of $\mathcal{B}_i$ with its copy inside $v_i\mathcal{G}[B]$, but
without performing any further identification (of vertices and/or
edges). The graph $\mathsf{CE}(G,\mathcal{K};B)$ thus appears as a
bunch of pairwise disjoint copies of $\mathcal{G}[B]$, connected by
edges labelled by letters from $A\setminus B$. The union of the
latter edges  with all the $\mathcal{B}_i$ then spans the graph $\mathcal{K}$.

We give a more formal definition of $\mathsf{CE}(G,\mathcal{K};B)$. Let $\mathcal{K}$ be given with $B$-components $\mathcal{B}_1,\dots,\mathcal{B}_k$ and selected vertices $v_i\in \mathcal{B}_i$  for $i=1,\dots, k$. For every $i$ let $\iota_i\colon \mathcal{B}_i\to\mathcal{G}[B]$ be the unique graph monomorphism mapping $v_i$ to $1$. Then
\begin{equation} \label{eq:CE(K,B)}
\mathsf{CE}(G,\mathcal{K};B) :=\left.\big(\mathcal{K}\cup \bigcup_{i=1}^k
  \mathcal{G}[B]\times \{i\}\big)\right.\big/\,\Theta
\end{equation}
where $\Theta$ is the equivalence relation all of whose non-singleton equivalence classes  are exactly the two-element sets
\[\{x,(\iota_i(x),i)\}\mbox{ with }x\in \mathcal{B}_i,\ i=1,\dots,k\]
where $x$ denotes a vertex or an edge of $\mB_i$.
The union on the right-hand side of \eqref{eq:CE(K,B)} is a union of
pairwise disjoint connected graphs and $\Theta$ is certainly a
congruence relation. 
The resulting graph
$\mathsf{CE}(G,\mathcal{K};B)$ is the \emph{$B$-coset extension of the
  $A$-graph $\mathcal{K}$}.
The congruence $\Theta$ does not identify any two elements (edges or vertices) of $\mK$ with each other, hence $\sfCE(G,\mK;B)$ contains $\mK$ as a subgraph in a canonical way which, in this context, is called the \emph{skeleton}
of $\mathsf{CE}(G,\mathcal{K};B)$. For $v_i\in \mB_i\subseteq
\mK\subseteq \sfCE(G, \mK;B)$ the $B$-component of $v_i$ in
$\sfCE(G,\mK;B)$ is isomorphic with the coset graph $\mG[B]$. Hence
these $B$-components of $\sfCE(G,\mK;B)$  will
also be denoted by $v_i\mG[B]$ and
addressed as \emph{constituent cosets} of $\sfCE(G,\mK;B)$ in this r\^ole. 

For $C\subsetneq B\subsetneq A$, 
condition~\eqref{eq:freeness} of  Definition~\ref{def:admissible}
(by taking $B_1=C=B_2$) implies
that  $\mathsf{CE}(G,\mathcal{K};C)$ is realised
as a subgraph of
$\mathsf{CE}(G,\mathcal{K};B)$. Moreover, for $C_1,C_2\subsetneq B$, $C_1\ne C_2$, once more condition \eqref{eq:freeness} (this time taking $C_1=B_1\ne B_2=C_2$) implies that 
\begin{equation} \label{eq:CE-boolean}
\mathsf{CE}(G,\mathcal{K};C_1)\cap\mathsf{CE}(G,\mathcal{K};C_2)= \mathsf{CE}(G,\mathcal{K};C_1\cap C_2)
\end{equation} 
where the intersection takes place in $\sfCE(G,\mK;B)$.
 Now let $\mathbb{P}$ be a set of proper subsets of $A$. Then the \emph{$\bP$-coset extension of $\mathcal{K}$} is defined as
\begin{equation}\label{eq:CEgeneral}
\mathsf{CE}(G,\mathcal{K};\mathbb{P}):=\left(\bigcup\big\{\mathsf{CE}(G,\mathcal{K};B)\times\{B\}\colon {B\in \mathbb{P}}\big\}\right) \big/\,\Psi
\end{equation}
where $\Psi$ is the congruence defined on the disjoint union of all $B$-coset extensions $\mathsf{CE}(G,\mathcal{K};B)$ with $B\in \mathbb{P}$, by setting
\[(x_1,B_1)\mathrel{\Psi}(x_2,B_2):\Longleftrightarrow x_1=x_2\in \mathsf{CE}(G,\mathcal{K};B_1\cap B_2).\]
In other words, an edge or a vertex of
$\mathsf{CE}(G,\mathcal{K};B_1)$ is identified with one in
$\mathsf{CE}(G,\mathcal{K};B_2)$ if they represent the same element in
$\mathsf{CE}(G,\mathcal{K}; B_1\cap B_2)$.
Transitivity of $\Psi$ follows from \eqref{eq:CE-boolean}:
indeed, for $i=1,2,3$, let $B_i\in \bP$ and $x_i\in \sfCE(G,\mK;B_i)$ be such that $(x_1,B_1)\mathrel{\Psi}(x_2,B_2)$ and $(x_2,B_2)\mathrel{\Psi}(x_3,B_3)$. Then
\[x_1=x_2\in \sfCE(G,\mK;B_1\cap B_2)\mbox{ and }x_2=x_3\in \sfCE(G,\mK;B_2\cap B_3)\] so that
\[x_1=x_3\in \sfCE(G,\mK;B_1\cap B_2)\cap \sfCE(G,\mK;B_2\cap B_3)=\sfCE(G,\mK;B_1\cap B_2\cap B_3)\]
by application of \eqref{eq:CE-boolean} for $C_1=B_1\cap B_2$,
$C_2=B_2\cap B_3$ and $B=B_2$,
where the intersection takes place in $\sfCE(G,\mK;B_2)$.
Provided that $B\in \mathbb{P}$, the coset extension
$\mathsf{CE}(G,\mathcal{K};B)$ is embedded in
$\mathsf{CE}(G,\mathcal{K};\mathbb{P})$ via $x\mapsto (x,B)\Psi$
where $(x,B)\Psi$ denotes the $\Psi$-class of $(x,B)$. 
For $v\in \mK$ and $B\in \bP$, the subgraphs $v\mG[B]$ of $\sfCE(G,\mK;\bP)$ are the \emph{constituent cosets} of $\sfCE(G,\mK;\bP)$
 and the subgraph $\mK$ is the \emph{skeleton} of $\sfCE(G,\mK;\bP)$.

Geometrically, the coset extension
$\mathsf{CE}(G,\mathcal{K};\mathbb{P})$ can be viewed as follows.
For every $B\in \mathbb{P}$ consider $\mathsf{CE}(G,\mathcal{K};B)$ and attach these graphs to each other by identification   of their skeleton $\mathcal{K}$, then form the largest $E$-graph quotient (that is, perform all identifications necessary to obtain an $E$-graph, but no more). 
The graph $\sfCE(G,\mK;\bP)$ then is  the union 
\[\sfCE(G,\mK;\bP)=\bigcup_{B\in \bP}\sfCE(G,\mK;B)\] of its subgraphs $\sfCE(G,\mK;B)$ with $B\in \bP$. 
 For $B_1,B_2\in \bP$ then
\begin{equation}\label{eq:boolean extended}
\sfCE(G,\mK;B_1)\cap\sfCE(G,\mK;B_2)=\sfCE(G,\mK;B_1\cap B_2).
\end{equation}
This is reminiscent of \eqref{eq:CE-boolean} but $B_1$ and $ B_2$ 
are now arbitrary members of $\bP$ (rather than subsets of some $B\subsetneq A$) and the intersection takes place in $\sfCE(G,\mK;\bP)$ (rather than in $\sfCE(G,\mK;B)$). Moreover, condition \eqref{eq:boolean extended} can be reformulated as a condition analogous to \eqref{eq:freeness}: for  any $B_1,B_2\in \bP$ and vertices $v_1,v_2\in\mK$:
\begin{equation}\label{eq:freeness extended}
v_1\mK[B_1]\cap v_2\mK[B_2]=\varnothing\Longrightarrow 
v_1\mG[B_1]\cap v_2\mG[B_2]=\varnothing
\end{equation}
where the intersections take place in $\sfCE(G,\mK;\bP)$.

If every label of $\mathcal{K}$ appears in some member $B$ of
$\mathbb{P}$, then $\mathsf{CE}(G,\mathcal{K};\mathbb{P})$ is weakly
complete since every edge of $\mathsf{CE}(G,\mathcal{K};\mathbb{P})$
occurs in some coset subgraph $v\mathcal{G}[B]$. Most relevant will be
the case $\mathbb{P}=\mathbb{P}_A$, the set of \textsl{all} proper
subsets of $A$: we call $\mathsf{CE}(G,\mathcal{K};\mathbb{P}_A)$
the \emph{full}
$\pss{A}$-coset extension
of $\mathcal{K}$. In case $\mathcal{K}=\{v\}$ (one vertex, no edge) the $\bP$-coset extension $\mathsf{CE}(G,\mathcal{K};\mathbb{P})$ reduces to the cluster $\mathsf{CL}(G[A],\mathbb{P})$.
\begin{Rmk}\label{rmk: proper A-graph} An $A$-graph $\mK$ which is admissible for
$\pss{A}$-coset extension
may actually only contain edges labelled by letters
 (and their inverses) from some set $B\subsetneq A$. In this case $\sfCE(G,\mK;B)\cong \mG[B]$; however, this is not in
conflict with the definition of the full
$\pss{A}$-coset extension.
For sets $C\subsetneq A$ with $C\nsubseteq B$, the $C$-components of
$\mK$ coincide with the $C\cap B$-components, but nevertheless every
such $C\cap B$-component is extended to a full $C$-coset $v\mG[C]$ in order to get $\sfCE(G,\mK;C)$.
\end{Rmk}

We continue with further investigations of
$\pss{A}$-coset extensions.
\begin{Prop}\label{prop:morphismCEtoGA} Let $\mathcal{K}\subseteq
  \mathcal{G}[A]$ be admissible for
$\pss{A}$-coset extension
and $\mathbb{P}$ be a set of proper subsets of $A$. Then the inclusion monomorphism $\iota\colon \mathcal{K}\hookrightarrow \mathcal{G}[A]$ admits a unique extension to a graph morphism ${\iota_\mathbb{P}}\colon \mathsf{CE}(G,\mathcal{K};\mathbb{P})\to \mathcal{G}[A]$.
\end{Prop}
\begin{proof} We first establish a unique extension $\iota_B\colon
  \mathsf{CE}(G,\mathcal{K};B)\to \mathcal{G}[A]$
  for each $B \in \mathbb{P}$.
Let $\mathcal{B}_1,\dots,\mathcal{B}_k$ be all $B$-components of $\mathcal{K}$ with {selected} vertices $v_i\in \mathcal{B}_i$ for all $i$. Then for every $i$ there is a unique graph monomorphism $\kappa_i\colon \mathcal{G}[B]\times\{i\}\to \mathcal{G}[A]$ such that $\kappa_i(1,i)=v_i$. The image of $\kappa_i$ coincides with the coset subgraph $v_i\mathcal{G}[B]$ of $\mathcal{G}[A]$. Then, the union $\kappa:=\iota\cup\bigcup_{i=1}^k\kappa_i$ is a morphism \[\kappa\colon\mathcal{K}\cup\bigcup_{i=1}^k\mathcal{G}[B]\times\{i\}\to \mathcal{G}[A]\] for which, for all $i$ and $x\in \mathcal{B}_i$,
\[ \kappa(x)=\iota(x)=x=\kappa_i(\iota_i(x),i)=\kappa(\iota_i(x),i)\]
where $\iota_i\colon \mB_i\to \mG[B]$ is the unique graph monomorphism
mapping $v_i$ to $1$ that
occurs in the definition of $\sfCE(G,\mK;B)$.
It follows that the congruence $\Theta$ in \eqref{eq:CE(K,B)} is
contained in the {kernel} of $\kappa$  
and hence $\kappa$ factors through $\mathsf{CE}(G,\mathcal{K};B)$ as
$\kappa=\iota_B\circ \pi_\Theta$  (where $\pi_\Theta$ is the canonical
projection $\pi_\Theta(x)=x\Theta$).

Next consider the disjoint union
\[\bigcup_{B\in\mathbb{P}}\mathsf{CE}(G,\mathcal{K};B)\times\{B\}\] and let
\[\kappa_{\mathbb{P}}:=\bigcup_{B\in \mathbb{P}}{\boldsymbol{\iota}_B}\colon
  \bigcup_{B\in\mathbb{P}}\mathsf{CE}(G,\mathcal{K};B)\times\{B\}\to
  \mathcal{G}[A]
\]
where ${\boldsymbol{\iota}_B}\colon \sfCE(G,\mK;B)\times \{B\}\to \mG[A]$ is defined by ${\boldsymbol{\iota}_B}(x,B)=\iota_B(x)$.
Similar to $\Theta$ and $\kappa$, the congruence $\Psi$ that occurs
in \eqref{eq:CEgeneral} is contained in the kernel of
${\kappa_{\mathbb{P}}}$,
whence ${\kappa_{\mathbb{P}}}$ factors through $\sfCE(G,\mK;\bP)$ as ${\kappa_{\mathbb{P}}}={\iota_{\mathbb{P}}}\circ\pi_\Psi$ for some unique morphism ${\iota_{\mathbb{P}}}\colon \mathsf{CE}(G,\mathcal{K};\mathbb{P})\to \mathcal{G}[A]$ (with $\pi_{\Psi}$ being again the projection $x\mapsto x\Psi$).
\end{proof}

The morphism ${\iota_B}\colon \sfCE(G,\mK;B)\to \mG[A]$ is injective when restricted
either to the skeleton $\mathcal{K}$ or to any constituent coset
$v\mathcal{G}[B]$. 
However, in general $\iota_B$ is not injective on its entire domain
$\sfCE(G,\mK;B)$. Within $\mG[A]$ it may happen that for distinct
vertices $v_i\ne v_j$ (as selected in the above proof) the
corresponding cosets coincide: $v_i\mG[B]=v_j\mG[B]=:v\mG[B]$.
In this case, $\iota_B$ maps $v_i\mG[B]$ as well as $v_j\mG[B]$ onto
$v\mG[B]\subseteq\mG[A]$, although $\iota_B(v_i\mK[B])$ and
$\iota_B(v_j\mK[B])$ are distinct (and hence disjoint) $B$-components
of $\mK$ within $v\mG[B]\subseteq\mG[A]$ (see Figure~\ref{fig:unfoldedCE}).
The coset $v\mG[B]$ then contains (at least) two distinct $B$-components
$\mB_i\ne \mB_j$ of $\mK$. As a consequence, the vertices $v_i$ and
$v_j$ can be connected by a $B$-path in $\mG[A]$, but there is no
$B$-path connecting these vertices  in $\mK$. This alludes to one of
the key ideas of the paper and will eventually lead to the proof of
the crucial Lemma~\ref{lem:contentconnected}.

\begin{Rmk} \label{rmk:unfolding}Suppose that $H\twoheadrightarrow G$ is an  expansion 
whose Cayley graph $\mH$ covers some completion of (some supergraph
of) $\mathsf{CE}(G,\mK;B)$.
Then the group $H$ avoids \textsl{every} relation $p=q$ where $p$ is any word
labelling a path in $\mK$ that connects two distinct $B$-components of
$\mK$ and $q$ is any $\til{B}$-word, essentially because the graph $\sfCE(G,\mK;B)$ unfolds the subgraph $\mK\cup\bigcup v_i\mG[B]$ of
$\mG[A]$ that arises as the image  of $\sfCE(G,\mK;B)$ under
$\iota_B$ (see Figure~\ref{fig:unfoldedCE}).
\end{Rmk}

\begin{figure}[ht]
\begin{tikzpicture}[xscale=0.55]
\draw plot[smooth cycle]coordinates {(-2,-1) (0,-1.5) (2,-1)(2,1) (0,1.5)(-2,1)};
\draw[thick](-2,-1)--(1,-1);
\draw[thick](-2,1)--(-1,1);
\draw[dashed,thick](-2,1)--(-3,1);
\draw[dashed,thick](-2,-1)--(-3,-1);
\draw(0,1.75) node{$v_i\mathcal{G}[B]=v_j\mathcal{G}[B]$};
\filldraw(-1,1)circle(1pt);
\filldraw(1,-1)circle(1pt);
\draw[above](-1,0.9)node{$v_i$};
\draw[below](1,-0.95)node{$v_j$};
\draw[-latex,blue,thick] (-1,1)--(0,0);
\draw[blue, thick] (0,0)--(1,-1);
\draw[blue,right](-0.1,0)node{$q$};
\draw[red,thick] (-1,1)--(2.5,1);
\draw[red,thick] (1,-1)--(2.5,-1);
\draw[-latex,red,thick] (2.5,1) to [out=0,in=0] (2.5,-1);
\draw[red](2.75,0)node{$p$};
\draw[above](-2.5,1)node{$\mathcal{K}$};
\end{tikzpicture}
\begin{tikzpicture}[xscale=0.5]
\draw(-4,1)node{};
\draw plot[smooth cycle]coordinates {(-2,-1) (0,-1.5) (2,-1)(2,1) (0,1.5)(-2,1)};
\draw plot[smooth cycle]coordinates {(5,-1) (7,-1.5) (9,-1)(9,1) (7,1.5)(5,1)};
\draw[thick](-2,1)--(-1,1);
\draw[dashed,thick](-2,1)--(-3,1);
\draw(0,1.75) node{$v_i\mathcal{G}[B]$};
\draw(7,1.75) node {$v_j\mathcal{G}[B]$};
\filldraw(-1,1)circle(1pt);
\filldraw(1,-1)circle(1pt);
\draw[above](-1,0.9)node{$v_i$};
\draw[below](6,-0.95)node{$v_j$};
\draw[-latex,blue,thick] (-1,1)--(0,0);
\draw[blue,thick](0,0)--(1,-1);
\draw[blue,right](0,0)node{$q$};
\draw[red,thick] (-1,1)--(2.5,1);
\draw[red,thick] (4.5,-1)--(6,-1);
\draw[-latex,red,thick]  (2.5,1) to [out=0,in=180] (4.5,-1);
\filldraw(6,-1)circle(1pt);
\draw[thick](6,-1)--(9,-1);
\draw[dashed,thick](9,-1)--(10,-1);
\draw[red](3.9,0)node{$p$};
\draw[above](-2.5,1)node{$\mathcal{K}$};
\filldraw(8,1)circle(1pt);
\draw[-latex,blue,thick](8,1)--(7,0);
\draw[blue,thick] (7,0)--(6,-1);
\draw[blue,right](6.2,0)node{$q$};
\end{tikzpicture}
\caption{Parts of $\mathcal{K}\cup\bigcup_{t=1}^{k} v_t\mathcal{G}[B]\subseteq\mathcal{G}[A]$ and of $\mathsf{CE}(G,\mathcal{K};B)$}\label{fig:unfoldedCE}
\end{figure}

Let $\mK$ be a connected $A$-graph admissible for
$\pss{A}$-coset extension,
let $B\subsetneq A$ and let $\mB = v \mK[B] \subseteq \mK$ be the $B$-component of some vertex $v$ in $\mK$.
By construction of $\sfCE(G,\mK;\bP_A)$,
\[v\in \mB\subseteq v\mG[B]\subseteq \sfCE(G,\mK;B) \subseteq
  \sfCE(G,\mK;\bP_A).\]We are able to refine this chain as follows:
$\mB$ is itself admissible for
$\pss{B}$-coset extension
and hence $\sfCE(G,\mB;\bP_B)$ is well defined. Admissibility of $\mK$
(Definition~\ref{def:admissible}) implies that in this case the
morphism $\iota_{\bP_B}\colon \sfCE(G,\mB;\bP_B)\to \mG[B]$ of
Proposition~\ref{prop:morphismCEtoGA} is injective. {Indeed, $\iota_{\bP_B}$ is injective on the skeleton $\mB$, and on every constituent coset $v\mG[C]$ for any $C\subsetneq B$ and any vertex $v$. If there were vertices $x\ne y$ such that $\iota_{\bP_B}(x)=\iota_{\bP_B}(y)$, then $x$ and $y$ would belong to two distinct constituent cosets $x\in v_1\mG[B_1]$ and $y\in v_2\mG[B_2]$ ($B_1, B_2\subsetneq B$, possibly $B_1=B_2$) so that $x$ and $y$ would coincide as elements of $v_1\mG[B]=v_2\mG[B]$. But this is excluded by Definition~\ref{def:admissible}. Hence}
we get the following. 
\begin{Lemma}\label{lem:CEBsubsetCEA}
Let $\mK$ be a subgraph of $\mG[A]$ which is admissible for
$\pss{A}$-coset extension (in particular $G[A]$ is retractable,
cf.\ Definition~\ref{def:admissible} and also Proviso~\ref{prov:Aretractable}). 
Let $B\subsetneq A$ with $|B|\ge 2$; then every
  $B$-component $\mB$ of $\mK$ is admissible for
$\pss{B}$-coset extension
and the morphism $\iota_{\bP_B}\colon \sfCE(G,\mB;\bP_B)\to \mG[B]$ is injective. In particular, for any vertex $v\in \mB$,
\[v\in \mB\subseteq \sfCE(G,\mB;\bP_B) \subseteq v\mG[B]\subseteq
  \sfCE(G,\mK;B)\subseteq  \sfCE(G,\mK;\bP_A).\]
\end{Lemma}
Another consequence 
concerns connectivity in the graph $\mK$; it will be of significant use later. In terms of~\cite{otto3} this
means that a graph $\mK$  which is admissible for
$\pss{A}$-coset extension
is $2$-acyclic.
\begin{Lemma}\label{lem:K2acyclic}
Suppose that the graph $\mathcal{K}\subseteq \mG$ is admissible for
$\pss{A}$-coset extension.
Then, for any $B,C\subsetneq A$,
the intersection $\mathcal{B}\cap\mathcal{C}$ of any $B$-component
$\mathcal{B}$ and any $C$-component $\mathcal{C}$ of $\mK$ is
connected and hence is a $(B\cap C)$-component.
\end{Lemma}

\begin{proof} Suppose that $B\ne C$ and let $u,v$ be vertices of
  $\mathcal{B}\cap \mathcal{C}$ and assume that they belong to
  distinct components of $\mathcal{B}\cap \mathcal{C}$. Admissibility
  of $\mathcal{K}$ (by taking $B_1=B\cap C=B_2$) implies that the
  cosets $u\mathcal{G}[B\cap C]$ and $v\mathcal{G}[B\cap C]$ are
  disjoint (that is, distinct), and both cosets are contained in
  $u\mathcal{G}[B]=v\mathcal{G}[B]$ as well as
  $u\mathcal{G}[C]=v\mathcal{G}[C]$. Consider the graph morphism
  ${\iota_{\mathbb{P}_A}}\colon
  \mathsf{CE}(G,\mathcal{K};\mathbb{P}_A)\to \mathcal{G}[A]$. It maps
  the cosets $u\mathcal{G}[B]$ as well as $v\mathcal{G}[C]$
  injectively to the corresponding coset subgraphs of
  $\mathcal{G}[A]$. Since $u\mathcal{G}[B\cap C]$ and
  $v\mathcal{G}[B\cap C]$ are disjoint, it follows that the
  intersection of the cosets $u\mathcal{G}[B]$ and $v\mathcal{G}[C]$
  (in $\mathcal{G}[A]$) is disconnected as it has at least the two
  components $u\mathcal{G}[B\cap C]$ and $v\mathcal{G}[B\cap
  C]$; this, however, contradicts the assumption that $G[A]$ is retractable.
\end{proof}

\subsubsection{Augmented coset extensions}
Similarly to
augmented
clusters we require augmented coset extensions. Again fix an $E$-group $G$,
let $A\subseteq E$ with $|A|\ge 2$ and assume that $G[A]$ is
retractable, according to Proviso~\ref{prov:Aretractable}. 
Let $\mathcal{K}\subseteq \mathcal{G}[A]$ be admissible for
$\pss{A}$-coset extension.
Recall that the full
$\pss{A}$-coset extension
$\mathsf{CE}(G,\mathcal{K};\mathbb{P}_A)$ can be seen as the union
$\bigcup_{B\subsetneq A}\mathsf{CE}(G,\mathcal{K},B)$ where for
$B,C\subsetneq
A$, \[\mathsf{CE}(G,\mathcal{K};B)\cap\mathsf{CE}(G,\mathcal{K};C)=\mathsf{CE}(G,\mathcal{K},
  B\cap C).\]  
Every vertex $x$ of $\mathsf{CE}(G,\mathcal{K};\mathbb{P}_A)$ is
sitting in some $\mathsf{CE}(G,\mathcal{K};B)$, and, inside
$\mathsf{CE}(G,\mathcal{K};B)$ in a unique \textsl{constituent coset}
$v\mathcal{G}[B]$ with $v\in \mathcal{K}$. The vertex $v$ is not
unique, but unique is its $B$-component $v\mathcal{K}[B]$. In this
situation we say that the pair $(B,v)$ \emph{supports} the vertex $x$
or \emph{provides support} for the vertex $x$ in
$\mathsf{CE}(G,\mathcal{K};\mathbb{P}_A)$; the \emph{size} of this
support is $|B|$. This actually means that the skeleton $\mathcal{K}$
may be accessed from the vertex $x$ by a $B$-path whose terminal
vertex is $v$. We say that $(B,v)$ provides \emph{unique minimal
  support} if, whenever $(C,w)$ provides support for $x$ then
$B\subseteq C$ and $v\mathcal{K}[B]\subseteq w\mathcal{K}[C]$. Now let
$\mathcal{J}$ be a subgraph of
$\mathsf{CE}(G,\mathcal{K};\mathbb{P}_A)$; for a set $B\subsetneq A$
and a vertex $v\in \mathcal{K}$ we say that $(B,v)$ 
\emph{provides unique minimal support for $\mathcal{J}$},
or that \emph{$\mathcal{J}$ has unique minimal support through $(B,v)$},
if $(B,v)$ supports 
some vertex $x$ of $\mathcal{J}$, and if some pair $(C,w)$ supports
any vertex $y$ of $\mathcal{J}$ then $B\subseteq C$ and $v\mathcal{K}[B]\subseteq w\mathcal{K}[C]$. In this case we say that the unique minimal support of $\mathcal{J}$ is \emph{attained} 
at the vertex $x$.
Notice that the condition $v\mK[B]\subseteq w\mK[C]$ implies 
the inclusion $v\mathcal{G}[B]\subseteq w\mathcal{G}[C]=v\mG[C]$ 
for the constituent cosets involved.
It follows from~\eqref{eq:boolean extended}  
that every one-vertex subgraph of $\mathsf{CE}(G,\mathcal{K};\mathbb{P}_A)$ has unique minimal support. 

We come to a crucial property, which the full
$\pss{A}$-coset extension
of a graph $\mathcal{K}$ may or may not have.
\begin{Def}[cluster property]
\label{def:cluster property}  The full coset extension $\mathsf{CE}(G,\mathcal{K};\mathbb{P}_A)$ has the \emph{cluster property} if, for every $B\subsetneq A$ the following hold:
\begin{enumerate}
\item every $B$-component $\mathcal{B}$ of $\mathsf{CE}(G,\mathcal{K};\mathbb{P}_A)$ which has empty intersection with the skeleton $\mathcal{K}$ is a $B$-cluster or a full $B$-coset;
\item every $\mathcal{B}$ of (1) has unique minimal support which is attained at some vertex $x$ of the core of $\mathcal{B}$ (if $\mB$ is a cluster).
\end{enumerate}
\end{Def}
Note that minimal support will typically not be attained at all core vertices.
We first show that the cluster property implies that components of the coset extension intersect nicely, that is, the coset extension is $2$-acyclic in 
terms of~\cite{otto3}.
\begin{Prop}\label{prop:CE=2acyclic} Suppose that $\mathcal{K}\subseteq \mathcal{G}[A]$ is admissible for
$\pss{A}$-coset extension
and that the full $\pss{A}$-coset extension
$\mathsf{CE}(G,\mathcal{K};\mathbb{P}_A)$ has the cluster property. Then, for all pairs $B,C\subsetneq A$ the intersection $\mathcal{B}\cap\mathcal{C}$ of any $B$-component $\mathcal{B}$ and any $C$-component $\mathcal{C}$ is connected and hence is a $(B\cap C)$-component of $\mathsf{CE}(G,\mathcal{K};\mathbb{P}_A)$.
\end{Prop} 
\begin{proof} We consider several cases and start with the most
difficult one: suppose that both $\mathcal{B}$ and $\mathcal{C}$
have empty intersection with the skeleton $\mathcal{K}$. We need to
show that $\mathcal{B}\cap \mathcal{C}$ is connected. We know that
$\mathcal{B}$ is a $B$-cluster, $\mathcal{C}$ is a $C$-cluster, that
is, $\mathcal{B}\cong \mathsf{CL}(G[B],\{B_1,\dots,B_k\})$ and
$\mathcal{C}\cong\mathsf{CL}(G[C],\{C_1,\dots, C_l\})$ for
$B_i\subsetneq B$ and $C_j\subsetneq C$; it may also happen that
$k=1$ and/or $l=1$ in which case it may happen that $B_1=B$ and/or
$C_1=C$ (that is, $\mB$ and/or $\mC$ is a $B$-coset and/or
$C$-coset) --- the argument for this subcase is similar but simpler.
Let $x$ be a vertex in the core of $\mathcal{B}$, $y$ a vertex in
the core of $\mathcal{C}$, such that the unique minimal support $(M,m)$
of $\mathcal{B}$ is attained at  $x$,
and the unique minimal support $(N,n)$ of $\mathcal{C}$ is attained at $y$.
Then $\mathcal{B}=\bigcup_{i=1}^k x\mathcal{G}[B_i]$ and
$\mathcal{C}=\bigcup_{j=1}^l y\mathcal{G}[C_j]$. Let $u_1\ne u_2$ be
vertices of $\mathcal{B}\cap\mathcal{C}$; we may assume that $u_1\in
x\mathcal{G}[B_1]\cap y\mathcal{G}[C_1]$ and $u_2\in
x\mathcal{G}[B_2]\cap y\mathcal{G}[C_2]$.
The vertices $u_1$ and $u_2$ 
also have unique minimal support
$(F_1,v_1)$ and $(F_2,v_2)$, say. Then $M,N\subseteq F_1,F_2$ and even more holds, namely
\[
  \begin{array}{rl@{\;\subseteq\;}c@{\;=\;}c@{\;=\;}l}
    & m\mathcal{G}[M],n\mathcal{G}[N]  & m\mathcal{G}[F_1]& v_1\mathcal{G}[F_1] & n\mathcal{G}[F_1]
    \\
    \hnt\mbox{and } & m\mathcal{G}[M],n\mathcal{G}[N]  & m\mathcal{G}[F_2] & v_2\mathcal{G}[F_2]& n\mathcal{G}[F_2].
  \end{array}
\]
The equality $m\mG[F_1]=v_1\mG[F_1]$ follows from the fact that $(F_1,v_1)$ provides some support for $\mB$, while $(M,m)$ provides unique minimal support for $\mB$ hence $M\subseteq F_1$ and $m\in m\mG[M]\subseteq v_1\mG[F_1]$; likewise, $(F_1,v_1)$ provides some support for $\mC$ while $(N,n)$ provides unique minimal support for $\mC$, hence $N\subseteq F_1$ and  $n\in n\mG[N]\subseteq v_1\mG[F_1]$ which implies $v_1\mG[F_1]=n\mG[F_1]$. The remaining two equalities are proved in the same fashion.
From \[m\mG[M]\cup n\mG[N]\subseteq v_1\mG[F_1]\cap v_2\mG[F_2]\] we get
$v_1\mathcal{G}[F_1]\cap v_2\mathcal{G}[F_2]\ne
\varnothing$, which by \eqref{eq:freeness extended} implies  $v_1\mathcal{K}[F_1]\cap
v_2\mathcal{K}[F_2]\ne \varnothing$.
By Lemma~\ref{lem:K2acyclic}, this intersection is an $F$-component of
$\mathcal{K}$ for $F=F_1\cap F_2$, that is, 
\[v_1\mK[F_1]\cap v_2\mK[F_2]=m\mK[F]=n\mK[F].\] 
From the definition of the full coset extension $\sfCE(G,\mK;\bP_A)$
and \eqref{eq:boolean extended} it follows that the intersection
$v_1\mG[F_1]\cap v_2\mG[F_2]$ itself is connected (it is isomorphic
with $m\mG[F]=n\mG[F]$). So the subgraph of
$\mathsf{CE}(G,\mathcal{K};\mathbb{P}_A)$ formed by the 
union $v_1\mathcal{G}[F_1]\cup v_2\mathcal{G}[F_2]$ is isomorphic with
the cluster $\mathsf{CL}(G[A],\{F_1,F_2\})$, see Figure~\ref{fig:CE2acyclic1}.

Moreover, the cosets $x\mG[B_1]$ and $v_1\mG[F_1]$ both are contained in some constituent coset $w\mG[D]$. Indeed, $x\mG[B_1]$ arises as the intersection of the $B$-component $\mB$ with some
 constituent coset, say $w\mG[D]$, for some vertex $w\in \mK$ and $D\subsetneq A$. Then $(D,w)$ supports $u_1$, whence $F_1\subseteq D$ and $v_1\mG[F_1]\subseteq v_1\mG[D]=w\mG[D]$. Since $G[D]$ is retractable the intersection $x\mG[B_1]\cap v_1\mG[F_1]$ 
is  connected.
 The same holds for the intersections 
\[x\mG[B_2]\cap v_2\mG[F_2],\ y\mG[C_1]\cap v_1\mG[F_1] \mbox{ and } y\mG[C_2]\cap v_2\mG[F_2].\] 

Setting $B':=(B_1\cap F_1)\cup(B_2\cap F_2)$ and $C':=(C_1\cap F_1)\cup(C_2\cap F_2)$ we see that $u_1$ and $u_2$ belong to the same $B'$- as well as $C'$-component of 
the cluster $v_1\mG[F_1]\cup v_2\mG[F_2]$, the intersection of which is a $(B'\cap C')$-component of that cluster, by Corollary~\ref{cor:BcapC for clusters}. Consequently, $u_1$ and $u_2$ are in the same $(B'\cap C')$-component of 
$v_1\mG[F_1]\cup v_2\mG[F_2]$ and hence in the same $(B\cap C)$-component of $\mathsf{CE}(G,\mathcal{K};\mathbb{P}_A)$; the configuration is depicted in Figure~\ref{fig:CE2acyclic1}.
\begin{figure}[ht]
\pgfdeclarelayer{background layer}
\pgfdeclarelayer{foreground layer}
\pgfsetlayers{background layer,main,foreground layer}
\begin{tikzpicture}[xscale=0.6,yscale=0.4]
\begin{pgfonlayer}{background layer}
 \filldraw[ gray, nearly transparent] plot coordinates {(-1,4.33)(-1,8.3)(1,8.3)(1,4.33)(1,0)(1,-4.33)(1,-8.3)(-1,-8.3)(-1,-4.33)(-1,0)(-1,4.33)}; 
  \end{pgfonlayer}
 \begin{pgfonlayer}{main}
\filldraw[white] plot [smooth cycle] coordinates {(-3,0)(-1.5,-0.9)(0.3,-0.3)(0.3,2.5)(-2,5)(-5.5,6)};
\filldraw[white] plot [smooth cycle] coordinates {(3,0)(1.5,-0.9)(-0.3,-0.3)(-0.3,2.5)(2,5)(5.5,6)};
\filldraw[white] plot [smooth cycle] coordinates {(-3,0)(-1.5,0.9)(0.3,0.3)(0.3,-2.5)(-2,-5)(-5.5,-6)};
\filldraw[white] plot [smooth cycle] coordinates {(3,0)(1.5,0.9)(-0.3,0.3)(-0.3,-2.5)(2,-5)(5.5,-6)};
 \end{pgfonlayer}
\draw[gray, dashed, opacity=.4]plot coordinates{(1,4.25)(1,-4.25)};
\draw[gray, dashed, opacity=.4]plot coordinates{(-1,4.25)(-1,-4.25)};
\filldraw (-6,0) circle (2.5pt);
\filldraw(-2.5,0) circle(2.5pt);
\filldraw(2.5,0) circle (2.5pt);
\filldraw(6,0) circle (2.5pt);
\filldraw(0,2) circle (2.5pt);
\filldraw(0,6) circle (2.5pt);
\filldraw(0,-2) circle (2.5pt);
\filldraw(0,-6) circle (2.5pt);
\draw[decorate, decoration={snake}](-6,0)--(-2.5,0);
\draw[decorate, decoration={snake}](2.5,0)--(6,0);
\draw[decorate, decoration={snake}](0,-2)--(0,-6);
\draw[decorate, decoration={snake}](0,2)--(0,6);
\draw plot [smooth cycle] coordinates {(0.6,9.5)(2,0) (0.6,-9.5)(-7,0)};
\draw plot [smooth cycle] coordinates {(-0.6,9.5)(-2,0) (-0.6,-9.5)(7,0)};
\draw[red] plot [smooth cycle] coordinates {(-3,0)(-1.5,-0.9)(0.3,-0.3)(0.3,2.5)(-2,5)(-5.5,6)};
\draw[red] plot [smooth cycle] coordinates {(3,0)(1.5,-0.9)(-0.3,-0.3)(-0.3,2.5)(2,5)(5.5,6)};
\draw[purple] plot [smooth cycle] coordinates {(-3,0)(-1.5,0.9)(0.3,0.3)(0.3,-2.5)(-2,-5)(-5.5,-6)};
\draw[purple] plot [smooth cycle] coordinates {(3,0)(1.5,0.9)(-0.3,0.3)(-0.3,-2.5)(2,-5)(5.5,-6)};
\draw[red](-3,3)node{$x\mathcal{G}[B_1]$};
\draw[red](3,3) node{$x\mathcal{G}[B_2]$};
\draw[purple](-3,-3)node{$y\mathcal{G}[C_1]$};
\draw[purple](3,-3)node{$y\mathcal{G}[C_2]$};
\draw[gray](0,6.5) node[above] {$m\mathcal{K}[F]$};
\draw(0,5.85) node[above]{$m$};
\draw(0,-5.85) node [below]{$n$};
\draw(0,2) node [below]{$x$};
\draw(0,-2) node [above]{$y$};
\draw(-6,0) node [below]{$v_1$};
\draw(-2.5,0) node [below] {$u_1$};
\draw(2.5,0) node [below]{$u_2$};
\draw(6,0) node [below] {$v_2$};
\draw (0,4) node [above right] {$M$};
\draw(0,-4) node [below right]{$N$};
\draw (-4.25,0) node [above]{$F_1$};
\draw (4.25,0) node [above] {$F_2$};
\draw (-7,0) node[left] {$v_1\mathcal{G}[F_1]$};
\draw (7,0) node[right] {$v_2\mathcal{G}[F_2]$};
\end{tikzpicture}
\caption{Configuration as in the proof of Proposition~\ref{prop:CE=2acyclic} (general case)}\label{fig:CE2acyclic1}
\end{figure}

Next we consider the case when $\mC$ has empty intersection with  the skeleton $\mK$ (as in the previous case), but $\mathcal{B}$ has not. Then $\mC\cong \sfCL(G[C],\{C_1,\dots,C_l\})$ and $\mB=v\mG[B]$ for some vertex $v\in \mK$. We let $u_1\ne u_2$ be vertices in $\mB\cap \mC$, and we may assume that $u_1\in\mC_1:=y\mG[C_1]$ and $u_2\in  \mC_2:=y\mG[C_2]$ (as in the previous case), where $y$ is a vertex in the core of $\mC$ which attains minimal support of $\mC$.
\begin{figure}[ht]
\begin{tikzpicture}[xscale=0.5,yscale=0.4]
\filldraw (-6,0) circle (2.5pt);
\filldraw(6,0) circle (2.5pt);
\filldraw(0,6) circle (2.5pt);
\filldraw(0,0) circle (2.5pt);
\draw plot [smooth cycle] coordinates {(0,8)(-5,5)(-8,0)(-4,0)(0,4)};
\draw plot [smooth cycle] coordinates {(0,8)(5,5)(8,0)(4,0)(0,4)};
\draw[decorate, decoration={snake}](0,6)--(0,0);
\draw plot [smooth cycle] coordinates {(0,7)(7,0)(0,-1.5)(-7,0)};
\draw(0,6) node[above]{$y$};
\draw(-6,0) node [above]{$u_1$};
\draw(6,0) node [above] {$u_2$};
\draw(0,0) node[below]{$v$};has the cluster property
\draw(-3.7,4.1)node [above left]{$\mathcal{C}_1$};
\draw(3.8,4.1)node [above right]{$\mathcal{C}_2$};
\draw(2,0) node {$v\mathcal{G}[B]$};
\end{tikzpicture}
\caption{Configuration as in the proof of Proposition~\ref{prop:CE=2acyclic} (special case)}\label{fig:CE2acyclic2}
\end{figure}
In this case $(B,v)$ supports $u_1$ as well as $u_2$ and therefore also $y$,
so that $u_1,y,u_2\in \mB=v\mG[B]$, see
Figure~\ref{fig:CE2acyclic2}. For the same reason as in the previous
case, the intersections $y\mG[C_1]\cap v\mG[B]$ and $y\mG[C_2]\cap
v\mG[B]$ both are connected. Hence there is a $(B\cap C)$-path
$u_1\longrightarrow y$ and also a $(B\cap C)$-path $y\longrightarrow
u_2$, and altogether there is a $(B\cap C)$-path $u_1\longrightarrow
u_2$.

Finally, the case when $\mB$ as well as $\mC$ have non-empty intersection with the skeleton $\mK$ is obvious, since in this case $\mB\cap \mC$ is a $(B\cap C)$-coset. 
\end{proof}

We are led to a further construction. Let $\mK$ be admissible for
$\pss{A}$-coset extension
and suppose that the full
$\pss{A}$-coset extension
$\sfCE(G,\mK;\mathbb{P}_A)$
has the cluster property. For a vertex $v\in \sfCE(G,\mK;\mathbb{P}_A)$ and some $B\subsetneq A$ the $B$-component $\mB$ of $v$ is either a $B$-coset $v\mG[B]$ (in this case, $\mB$ may or may not intersect with the skeleton $\mK$) or a proper $B$-cluster (in which case it does not intersect with the skeleton $\mK$). In any case, $\mB$ embeds into $\mG[B]$ via some graph monomorphism $\iota\colon \mB\hookrightarrow \mG[B]$ (which is unique if one additionally assumes that $\iota(v)=1$). We define the 
\emph{$B$-augmentation at $v$ of $\sfCE(G,\mK;\bP_A)$} by
\[\sfCE(G,\mK;\bP_A)\circv
  \mG[B]:=\left.\sfCE(G,\mK;\bP_A)\sqcup\mG[B]\right./\;\Omega\]
where $\Omega$ is the congruence whose 
non-singleton congruence classes are the two-element sets
$\{x,\iota(x)\}$ for $x\in \mB$.
We note that $\sfCE(G,\mK;\bP_A)\circv\mG[B]$ can be written as the
union $$\sfCE(G,\mK;\bP_A)\cup v\mG[B]$$ of
its two subgraphs $\sfCE(G,\mK;\bP_A)$ and $v\mG[B]$ whose intersection is just the $B$-component $\mB$ of $v$ in $\sfCE(G,\mK;\bP_A)$.
\begin{Prop}\label{prop:2acyclicexpandedAext} Let $B,C\subsetneq A$
and $\mK$ be admissible for
$\pss{A}$-coset extension
and such that the full
  $\pss{A}$-coset extension
  $\sfCE(G,\mK;\bP_A)$
enjoys the cluster property. Then every $C$-component of
any $B$-augmented  
full coset extension $\sfCE(G,\mK;\bP_A)\circv \mG[B]$ is either a
$C$-coset, a $B\cap C$-coset, a $C$-cluster or a $(B\cap C)$-augmented $C$-cluster.
\end{Prop}

\begin{proof}
Let $\mC$ be a $C$-component of $\sfCE(G,\mK;\bP_A)\circv\mG[B]$.
If $\mC\subseteq \sfCE(G,\mK;\bP_A)$ or $\mC\subseteq v\mG[B]$ we are done: $\mC$ happens to be a $C$-coset or a $B\cap C$-coset or a $C$-cluster. Let us assume that $\mC$ is contained neither in $\sfCE(G,\mK;\bP_A)$ nor in $v\mG[B]$. We have
\[\mC=\underbrace{(\sfCE(G,\mK;\bP_A)\cap\mC)}_{\mC_1}\cup
 \underbrace{(v\mG[B]\cap\mC)}_{\mC_2}\]
  and $\mC_1$ is a proper $C$-cluster (if it were a $C$-coset it would coincide with $\mC$, which would  be contained in $\sfCE(G,\mK;\bP_A)$). Let $\mB_v$ be the $B$-component of $v$ in $\sfCE(G,\mK;\bP_A)$. 
Our assumption implies that $\mC\cap\mB_v\ne \varnothing$.  Let $w$ be a vertex of $\mC\cap\mB_v$. By Proposition~\ref{prop:CE=2acyclic},    $\mC\cap \mB_v=\mC_1\cap \mB_v$ is the $(B\cap C)$-component of $w$ in $\sfCE(G,\mK;\bP_A)$, which is a $(B\cap C)$-cluster or a $(B\cap C)$-coset. Moreover,
\[\mC_2=\mC\cap v\mG[B]=\mC\cap w\mG[B]=w\mG[B\cap C].\] 
If  $\mC_1\cap \mB_v$ were a $(B\cap C)$-coset, then it would coincide
with $w\mG[B\cap C]$ and again $\mC\subseteq
\sfCE(G,\mK;\bP_A)$. Hence, under our assumption, $\mC_1\cap\mB_v$ is
indeed a proper $(B\cap C)$-cluster.
So we see that $\mC=\mC_1\cup w\mG[B\cap C]$ and 
\[\mC_1\cap w\mG[B\cap C]=\sfCE(G,\mK;\bP_A)\cap w\mG[B\cap C]\] 
is the $(B\cap C)$-component of $w$ in $\sfCE(G,\mK;\bP_A)$.
Altogether this just means that 
$\mC=\mC_1\circw \mG[B\cap C]$, that is, $\mC$ is the $(B\cap C)$-augmentation of the $C$-cluster $\mC_1$ at $w$.
\end{proof}

\section{Two crucial inductive procedures}
\label{sec:2results}
In this section we formulate and prove two important technical results. They will be
essential to set up the inductive procedure to gain the series
\eqref{eq:series of G}. In order to do so, we need another crucial
definition {(Definition~\ref{def:bridge free} below)}. Assume, as above, that $|A|\ge 2$, that $G[A]$ is
retractable and  that $\mK\subseteq \mG[A]$ is admissible for
$\pss{A}$-coset extension.

{\begin{Def}[embedded coset extension]\label{def:embedded}
  The full coset extension \[\mathsf{CE}(G,\mathcal{K};\mathbb{P}_A)\] is \emph{embedded}
if the morphism ${\iota_{\mathbb{P}_A}}\colon \mathsf{CE}(G,\mathcal{K};\mathbb{P}_A)\to \mathcal{G}[A]$ 
(of Proposition~\ref{prop:morphismCEtoGA})  is an embedding.
\end{Def}
\begin{Def}[bridge freeness]\label{def:bridge free}
The embedded full coset extension \[\mathsf{CE}(G,\mathcal{K};\mathbb{P}_A)\] is \emph{bridge free} in $\mG[A]$ if
for every $B\subsetneq A$, if two vertices $u,v\in \mathsf{CE}(G,\mathcal{K};\mathbb{P}_A)\subseteq \mathcal{G}[A]$ (as per Definition~\ref{def:embedded}) are $B$-connected in $\mathcal{G}[A]$, then they are $B$-connected even in $\mathsf{CE}(G,\mathcal{K};\mathbb{P}_A)$. 
\end{Def}}

  The two above-mentioned technical results will, in fact, be two
  inductive  procedures --- \emph{forward induction} (Theorem~\ref{thm:forward induction}) and \emph{upward induction} (Theorem~\ref{thm:upward induction}). Roughly speaking, forward induction
  guarantees that bridge freeness implies the cluster property --- in
  the same group but with the number of letters being increased by one;
  upward induction, on the other hand, guarantees that the cluster
  property implies bridge freeness --- with respect to the same set of
  letters but for the next group.
  For the construction of the series~\eqref{eq:series of G}, these two
  procedures are applied alternatingly; the essence of the whole procedure is
  as follows (details will be worked out in
  Section~\ref{subsec:HktoGk+1}).
  Suppose we have already defined the $k$-retractable group $G_k$.
  We apply Theorem~\ref{thm:first basic} and produce a
  $(k+1)$-retractable and $k$-stable expansion $H_k$ of $G_k$. Then
  take any connected $A$-subgraph $\mL$ of the Cayley graph $\mH_k$ of
  $H_k$ for a subset $A\subseteq E$ of size $k+1$ and assume that
  $\mL$ is admissible for $\pss{A}$-coset extension (with respect to
  $H_k$). For $B\subsetneq A$, all $B$-components $v\mL[B]$ of $\mL$
  are subgraphs of $\mH_k[B]$ and hence of $\mG_k[B]$, by
  $k$-stability. Assuming inductively that all corresponding coset
  extensions  $\sfCE(G_k,v\mL[B];\mathbb{P}_B)$ are bridge-free, the
  same is true for the corresponding coset extensions
  $\sfCE(H_k,v\mL[B];\mathbb{P}_B)$ with respect to
  $H_k$. \textsl{Forward induction} (Theorem~\ref{thm:forward
    induction})  now implies that the coset extension
  $\sfCE(H_k,\mL;\mathbb{P}_A)$ of the $A$-graph $\mL$ has the cluster
  property. Finally, \textsl{upward induction}
  (Theorem~\ref{thm:upward induction}) implies that for a suitable
  $k$-stable expansion $G_{k+1}$ of $H_k$, any $\mG_{k+1}$-cover
$\hat{\mL}$ of $\mL$
  is admissible for $\pss{A}$-coset extension (with respect to $G_{k+1}$) and
  that the coset extension
$\sfCE(G_{k+1},\hat{\mL};\mathbb{P}_A)$
is bridge-free (for a
  precise definition of \emph{cover} see Definition~\ref{def:cover} below).

The following lemma
is the essential technical step to obtain the inductive procedure
\emph{forward induction} (Theorem~\ref{thm:forward induction}).
For this lemma take into account Lemma~\ref{lem:CEBsubsetCEA}: 
if some subgraph $\mL\subseteq \mH[A]$ of the Cayley graph
of the group $H$ is admissible for $\pss{A}$-coset extension,
        then all its $B$-components $v\mL[B]$, for $B\subsetneq A$,
        are admissible for $\pss{B}$-coset extension and  
       the morphisms of Proposition~\ref{prop:morphismCEtoGA}
       are embeddings $\sfCE(H,v\mL[B];\bP_B)\hookrightarrow v\mH[B]$.
 
\begin{Lemma}\label{lem:forward induction} Let $H$ be
  an $E$-group, $A\subseteq E$, $|A|\ge 3$ and suppose that $H[A]$ is
  retractable. Let $\mL\subseteq \mH[A]$ be a connected $A$-graph
  which is admissible for
$\pss{A}$-coset extension.
Assume that for all $B\subsetneq A$ and every vertex $v\in \mL$,
the full $\pss{B}$-coset extension $\sfCE(H,v\mL[B];\bP_B)$
\begin{enumerate}
\item
is {embedded and} bridge-free in $\mH[B]$, and 
\item
has the cluster property.
\end{enumerate}
Then the full
$\pss{A}$-coset extension
$\sfCE(H,\mL;\bP_A)$ has the cluster property.
\end{Lemma}

\begin{proof}
Let $B\subsetneq A$,  let $\mB$ be a $B$-component of
$\sfCE(H,\mL;\bP_A)$ and suppose that $\mB$ has empty intersection
with the skeleton $\mL$. We first show the following: if $\mB$ is not
fully contained in any one constituent coset of $\sfCE(H,\mL;\bP_A)$,
then the intersection of $\mB$ with \textsl{any} constituent coset is
either empty or
contains a vertex that is supported by fewer than $|A|-1$ elements.
Indeed, let $\mB\cap v_1\mH[A_1]\ne\varnothing$, w.l.o.g.\
$|A_1|=|A|-1$, and assume that $\mB$ is not contained in $v_1\mH[A_1]$.
Then some vertex $s_1\in \mB\cap
v_1\mH[A_1]$ must be connected by an edge $e$ in $\mathcal{B}$
to some vertex $s_2\in (\mB\cap
v_2\mH[A_2])\setminus v_1\mH[A_1]$ in some other
constituent coset $v_2\mH[A_2]$, for some
$A_2\ne A_1$.
Since $s_2\notin v_1\mH[A_1]$, also $e\notin v_1\mH[A_1]$.
Then $e$ belongs to a coset $v_3\mH[A_3]$
(possibly coinciding with $v_2\mH[A_2]$)
with $A_3\ne A_1$. In any case, $s_1,s_2\in v_3\mH[A_3]$ (if a graph
contains an edge then also its initial and terminal vertices). It
follows that $s_1$ is supported by $(A_3,v_3)$, that is, $s_1\in
v_1\mH[A_1]\cap v_3\mH[A_3]=v\mH[A_1\cap A_3]$ for some vertex
$v$, and $|A_1\cap A_3| < |A|-1$. 

Therefore,  if no vertex of $\mB$ has support of size smaller than
$|A|-1$, then $\mB$ is contained in some constituent coset
$v_1\mH[A_1]$ with $|A_1|=|A|-1$, and therefore is a $B\cap
A_1$-coset with minimal support 
$(A_1,v_1)$. 

We are left with the case that $\mB$ admits support of size strictly
smaller than $|A|-1$. 
We collect some constituent cosets $v_1\mH[A_1],\dots,v_n\mH[A_n]$ of $\sfCE(H,\mL;\bP_A)$
for generator sets $A_i \subsetneq A$ of size $|A_i|=|A|-1$ 
such that $\mB\subseteq \bigcup_{i=1}^nv_i\mH[A_i]$ and we assume that the choice of the constituent cosets $v_i\mH[A_i]$ is minimal for $\mB\subseteq \bigcup_{i=1}^nv_i\mH[A_i]$ 
in the sense that $\mB$ is not contained in any union of fewer than $n$ constituent cosets. Then
\[\mB=\mB\cap\left(\bigcup_{i=1}^n v_i\mH[A_i]\right)=\bigcup_{i=1}^n (\mB\cap v_i\mH[A_i])=\bigcup_{i=1}^n\mB_i\]
for $\mB_i=\mB\cap v_i\mH[A_i]$. 
Every $\mB_i$ is a non-empty $B_i$-coset subgraph of $v_i\mH[A_i]$
where $B_i=B\cap A_i$ and all $B_i$ have size at most $|A|-2$. (If for
some $i$, $|B_i|=|A|-1$ then $B_i=A_i$ and $\mB_i=v_i\mH[A_i]$ would
have non-empty intersection with the skeleton $\mL$.) In addition,
every $\mB_i$ has a vertex supported by fewer than $|A_i|=|A|-1$
letters: if $n=1$ this is immediate and if $n>1$ then $\mB$ is not
contained in a single constituent coset,
and the situation is as discussed  at the start of the proof.

We need to verify items (1) and (2) of Definition~\ref{def:cluster
  property}. For $i=1,\dots,n$ denote by $\mA_i$ the $A_i$-component
$v_i\mL[A_i]$ of $v_i$ in $\mL$. By Lemma~\ref{lem:CEBsubsetCEA},
$\mA_i$ is admissible for
$\pss{A_i}$-coset extension
and the full $\pss{A_i}$-coset extension
$\sfCE(H,\mA_i;\bP_{A_i})$
embeds into $v_i\mH[A_i]$ (via the mapping of
Proposition~\ref{prop:morphismCEtoGA}).  Since $\mB_i$
admits vertices supported by
fewer than $|A_i|=|A|-1$ letters, we have that $\mB_i\cap\sfCE(H,\mA_i;\bP_{A_i})\ne\varnothing$ --- once more we take into account  that
\[\sfCE(H,\mA_i;\bP_{A_i})\subseteq v_i\mH[A_i]\subseteq \sfCE(H,\mL;\bP_A).\]
Bridge freeness of $\sfCE(H,\mA_i;\bP_{A_i})$ (assumption (1)) implies
that \[\mB_i\cap\sfCE(H,\mA_i;\bP_{A_i})\] is connected. By assumption
(2) therefore, $\mB_i\cap\sfCE(H,\mA_i;\bP_{A_i})$ has unique minimal
support in $\sfCE(H,\mA_i;\bP_{A_i})$, say $(D_i,u_i)$. But then the
pair $(D_i,u_i)$ also provides unique minimal support of $\mB_i$  in
$\sfCE(H,\mL;\bP_A)$. If $n=1$ we are already done; so let us assume that $n\ge 2$. 
Minimality of $(D_i,u_i)$ implies in particular that any path
connecting $\mB_i$ to $\mL$ in
$\sfCE(H,\mL;\bP_A)$
must use (at least) all labels in $D_i$ and, in case it uses only labels from $D_i$, 
necessarily leads to the $D_i$-component $u_i\mA_i[D_i]$.
So, for every  $i$, there exist vertices $s_i\in \mB_i$, $u_i\in \mA_i$ and a
word $m_i\in \til{D_i}^*$ labelling a path $s_i\longrightarrow u_i$
which runs entirely inside the coset $u_i\mH[D_i]$, which in turn is
contained in $v_i\mH[A_i]=u_i\mH[A_i]$.

Since $\mB=\bigcup_{i=1}^n\mB_i$ is connected, there are $i,j$ such
that $\mB_i\cap\mB_j\ne\varnothing$; after some renumbering we may
assume that $\mB_1\cap \mB_2\ne\varnothing$. Then also
$v_1\mH[A_1]\cap v_2\mH[A_2]\ne \varnothing$;
from (\ref{eq:freeness extended}) we get $\mA_1\cap\mA_2\ne\varnothing$ and by  Lemma~\ref{lem:K2acyclic}, $\mA_1\cap\mA_2$ 
is an $(A_1\cap A_2)$-component
of $\mL$, say $v\mL[A_1\cap A_2]$ for some $v\in \mA_1\cap \mA_2$. 
From~\eqref{eq:boolean extended} it follows that
$v_1\mH[A_1]\cap
v_2\mH[A_2]=v\mH[A_1\cap A_2]$. The intersection
$\mB_1\cap\mB_2$ is a $(B\cap A_1\cap A_2)$-coset in $v\mH[A_1\cap
A_2]$. Similarly as for $\mB_1$ one argues that $\mB_1\cap \mB_2$ has
unique minimal support in $\sfCE(H,\mA_1\cap\mA_2;\bP_{A_1\cap A_2})$,
$(D,u)$ say, which (as for $\mB_1$)  provides unique minimal support
of $\mB_1\cap \mB_2$ in $\sfCE(H,\mL;\bP_A)$. Let $s\in \mB_1\cap
\mB_2$ be a vertex which attains the support $(D,u)$. So far, the
situation is depicted as in Figure~\ref{fig:forward1}.
\begin{figure}[ht]
\begin{tikzpicture}[scale=0.5]
\draw (-7,-2)--(7,-2);
\draw (7.2,-2.1) node [above left]{$\mathcal{L}$};
\filldraw (-4,-2) circle (2.5pt);
\draw (-4,-2) node[below]{$u_1$};
\draw plot [smooth cycle] coordinates {(-5,-3) (1,-3) (1,3) (-5,3)};
\draw plot [smooth cycle] coordinates {(-1,-3) (5,-3) (5,3) (-1,3)};
\filldraw(-4,0) circle(2.5pt);
\draw(-3.9,0)node[left]{$s_1$};
\filldraw(0,-2)circle(2.5pt);
\draw(0,-2)node[below]{$u$};
\filldraw(0,2)circle(2.5pt);
\draw(-0.1,2)node[right]{$s$};
\draw[-latex][thick] (-4,0)--(-0.1,2);
\draw[-latex][thick](-4,0)--(-4,-1.9);
\draw[-latex][thick] (-4,-2)--(-0.1,-2);
\draw[-latex][thick](0,-2)--(0,1.9);
\draw(-5,-3)node[left]{$v_1\mathcal{H}[A_1]$};
\draw(5,-3)node[right]{$v_2\mathcal{H}[A_2]$};
\draw plot [ smooth cycle] coordinates {(-4.5,-1.5) (0.5,0.5) (0,2.7)(-4.5,2.5)};
\draw plot [ smooth cycle] coordinates {(4.5,-1) (0,0.5) (0,2.5)(4.5,2.5)};
\draw(-3,2.3)node{$\mathcal{B}_1$};
\draw(3.3,2.3)node{$\mathcal{B}_2$};
\draw(-0.1,-0.9)node[right]{$m$};
\draw(-4.1,-0.9)node[right]{$m_1$};
\draw(-2,-1.9)node[below]{$k$};
\draw(-2.2,1.3)node{$p$};
\end{tikzpicture}
\caption{Configuration as in the proof of Lemma~\ref{lem:forward induction} (generic)}\label{fig:forward1}
\end{figure}
We note that $D\subseteq A_1\cap A_2$ and so \[u\mH[D]\subseteq
  u\mH[A_1\cap A_2]= v_1\mH[A_1]\cap v_2\mH[A_1].\]  Since $(D,u)$ is
some support of $\mB_1$, we have $D_1\subseteq D$ and
$u_1\mH[D_1]\subseteq u\mH[D]$. Hence there is a $D$-path
$u_1\longrightarrow u$ labelled $k$, say, which runs inside $\mA_1$,
and a $D$-path $u\longrightarrow s$ labelled $m$. Altogether, there is
a $D$-path $s_1\longrightarrow s$ labelled $m_1km$ (this path runs
entirely in $v_1\mH[A_1]$). 
Since $s_1,s\in \mB_1$, there is also a $B_1$-path $s_1\longrightarrow
s$ where $B_1=B\cap A_1$, labelled $p$, say. Again, this path runs
inside $v_1\mH[A_1]$. Since $H[A_1]$ is retractable, we have
$[p]_{H[A_1]}=[p']_{H[A_1]}$ where $p'$ is the word obtained from $p$
by deletion of all letters not in $D$. Hence there is a $D$-path
$s_1\longrightarrow s$ which runs entirely in $\mB_1\cap u\mH[D]$ and,
in particular, $s_1\in u\mH[D]\subseteq u\mH[A_1\cap A_2]=
v_1\mH[A_1]\cap v_2\mH[A_2]$ so that $s_1\in \mB_1\cap\mB_2$. Since
$(D_1,u_1)$ supports $s_1$ and therefore also $\mB_1\cap\mB_2$, 
it follows that $D\subseteq D_1$ and therefore $D=D_1$ as the converse inclusion has been already shown. In particular, $(D,u)$ provides unique minimal support of $\mB_1$ which is attained at $s_1\in \mB_1\cap \mB_2$.
So the configuration in Figure~\ref{fig:forward1} really looks as depicted in Figure~\ref{fig:forward2}.
\begin{figure}[ht]
\begin{tikzpicture}[scale=0.5]
\draw (-7,-2)--(7,-2);
\draw (7.2,-2.1) node [above left]{$\mathcal{L}$};
\filldraw (-3,-2) circle (2.5pt);
\draw (-3,-2) node[below]{$u_1$};
\draw plot [smooth cycle] coordinates {(-5,-3) (1,-3) (1,3) (-5,3)};
\draw plot [smooth cycle] coordinates {(-3,-3) (5,-3) (5,3) (-3,3)};
\filldraw(-2,1) circle(2.5pt);
\draw(-2,1)node[above]{$s_1$};
\filldraw(0,-2)circle(2.5pt);
\draw(0,-2)node[below]{$u$};
\filldraw(0,2)circle(2.5pt);
\draw(-0.1,2)node[right]{$s$};
\draw[-latex][thick] (-2,1)--(-0.1,2);
\draw[-latex][thick](-2,1)--(-3,-1.9);
\draw[-latex][thick] (-3,-2)--(-0.1,-2);
\draw[-latex][thick](0,-2)--(0,1.9);
\draw(-5,-3)node[left]{$v_1\mathcal{H}[A_1]$};
\draw(5,-3)node[right]{$v_2\mathcal{H}[A_2]$};
\draw plot [ smooth cycle] coordinates {(-4.5,-1.5) (0.5,0.5) (0,2.7)(-4.5,2.5)};
\draw plot [ smooth cycle] coordinates {(4.5,-1) (-2,0.5) (-2,2.5)(4.5,2.5)};
\draw(-3,2.3)node{$\mathcal{B}_1$};
\draw(3.3,2.3)node{$\mathcal{B}_2$};
\draw(-0.1,-0.9)node[right]{$m$};
\draw(-3.9,-0.4)node[right]{$m_1$};
\draw(-1.4,-1.9)node[below]{$k$};
\draw(-0.9,1.1)node{$p$};
\end{tikzpicture}
\caption{Configuration as in the proof of Lemma~\ref{lem:forward induction} (updated)}\label{fig:forward2}
\end{figure}
By the same reasoning we obtain that $s_2\in \mB_1\cap \mB_2$ and $D_2=D$. Altogether, $s_1,s_2\in \mB_1\cap \mB_2$ and $(D,u)$ provides unique minimal support of $\mB_1$ as well as $\mB_2$, attained at $s_1$ as well as $s_2$.
Now we continue by induction. Let $2\le k<n$ and suppose,
subject to some renumbering of the cosets $\mB_i$, we have already shown that $s_1,\dots,s_k\in \mB_1\cap\cdots\cap\mB_k$ and  all these $\mB_i$ have unique minimal support $(D,u)$ attained at all these $s_i$. Again there are $j\in \{1,\dots,k\}$ and $i\in\{k+1,\dots,n\}$ such that $\mB_j\cap\mB_i\ne\varnothing$ and after some renumbering we may assume that $j=k$ and $i=k+1$. Then, as for the case $k=1$, $s_k,s_{k+1}\in\mB_k\cap\mB_{k+1}$ and the unique minimal support of $\mB_k\cap \mB_{k+1}$ is $(D,u)$. Again, $s_k\in \mB_1\cap\cdots\cap\mB_k\cap\mB_{k+1}$ and so $\mB_j\cap \mB_{k+1}\ne\varnothing$ for all $j\le k$, therefore $s_j,s_{k+1}\in \mB_j\cap\mB_{k+1}$ and hence $s_1,\dots,s_{k+1}\in \mB_1\cap\cdots\cap\mB_{k+1}$ and $(D,u)$ provides unique minimal support for $\mB_{k+1}$ attained at $s_{k+1}$.
So $s_1,\dots,s_n\in \mB_1\cap\cdots\cap\mB_n$ and $\bigcup_{i=1}^n \mB_i$ has unique minimal support $(D,u)$  attained at some vertices of $\bigcap_{i=1}^n\mB_i$. 

It remains to argue that $\mB$ is indeed a $B$-cluster.
From $\bigcap_{i=1}^n\mB_i\ne \varnothing$ we have in particular that $\bigcap_{i=1}^n v_i\mH[A_i]\ne \varnothing$. 
By induction and using \eqref{eq:freeness extended} and Lemma~\ref{lem:K2acyclic} we can show that $\bigcap_{i=1}^n v_i\mH[A_i]=w\mH[C]$ for some vertex $w\in \mL$ and $C=\bigcap_{i=1}^n A_i$.
From the definition of $\sfCE(H,\mL;\bP_A)$ it follows that the graph $\bigcup_{i=1}^n v_i\mH[A_i]=\bigcup_{i=1}^n w\mH[A_i]$ is isomorphic with the $A$-cluster $\sfCL(H[A],\{A_1,\dots,A_n\})$. Corollary~\ref{cor:Bcomponentscluster} now implies that  $\mB=\bigcup_{i=1}^n \mB_i$ is isomorphic with the $B$-cluster $\sfCL(H[B],\{B\cap A_1,\dots,B\cap A_n\})$.
\end{proof}
The case $|A|=2$, which is not handled in Lemma~\ref{lem:forward induction},
is actually trivial.
\begin{Prop}\label{prop:|A|=2} Let $H$ be an $E$-group, $A\subseteq E$
  with $|A|=2$ and $H[A]$ be retractable. Then every connected
  $A$-subgraph $\mL$ of $\mH[A]$
  is admissible for
  $\pss{A}$-coset extension
  and the full 
  $\pss{A}$-coset extension
  $\sfCE(H,\mL;\bP_A)$ has the cluster property.
\end{Prop}
\begin{proof} Definition~\ref{def:cluster property} is fulfilled for trivial reasons: only the empty set $C=\varnothing$ satisfies $C\subsetneq B\subsetneq A$. Every constituent coset of $\sfCE(H,\mL;\bP_A)$ is of the form $v\mH[a]$ for some letter $a\in A$. Hence, for $B\subsetneq A$, the only $B$-components of $\sfCE(H,\mL;\bP_A)$ which have empty intersection with $\mL$ are singleton vertices which clearly have unique minimal support. 
\end{proof}

Combination of this with Lemma~\ref{lem:forward induction} implies the following result; it encapsulates 
the first of the two inductive procedures discussed above.

\begin{Thm}[forward induction]\label{thm:forward induction} Let $H$ be an $E$-group,
  $A\subseteq E$, $|A|\ge 3$ and suppose that $H[A]$ is
  retractable. Let $\mL\subseteq \mH[A]$ be a connected $A$-graph
  which is admissible for
  $\pss{A}$-coset extension.
  Assume that for all $B\subsetneq A$  and every vertex
  $v\in \mL$ the full
  $\pss{B}$-coset extension
  $\sfCE(H,v\mL[B];\bP_B)$ embeds into $v\mH[B]$ and  is
  bridge-free; then the full
  $\pss{A}$-coset extension
  $\sfCE(H,\mL;\bP_A)$ has the cluster property.
\end{Thm}

\begin{proof} 
In order to reduce the claim of the theorem to
Lemma~\ref{lem:forward induction}, 
we merely need to argue that the 
graphs $\sfCE(H,v\mL[B];\bP_B)$ for $B
\subsetneq A$ have the cluster property. This is proved by induction on $|A|$.
For $|A|=3$ we only need to consider $|B|=2$, so that 
$\sfCE(H,v\mL[B];\bP_B)$ has the cluster property by Proposition~\ref{prop:|A|=2}.
For $|A|>3$ we can use the inductive claim for all 
$|B|<|A|$ (in the r\^ole of $A$) to find that
$\sfCE(H,v\mL[B];\bP_B)$ has the cluster property.
\end{proof}

For the following recall
Definition~\ref{def:covering relation} 
 of when a Cayley graph $\mG$  covers a graph $\mC$
(in terms of canonical morphisms),
 Definition~\ref{def:retractable group} of a $k$-retractable group and Definition~\ref{def:stable}  of a $k$-stable expansion.

 \begin{Def}\label{def:cover}\rm
Suppose that a Cayley graph $\mG$ covers a complete connected graph $\mC$ via a canonical morphism $\varphi\colon \mG\twoheadrightarrow \mC$ and let $\mL\subseteq \mC$ be a connected subgraph.  A
\emph{cover of $\mL$ in $\mG$} (a \emph{$\mG$-cover} for short) is any
connected component of the graph $\varphi\inv(\mL)\subseteq \mG$. 
\end{Def}
Recall
that a crucial feature of covers is the \emph{path lifting property}:
if $\mL$
admits a path $u\longrightarrow v$ labelled $p\in \til{E}^*$ and $u'$
is any vertex of $\varphi^{-1}(\mL)$ such that $\varphi(u')=u$, then
$\varphi^{-1}(\mL)$
admits a path  labelled $p$ with initial vertex $u'$ 
that maps onto the original path in $\mL$ under $\varphi$.

\begin{Thm}[upward induction]\label{thm:upward induction} Let $1\le k<|E|$ and let $H$ be an $E$-group which is $(k+1)$-retractable. Let $A\subseteq E$ with $|A|=k+1$ and let $\mL_H$ be a connected $A$-subgraph of $\mH[A]$ such that
\begin{enumerate}
\item $\mL_H$ is admissible for
$\pss{A}$-coset extension (with respect to $H$),
\item the full
$\pss{A}$-coset extension
$\sfCE(H,\mL_H;\bP_A)$ has the cluster property.
\end{enumerate}
Let $G\twoheadrightarrow H$ be a $k$-stable expansion of $E$-groups
such that the Cayley graph $\mG$ of $G$ covers all graphs of the form
$\overline{\sfCE(H,\mL_H;\bP_A)\circv \mH[B]}$ for $B\subsetneq A$ and
$v$ a vertex of $\sfCE(H,\mL_H;\bP_A)$ 
(thus, in particular, 
$\mG$ covers the graph $\overline{\sfCE(H,\mL_H;\bP_A)}$ itself).
Let $\mL_G$ be any cover of $\mL_H$ in $\mG$.
Then the following hold:
\begin{enumerate}
\item[(i)]  $\mL_G$ is admissible for
$\pss{A}$-coset extension (with respect to $G$),
\item[(ii)] the full
$\pss{A}$-coset extension
 $\sfCE(G, \mL_G;\bP_A)$ embeds into $\mG[A]$,
\item[(iii)]  the {embedded} full
$\pss{A}$-coset extension
  $\sfCE(G,\mL_G;\bP_A)$ is bridge-free in $\mG[A]$.
\end{enumerate}
\end{Thm}
\begin{proof}
  As for~(i), that $\mL_G$ is admissible for
$\pss{A}$-coset extension
follows from the fact that $\mL_H$ is admissible for
$\pss{A}$-coset extension
and that the canonical morphism $G\twoheadrightarrow H$ is
$k$-stable. In this case,  the canonical morphism $\varphi\colon
\mG\twoheadrightarrow \mH$ is injective on $B$-components for
$B\subsetneq A$ so that condition \eqref{eq:freeness} is satisfied
for $\mL_G$ if it is satisfied for $\mL_H=\varphi(\mL_G)$.

Towards injectivity as required for~(ii),
let $\psi\colon \sfCE(G,\mL_G;\bP_A)\to \mG[A]$ be the canonical graph
morphism of Proposition~\ref{prop:morphismCEtoGA}. We first show that
for every $B\subsetneq A$ the restriction of $\psi$
to $\sfCE(G,\mL_G;B)$ is injective. Suppose this were not the case. Since the restriction to $\mL_G$ is an embedding,
that could only happen if two vertices of two distinct constituent
cosets $u\mG[B]$ and $v\mG[B]$ of $\sfCE(G,\mL_G;B)$ were mapped
to the same vertex of $\mG[A]$ and therefore the cosets $u\mG[B]$ and
$v\mG[B]$ coincide as cosets of $\mG[A]$ (see the discussion leading to Remark~\ref{rmk:unfolding}). 
The result in $\mG[A]$ is depicted in Figure~\ref{fig:upward0} (left-hand side).
\begin{figure}[ht]
\begin{tikzpicture}[xscale=0.3,yscale=0.5]
\draw (-7,-2)--(7,-2);
\draw (7.2,-1.9) node [below left]{$\mathcal{L}_G$};
\draw[very thick](-5,-2)--(-2,-2);
\draw[very thick] (2,-2)--(5,-2);
\draw plot [smooth cycle] coordinates {(-5,-2) (-2,-2) (0,0) (2,-2)  (5,-2) (5,2) (-5,2)};
\draw(-4,-2)node[above]{$u$};
\draw(4,-2)node[above]{$v$};
\filldraw (-4,-2) circle (2.5pt);
\filldraw(4,-2)circle(2.5pt);
\draw(0,1.2)node{$u\mathcal{G}[B]=v\mathcal{G}[B]$};
\end{tikzpicture}
\begin{tikzpicture}[scale=0.6]
\draw (-4,0) -- (4,0);
 \filldraw (-2,0) circle (2.5pt);
\filldraw (2.3,0) circle (2.5pt);
\draw[very thick] (-3,0)-- (-1,0);
\draw[very thick] (1,0)--(3,0);
\draw (-2,0) node[above] {$u$};
\draw (2.3,0) node[above] {$v$};
\filldraw (0,2.5) circle (2.5pt);
\draw(0,2.15) node {$z$};
\draw(2.8,0.08) node[below right]  {$\mathcal{L}_G$};
\draw(-1.4,0.9) node{$u\mathcal{G}[B]$};
\draw(1.2,.8) node {$v\mathcal{G}[C]$};
\draw plot [smooth cycle] coordinates {(-3,0) (-1,0)  (1,2.5) (-1,2.5)};
\draw plot [smooth cycle] coordinates {(1,0) (3,0) (1,3) (-2,3)};
\end{tikzpicture}
\caption{Forbidden patterns in the proof of Theorem~\ref{thm:upward induction} (ii)}\label{fig:upward0}
\end{figure}
By assumption, $\mG$ covers the graph
$\overline{\sfCE(H,\mL_H;\bP_A)}$; 
so there is a canonical graph morphism $\varphi\colon \mG\twoheadrightarrow
\overline{\sfCE(H,\mL_H;\bP_A)}$ mapping $\mL_G$ onto $\mL_H$.
Since the expansion
$G\twoheadrightarrow H$ is $k$-stable and $|B|\le k$, the morphism
$\varphi$ maps $u\mG[B]=v\mG[B]$ isomorphically onto
$\varphi(u)\mH[B]$ and likewise onto $\varphi(v)\mH[B]$. Hence
$\varphi(u)\mH[B]=\varphi(v)\mH[B]$ in $\sfCE(H,\mL_H;\bP_A)$, so that
$\varphi(u)$ and $\varphi(v)$ are in the same $B$-component of
$\mL_H$. It follows that $\varphi(u)$ and $\varphi(v)$ can be
connected by a $B$-path which runs in $\mL_H$.
Under $\varphi$ that path lifts to a path $u\longrightarrow v'$ which runs in $\mL_G\cap u\mG[B]$. In particular, $v'\in u\mG[B]=v\mG[B]$ and $\varphi(v')=\varphi(v)$. Since $\varphi$ is injective on $B$-cosets,  $v'=v$ and therefore $u$ and $v$ belong to the same $B$-component of $\mL_G$.
It follows that the constituent cosets $u\mG[B]$ and $v\mG[B]$ of ${\sfCE(G,\mL_G;B)}$ coincide.

So, for the injectivity claim of~(ii), it remains to consider the case when vertices of distinct
coset extensions $\sfCE(G,\mL_G;B)$ and $\sfCE(G,\mL_G;C)$
would violate injectivity. 
Let $B,C\subsetneq A$, $B\ne C$ and  $x\in \sfCE(G,\mL_G;B)$, $y\in \sfCE(G,\mL_G;C)$ be vertices such that $\psi(x)=\psi(y)$. We need to show that $x=y$ in $\sfCE(G,\mL_G;\bP_A)$ (that is, $x$ and $y$ both are in $\sfCE(G,\mL_G;B\cap C)$ and coincide). From $\psi(x)=\psi(y)$ we see that in $\mG[A]$  the  situation is as depicted in Figure~\ref{fig:upward0} (right-hand side)
with $\psi(x)=z=\psi(y)$. That is, $u$ and $z$ are connected by a
$B$-path while $v$ and $z$ are connected by a $C$-path, so that 
$z\in u\mG[B]\cap v\mG[C]$.
Let us consider some canonical graph
morphism $\mG\twoheadrightarrow \overline{\sfCE(H,\mL_H;\bP_A)}$
(according to the statement of the Theorem), which maps $\mL_G$ onto
$\mL_H$. Let $u',v',z'$ be the image vertices of $u,v,z$,
respectively, under this morphism. Then $u',v'\in \mL_H$ and $z'\in
u'\mH[B]\cap v'\mH[C]$. 
The latter intersection is a constituent $(B\cap C)$-coset of
$\sfCE(H,\mL_H;\bP_A)$, having non-empty intersection with the
skeleton $\mL_H$,  say $u'\mH[B]\cap v'\mH[C]=c\mH[B\cap C]$ for some
$c\in \mL_H$. Moreover, the intersections $\mL_H\cap u'\mH[B]$ and
$\mL_H\cap v'\mH[C]$ both are connected (namely $B$- respectively $C$-components of $\mL_H$). This situation is depicted in Figure~\ref{fig:upward2}.
\begin{figure}[ht]
\begin{tikzpicture}[scale=0.6]
\filldraw (-3,0) circle (2.5pt);
\filldraw (3,0) circle (2.5pt);
\filldraw (0,2) circle (2.5pt);
\filldraw (0,4) circle (2.5pt);
\draw[-latex][thick](-3,0)--(-0.1,1.9);
\draw[-latex][thick](3,0)--(0.1,1.9);
\draw [dashed, thick] (-4.5,-1) -- (-3,0);
\draw[dashed, thick] (3,0)--(4.5,-1);
\draw[-latex][thick](0,2)--(0,3.85);
\draw (-3,0.4)node  {$u'$};
\draw (3.1,0.4) node {$v'$};
\draw (0,4)node[above]{$z'$};
\draw(0,1.9)node[below ] {$c$};
\draw(-1.3,0.8)node[above left] {$p$};
\draw(1.3,0.8)node[above right]{$q$};
\draw(-0.1,3) node[right] {$r$};
\draw plot [smooth cycle] coordinates {(-4,0) (-1,-1) (1,4.5) (-2,4.5)};
\draw plot [smooth cycle] coordinates {(2,-1) (4,0)(2,4.5)(-3,5)};
\draw(4.3,-1.2)node[above right]{$\mathcal{L}_H$};
\draw(-2.3,2)node{$u'\mathcal{H}[B]$};
\draw(2.3,2)node{$v'\mathcal{H}[C]$};
\end{tikzpicture}
\caption{Configuration as in the proof of Theorem~\ref{thm:upward induction} (ii)}\label{fig:upward2}
\end{figure}
So there are paths $u'\overset{p}{\longrightarrow}c$ in $\mL_H\cap u'\mH[B]$, $v'\overset{q}{\longrightarrow}c$ in $\mL_H\cap v'\mH[C]$ and $c\overset{r}{\longrightarrow}z'$ in $u'\mH[B]\cap v'\mH[C]$. In particular, $pr$ labels a path $u'\longrightarrow z'$, $qr$  labels a path $v'\longrightarrow z'$. From $k$-stability of the expansion $G\twoheadrightarrow H$ it follows that the morphism $\mG\twoheadrightarrow \overline{\sfCE(H,\mL_H;\bP_A)}$ is injective on all cosets $x\mG[D]$ for all $D\subsetneq A$. In particular, this morphism is bijective between $u\mG[B]$ and $u'\mH[B]$ as well as between $v\mG[C]$ and $v'\mH[C]$. From this it follows that the paths in $u'\mH[B]\cup v'\mH[C]$ just mentioned lift to paths in $u\mG[B]\cup v\mG[C]$: hence there is a path $u\longrightarrow z$ labelled $pr$ and one $v\longrightarrow z$ labelled $qr$. It follows that, in $\sfCE(G,\mL_G;\bP_A)$,
\[u\cdot p=z\cdot r^{-1}=v\cdot q.\]
Since $p\colon u'\longrightarrow c$ runs in $\mL_H$ and so does $q\colon v'\longrightarrow c$, the path $p\colon u\longrightarrow z\cdot r^{-1}$ runs in $\mL_G$, and so does the path $q\colon v\longrightarrow z\cdot r^{-1}$. It follows that \[u\mG[B]=(z\cdot r^{-1})\mG[B]\mbox{ and }v\mG[C]=(z\cdot r^{-1})\mG[C],\] thus  $u\mG[B]\cap v\mG[C]=(z\cdot r^{-1})\mG[B\cap C]$ so that, in $\sfCE(G,\mL_G;\{B, C\})$: \[x=(z\cdot r^{-1})\cdot r=y,\] that is, $x$ and $y$ represent the same vertex in $\sfCE(G,\mL_G;B\cap C)$, as required. Altogether, $\sfCE(G,\mL_G;\bP_A)$ embeds in $\mG[A]$ via the morphism of Proposition~\ref{prop:morphismCEtoGA}.

It remains to argue for~(iii), that $\sfCE(G,\mL_G;\bP_A)$ is bridge-free.
So we look at a pair of 
vertices $v_1,v_2\in \mL_G$, subsets $A_1,A_2\subsetneq A$, and vertices
$s_1\in v_1\mG[A_1]$, $s_2\in v_2\mG[A_1]$, 
and assume that,
for some $B\subsetneq A$, there is a $B$-path $ s_1\overset{p}{\longrightarrow} s_2$ running in $\mG[A]$ (all the following takes place in $\mG[A]$ as depicted in Figure~\ref{fig:upward3}).
\begin{figure}[ht]
\begin{tikzpicture}[scale=0.5]
\draw plot [smooth cycle] coordinates {(-7,-3) (0,-4) (7,-3)(7,3) (0,4)(-7,3)};
\draw (-7,-2)--(7,-2);
\draw (7.2,-2.1) node [above left]{$\mathcal{L}_G$};
\filldraw (-4,-2) circle (2.5pt);
\filldraw(4,-2) circle (2.5pt);
\draw (-4,-2) node[below]{$v_1$};
\draw(4,-2) node[below] {$v_2$};
\draw plot [smooth cycle] coordinates {(-5,-3) (-3,-3) (-3,3) (-5,3)};
\draw plot [smooth cycle] coordinates {(3,-3) (5,-3) (5,3) (3,3)};
\filldraw(-4,2) circle(2.5pt);
\draw(-4,2)node[above]{$s_1$};
\filldraw(4,2) circle (2.5pt);
\draw(4,2)node[above]{$s_2$};
\draw[-latex][thick](-4,2)--(3.9,2);
\draw[-latex][thick](-4,-2)--(3.9,-2);
\draw[-latex][thick](-4,-2)--(-4,1.9);
\draw[-latex][thick](4,-2)--(4,1.9);
\draw(0,2)node[above] {$p$};
\draw(0,2)node[below] {$B$};
\draw(0,-2)node[below]{$q$};
\draw(-4.1,-0)node[right]{$f_1$};
\draw(4.1,0)node[left]{$f_2$};
\draw(7,3)node[right]{$\mathcal{G}[A]$};
\draw(-5.9,0.5)node{$v_1\mathcal{G}[A_1]$};
\draw(5.8,0.5)node{$v_2\mathcal{G}[A_2]$};
\end{tikzpicture}
\caption{First configuration as in the proof of Theorem~\ref{thm:upward induction} (iii)}\label{fig:upward3}
\end{figure}
In addition, there are an $A$-path $v_1\overset{q}{\longrightarrow}
v_2$ running in $\mL_G$ and $A_i$-paths 
$v_i\overset{f_i}{\longrightarrow} s_i$ running in
$v_i\mG[A_i]$. Consider the canonical graph morphism
$\varphi\colon\mG\twoheadrightarrow \overline{\sfCE(H, \mL_H;\bP_A)}$, which maps
$\mL_G$ onto $\mL_H$.  
Let $v_i'$ be the image of $v_i$ in $\mL_H$ under this morphism. The path $
v_1\overset{q}{\longrightarrow} v_2$ is mapped to the path
$v_1'\overset{q}{\longrightarrow} v_2'$
in $\mL_H$.
Let us denote the image of $s_i$ by $s_i'$; then the path
$v_i\overset{f_i}{\longrightarrow} s_i$ running in $v_i\mG[A_i]$ is
mapped to the path $v_i'\overset{f_i}{\longrightarrow} s_i'$ which
runs in $v_i'\mH[A_i]$. So far, all these paths run in $\sfCE(H,\mL_H;\bP_A)$. Further, the path $s_1\overset{p}{\longrightarrow} s_2$ is mapped to the path $s_1'\overset{p}{\longrightarrow} s_2'$, which runs in $\overline{\sfCE(H,\mL_H;\bP_A)}$. It follows that there is a $B$-path
$s_1'\overset{ p^\circ}{\longrightarrow}s_2'$ which runs in $\sfCE(H,\mL_H;\bP_A)$ (in fact, $p^\circ$ is the word obtained from $p$ by deletion of the letters which traverse loop edges of $\ol{\sfCE(H,\mL_H;\bP_A)}\setminus \sfCE(H,\mL_H;\bP_A)$).

So consider the $B$-component $\mB$ of $\sfCE(H,\mL_H;\bP_A)$ which contains the two vertices $s_1'$ and $s_2'$.
The cluster property of $\sfCE(H,\mL_H;\bP_A)$ shows the following:
either $\mB$ has non-empty intersection with the skeleton $\mL_H$, or
else $\mB$ is a $B$-cluster (the existence of unique minimal support
is not needed in this context).
Assume the latter case first: as a $B$-cluster, $\mB$ is the union
$\mB=\mB_1\cup\cdots\cup \mB_n$ of $(B\cap C_i)$-cosets where
$C_i\subsetneq A$, $|C_i|=|A|-1$ and assume first that $n\ge 2$;
the case $n=1$ will be handled below.
We may assume that $s_i'\in
\mB_i$ for $i=1,2$.
The pairs $(A_1,v_1')$ and $(A_2,v_2')$ provide support for $s_1'$ and
$s_2'$, respectively.
The cosets $\mB_1=s_1'\mH[C_1\cap B] \subseteq v_1'\mH[C_1]$
and $\mB_2=s_2'\mH[C_2\cap B] \subseteq v_2'\mH[C_2]$ have 
non-empty intersection (indeed,
 $\mB_1\cap \mB_2$ contains the core of $\mB$). Hence $v_1'\mH[C_1]\cap v_2'\mH[C_2]\ne \varnothing$ so that $v_1'\mH[C_1]\cap v_2'\mH[C_2]=v\mH[C]$ for $C=C_1\cap C_2$ and some vertex $v\in \mL_H$. The situation is depicted in Figure~\ref{fig:upward4}.
\begin{figure}[ht]
\begin{tikzpicture}[xscale=0.7,yscale=0.7]
\draw (-7,-2)--(7,-2);
\draw (7.2,-2.1) node [above left]{$\mathcal{L}_H$};
\filldraw (-4,-2) circle (2.5pt);
\filldraw(4,-2) circle (2.5pt);
\draw (-4,-2) node[above]{$v'_1$};
\draw(4,-2) node[above] {$v'_2$};
\draw plot [smooth cycle] coordinates {(-5,-3) (1,-3) (1,3) (-5,3)};
\draw plot [smooth cycle] coordinates {(-1,-3) (5,-3) (5,3) (-1,3)};
\filldraw(-4,2) circle(2.5pt);
\draw(-4,2)node[left]{$s'_1$};
\filldraw(4,2) circle (2.5pt);
\draw(4,2)node[right]{$s'_2$};
\filldraw(0,-2)circle(2.5pt);
\draw(0,-2)node[below]{$v$};
\filldraw(0,2)circle(2.5pt);
\draw(0,2.04)node[below]{$s$};
\draw[-latex][thick] (-4,2)--(-0.1,2);
\draw[-latex][thick](0,2)--(3.9,2);
\draw(-2.2,1.9)node[above]{$p_1$};
\draw(-2,2.85) node {$\mB_1$};
\draw(2.3,1.85)node[above]{$p_2$};
\draw(2,2.9) node{$\mB_2$};
\draw(-5,-3)node[left]{$v_1'\mathcal{H}[C_1]$};
\draw(5,-3)node[right]{$v_2'\mathcal{H}[C_2]$};
\draw(-2.1,2.1)node[below]{$B\cap C_1$};
\draw(2.2,2.1)node[below]{$B\cap C_2$};
\draw plot [ smooth cycle] coordinates {(-5,-2.5) (-3,-2.5) (-3,2.5)(-5,2.5)};
\draw(-4,0)node{$v_1'\mathcal{H}[A_1]$};
\draw plot [ smooth cycle] coordinates {(5,-2.5) (3,-2.5) (3,2.5)(5,2.5)};
\draw(4,0)node{$v_2'\mathcal{H}[A_2]$};
\draw plot [ smooth cycle] coordinates {(-4.5,1.5) (0,1.5) (0,2.5)(-4.5,2.5)};
\draw plot [ smooth cycle] coordinates {(4.5,1.5) (0,1.5) (0,2.5)(4.5,2.5)};
\end{tikzpicture}
\caption{Second configuration as in the proof of Theorem~\ref{thm:upward induction} (iii)}\label{fig:upward4}
\end{figure}
In particular, there is a vertex $s\in s_1'\mH[B\cap C_1]\cap s_2'\mH[B\cap C_2]$ and there are $B\cap C_i$-paths
\[s_1'\overset{p_1}{\longrightarrow}s\overset{p_2}{\longrightarrow}s_2'\]
labelled $p_i$ ($i=1,2$). We now consider the $B$-augmentation
of $\sfCE(H,\mL_H;\bP_A)$ at the vertex $s$ and the canonical graph morphism
\[\psi\colon\mG\twoheadrightarrow \ol{\sfCE(H,\mL_H;\bP_A)\circs
    \mH[B]}
\]
which maps the covering graph
$\mL_G$ onto $\mL_H$.
The graphs $\sfCE(H,\mL_H;\bP_A)$ and   
$\sfCE(H,\mL_H;\bP_A)\circs\mH[B]$ are almost the same except that
the cluster $\mB$ in the coset extension $\sfCE(H,\mL_H;\bP_A)$
is blown up to the full coset $s\mH[B]$ in the latter graph.
The morphism $\psi$ now
 maps the path $s_1\overset{p}{\longrightarrow} s_2$ to the path
 $s_1'\overset{p}{\longrightarrow} s_2'$ which runs in $s\mH[B]$; but
 $s_1'\overset{p_1}{\longrightarrow}s\overset{p_2}{\longrightarrow}s_2'$
 also run in $s\mH[B]$ which implies that $[p]_H=[p_1p_2]_H$. Since
 the expansion $G\twoheadrightarrow H$ is $k$-stable and $|B|\le k$,
 it follows that $[p]_G=[p_1p_2]_G$. In addition, $k$-stability implies that $\psi$ provides isomorphisms $v_1\mG[C_1]\twoheadrightarrow v_1'\mH[C_1]$ and $v_2\mG[C_2]\twoheadrightarrow v_2'\mH[C_2]$
and therefore also an isomorphism $v_1\mG[C_1]\cup
v_2\mG[C_2]\twoheadrightarrow v_1'\mH[C_1]\cup v_2'\mH[C_2]$ (see Lemma~\ref{lem:small cosets 1}). 
It follows that the path $s_1\overset{p_1}{\longrightarrow}s\cdot p_1$ runs in $v_1\mG[C_1]$ while $s_1\cdot p_1\overset{p_2}{\longrightarrow}s_1\cdot p_1p_2=s_2$ runs in $v_2\mG[C_2]$. 
So the path $s_1\overset{p_1p_2}{\longrightarrow} s_2$ runs entirely in $v_1\mG[C_1]\cup v_2\mG[C_2]\subseteq\sfCE(G,\mL_G;\bP_A)$ and thus provides a $B$-path between $s_1$ and $s_2$ in the coset extension $\sfCE(G,\mL_G;\bP_A)$.
 Finally, for the same reason, we see that in case $n=1$, that is, $\mB=\mB_1\subseteq v_1'\mH[C_1]$ the  path $s_1\mathrel{\overset{p_1p_2}{\longrightarrow}}s_2$ runs in $v_1\mG[C_1]$ which is contained in $\sfCE(G,\mL_G;\bP_A)$.

The remaining case, where  
$\mB$ has non-empty intersection with the skeleton $\mL_H$, is easy: 
in this case $\mB$ is a full $B$-coset $\mB=v\mH[B]$ for some vertex $v\in \mL_H$. The canonical morphism $\varphi\colon \mG\twoheadrightarrow\ol{\sfCE(H,\mL_H;\bP_A)}$ induces an isomorphism 
$\phi\colon s_1\mG[B]=s_2\mG[B]\twoheadrightarrow v\mH[B]$ where $\phi=\varphi\restriction s_1\mG[B]$ is the restriction. Then $s_1\mG[B]=s_2\mG[B]=\phi^{-1}(v\mH[B])=\phi^{-1}(v)\mG[B]$. But $\phi^{-1}(v)\in \mL_G$ so that $s_1\mG[B]=s_2\mG[B]$ is contained in $\sfCE(G,\mL_G;\bP_A)$.
\end{proof}

\section{Construction of the group $G$}\label{sec:groupoids}
The  group $G$  announced in
Lemma~\ref{thm:main theorem} will be constructed via a series of expansions
\begin{equation}\label{eq:series of H and G}
  G_1\twoheadleftarrow H_1\twoheadleftarrow G_2\twoheadleftarrow\cdots\twoheadleftarrow G_{|E|-1}\twoheadleftarrow H_{|E|-1}\twoheadleftarrow G_{|E|}=G
  \end{equation}
  where, for every $k$, the expansions $G_k\twoheadleftarrow H_k$ and
  $H_k \twoheadleftarrow G_{k+1}$  are $k$-stable and the groups $H_k$
  and $G_{k+1}$ are $(k+1)$-retractable.
Here the series~\eqref{eq:series of G} is interleaved with 
the intermediate stages~$H_k$ in~\eqref{eq:series of H and G}.  
The series \eqref{eq:series of H and G} is defined by an ascending series
\begin{equation} \label{eq:series of graphs}
\mX_1\subseteq \mY_1\subseteq \mX_2\subseteq \cdots\subseteq \mX_{|E|-1}\subseteq \mY_{|E|-1}\subseteq \mX_{|E|}
\end{equation}
of complete $E$-graphs such that
each group in the series \eqref{eq:series of H and G}  is the transition group of the
corresponding graph in \eqref{eq:series of graphs}, that is, 
\[G_k=\mathrsfs{T}({\mathcal X}_k)\mbox{ and }H_k=\cT(\mY_k)\]
for all $k$ in question.
Every graph in the series~\eqref{eq:series of graphs} is obtained from its predecessor by adding certain complete components.
These components are constructed by an inductive procedure, the idea of
which is as follows.
The graph ${\mathcal X}_1$ is obtained as a suitable completion of
the given oriented graph $\mE=(V,\til{E};\alpha,\omega,{}^{-1})$,
here considered as an $E$-labelled graph where every edge gets its
own label. This serves to initialise the series \eqref{eq:series of H and G}
with $G_1:=\mathrsfs{T}({\mathcal X}_1)$.

Suppose that for
$k\ge 1$ the graph $\mX_k$ and therefore its
transition group $G_k$ have already been constructed.
Then the step $\mX_k\leadsto \mY_k$, and hence the step $G_k\leadsto H_k$,
raises the ``degree of retractability" from $k$ to $k+1$ and thereby
lays the ground for the transition $H_k\leadsto G_{k+1}$.
{The latter} step is intended to
ensure
the following: suppose that $p$ is a word over $k+1$ letters
which forms a path $u\longrightarrow v$ in $\mE$ and $a\in \co(p)$ for
some $a\in E$; if $H_k$ satisfies the relation $p=p_{a\to 1}$, but
there is no word $q$ in the letters
$B:=\co(p)\setminus\{a\}$ (and their inverses)
forming a path $u\longrightarrow v$ in $\mE$ such that $H_k$ satisfies the relation 
$p=q$,
then some component of $\mX_{k+1}\setminus \mY_k$ guarantees that $G_{k+1}$ avoids the relation $p=p_{a\to 1}$ and therefore \textsl{every} relation $p=q$ with $q\in \til{B}^*$.

\subsection{Definition of $G_1$ and the transition $G_k\leadsto H_k$}
\label{subsec:GktoHk}
The idea {of the construction of} the graph ${\mathcal X}_1$ is to
extend the given oriented graph
$\mE=(V,\til{E},\alpha,\omega,{}^{-1})$ to a complete $E$-graph on
the vertex set $V$ in whose transition group the permutation $[e]$
corresponding to any non-loop edge $e$ is the transposition in $V$ 
that swaps the two vertices $\alpha e$ and $\omega e$. Let $\mE=(V,
\til{E};\alpha,\omega,{}\inv)$ be a finite connected oriented
graph. We let the set of positive edges $E$ be our alphabet and label
every edge $e$ by itself.
Thereby we get the $E$-labelled graph $(V,
\til{E};\alpha,\omega,{}^{-1},\ell, E)$
where  $\ell$ is the identity function mapping every $e\in \til{E}$,
considered as an edge, to itself, considered as a label.
The resulting graph is
an $E$-graph for
trivial reasons, since every label appears exactly once. 

Next,  for every non-loop edge $e$ we add a
new  edge $\bar{e}$ and set
\[\alpha \bar{e}:=\omega e,\ \omega \bar{e}:= \alpha e,\ \ell(\bar{e}):=\ell(e)=e.\]
We have thus completed every non-loop edge\begin{tikzpicture}[scale=0.35, baseline = {(0, -0.44)}]
\draw[-latex] (0.2,-1) to (1.8,-1);
\filldraw(0,-1+0)circle(1.5pt);
\filldraw(2,-1+0)circle(1.5pt);
\draw (1,-1-0.4)node{$e$};
\draw[left](0,-1+0)node{$u$};
\draw[right](2,-1+0)node{$v$};
\end{tikzpicture}\!
to a $2$-cycle
\begin{tikzpicture}[scale=0.4,  baseline = {(0, -0.49)}]
\draw[-latex] (0.1,-1-0.1)..controls (1,-1-0.75)..(1.9,-1-0.1);
\draw[-latex] (1.9,-1+0.1)..controls (1,-1+0.75).. (0.1,-1+0.1);
\filldraw(0,-1+0)circle(1.5pt);
\filldraw(2,-1+0)circle(1.5pt);
\draw (1,-1-0.9)node{$e$};
\draw(1,-1+0.9)node{$e$};
\draw[left](0,-1+0)node{$u$};
\draw[right](2,-1+0)node{$v.$};
\end{tikzpicture}
Let us denote the set of all positive edges so obtained (the original ones and the added ones) by $D$; then 
the oriented $E$-graph $\mD=(V,\til{D};\alpha,\omega,{}\inv,\ell,E)$  is weakly complete. 
Let ${\mathcal X}_1:=\ol{\mD}$ be its trivial completion.
The transition group $G_1:=\cT({\mathcal X}_1)$ is an $E$-group of permutations acting on the vertex set $V$.  For every $e\in E$, $[e]_{G_1}$ is either a transposition (if $e$ is not a loop edge then $[e]$ swaps $\alpha e$ and $\omega e$) or the identity permutation (if $e$ is a loop edge). Note that two distinct labels $e,f\in E$ may represent the same permutation of $V$ (since we allow multiple edges in $\mE$). 
\begin{Rmk}  Instead of completing all non-loop edges to  $2$-cycles we could equally well complete every such edge $e$ to an $n$-cycle for any fixed $n\ge 2$, by attaching to  the edge $u\overset{e}{\longrightarrow} v$ an $e$-path $u\overset{e}{\longleftarrow}\cdots\overset{e}{\longleftarrow} v$ consisting of a sequence of $n-1$ new edges labelled  $e$ and $n-2$ new intermediate vertices. In the resulting transition group, the permutation $[e]$ assigned to  $e$ then is a cyclic permutation of length $n$ mapping $\alpha e$ to $  \omega e $. Distinct labels coming from non-loop edges then automatically represent different permutations provided that $n\ge 3$.
\end{Rmk}

The transition from $G_k$ to $H_k$ is easily described.
Suppose we have already defined the graph ${\mathcal X}_k$ and thus the group $G_k=\cT({\mathcal X}_k)$. We set
\begin{equation}\label{eq: mY_k}
\mY_k:={\mathcal X}_k\,\sqcup\,\bigsqcup \bigl\{\ol{\mG_k[A]}\colon
A\subseteq E, |A|=k\bigr\}.
\end{equation}
Provided that $G_k$ is $k$-retractable, 
the transition group  $H_k = \cT({\mathcal Y}_k)$ is $(k+1)$-retractable and the expansion $G_k\twoheadleftarrow H_k$ is $k$-stable (Theorem~\ref{thm:first basic}).
In particular, $H_1$ is $2$-retractable.

\subsection{The transition $H_k\leadsto G_{k+1}$}
\label{subsec:HktoGk+1}
The expansion $H_k\twoheadleftarrow G_{k+1}$
is more delicate.
  We assemble  a complete $E$-graph
   ${\mathcal X}_{k+1} =\mY_k\sqcup \ol{\mZ_k}$ to obtain
   $G_{k+1}$ as the transition group  $G_{k+1} = \cT({\mathcal X}_{k+1})$.
   The new, weakly complete components of $\mZ_k$ will be constructed
   as augmentations of clusters and coset extensions based
   on $H_k$. To this end we first collect, for $k\ge 2$, properties
   of the precursors $G_k$ and $H_{k-1}$ of $H_k$,
   which then serve as conditions to be maintained
   inductively also in the passage to $G_{k+1}$.
   At level~$k$, we denote these inductive conditions as
   $\mbox{\sc Cond}_k$ for the pair $(G_k, H_{k-1})$.  
   So $\mbox{\sc Cond}_k$ will serve as inductive hypothesis
  for the construction of $\mZ_k$, and hence $G_{k+1}$,
   which then needs to guarantee that 
   the conditions  $\mbox{\sc Cond}_{k+1}$
   are  satisfied by the pair $(G_{k+1}, H_k)$. In the following, we
   identify subgraphs of $\mE$ with their labelled versions inside
   $\mX_1$.

 \begin{Cond}\label{def:condk}\rm
As conditions $\mbox{\sc Cond}_k$, for $k \geq 2$, we collect the following:
   \begin{enumerate}
   \item[(i)]
    $H_{k-1}$ and $G_k$ are $k$-retractable and the expansion
     $H_{k-1}\twoheadleftarrow G_{k}$ 
     is $(k-1)$-stable,
   \end{enumerate}
   and,  for every $B\subseteq E$ with $\vert B\vert\le k$,
   for any $\mG_k$-cover $\mC_{\mG_k}$ 
of any connected component
$\mC$ of $\mB=\langle B\rangle$ in $\mE\subseteq \mX_1$, 
the following hold:
\begin{enumerate}
\item[(ii)] $\mC_{\mG_k}$ is admissible for $\pss{B}$-coset extension, 
\item[(iii)] the full $\pss{B}$-coset extension
$\sfCE(G_k, \mC_{\mG_k};\bP_B)$ embeds into $\mG_k[B]$,
\item[(iv)] the {embedded} full $\pss{B}$-coset extension $\sfCE(G_k,\mC_{\mG_k};\bP_B)$ is bridge-free.
\end{enumerate}
\end{Cond}

By Theorem~\ref{thm:first basic}, (i) implies that
$H_k\twoheadrightarrow G_k$ in particular is $k$-stable {and $H_k$ is $k+1$-retractable by construction}, as already mentioned 
in connection with the definition of $H_k$.
Let $\psi_k\colon \mG_k\twoheadrightarrow {\mathcal X}_1$ be some
canonical graph morphism, $\chi_k\colon \mH_k\twoheadrightarrow \mG_k$
the graph morphism induced by the canonical morphism
$H_k\twoheadrightarrow G_k$, and let $\varphi_k=\psi_k\circ\chi_k$.
By $k$-stability, $\chi_k$ is injective on connected $B$-subgraphs for
$|B|\le k$.

Let $A\subseteq E$ be a set of $|A|=k+1$ (positive) edges of
$\mE\subseteq {\mathcal X}_1$ and 
$\mA=\langle A\rangle$ be the subgraph of $\mE$ spanned by $A$.
Let $\mC$ be a connected component of 
$\mA$ and $\mC_{\mH_k}$ be  an
{$\mH_k$-cover} of $\mC$, that is, some connected component of
$\varphi_k\inv (\mC)$.
We show that 
$\mC_{\mH_k}$ is admissible for $\pss{A}$-coset extension
(with respect to $H_k$)
and that the full $\pss{A}$-coset extension
$\sfCE(H_k,\mC_{\mH_k}; \bP_A)$ has the cluster property. 
From this it will follow that augmented coset extensions of the form
$\sfCE(H_k,\mC_{\mH_k}; \bP_A)\circv \mH_k[B]$ are well defined; they
will be essential ingredients of the graph $\mZ_k$ to be defined below
(Definition~\ref{def:Z_k}). 

Let $B\subsetneq A$ and let $\mU\subseteq \mC_{\mH_k}$ be some
$B$-component of $\mC_{\mH_k}$. Then $\varphi_k(\mU)\subseteq \mC$
is a $B$-component of $\mC$ and hence is a connected component of
$\langle B\rangle\subseteq \mC$.
\begin{center}
{\begin{tikzcd} \mathcal{U}\ar[draw=none]{r}[sloped,auto=false]{\hspace{-0.3cm}\subseteq\hspace{-0.435cm}} \arrow[d, two heads]
	&\mathcal{C}_{\mathcal{H}_k}\ar[draw=none]{r}[sloped,auto=false]{\hspace{-0.3cm}\subseteq\hspace{-0.435cm}} \arrow[d, two heads]
		&\mathcal{H}_k
		\arrow[d, two heads,  "\varphi_k"']\\
\varphi_k(\mathcal{U}) \ar[draw=none]{r}[sloped,auto=false]{\hspace{-0.3cm}\subseteq\hspace{-0.435cm}} 
	&\mathcal{C} \ar[draw=none]{r}[sloped,auto=false]{\hspace{-0.3cm}\subseteq\hspace{-0.435cm}}
		& \mathcal{X}_1
              \end{tikzcd}}
          \end{center}

By the inductive hypothesis, any $\mG_k$-cover $\mU'$ of
$\varphi_k(\mU)$ is admissible for
$\pss{B}$-coset extension
(with respect to $G_k$) and  $\sfCE(G_k,\mU';\bP_B)$ embeds
into $\mG_k[B]$ and is bridge-free by~(ii)--(iv). 
Since the morphism $\chi_k\colon \mH_k\twoheadrightarrow \mG_k$ is injective on $B$-components (that is, injective on $B$-cosets), it follows that $\mU'\cong\mU$ and hence also
\begin{equation}\label{eq:small components}
\sfCE(G_k,\mU';\bP_B)\cong \sfCE(H_k,\mU;\bP_B).
\end{equation}
Altogether, by~(iii) we have 
\begin{center}
{\begin{tikzcd}
\mathsf{CE}(H_k,\mU;\bP_B) \ar[draw=none]{d}[sloped,auto=false]{\hspace{-0.3cm}\cong\hspace{-0.435cm}} 
			&\mathcal{H}_k[B] \ar[draw=none]{d}[sloped,auto=false]{\hspace{-0.3cm}\cong\hspace{-0.435cm}} \\
\sfCE(G_k,\mU';\bP_B) \ar[r, hookrightarrow]
	& \mathcal{G}_k[B] 
              \end{tikzcd}}
\end{center}
so that $\sfCE(H_k,\mU;\mathbb{P}_B)$ canonically embeds into $\mH_k[B]$. 
It follows that condition \eqref{eq:freeness} of
Definition~\ref{def:admissible}  is fulfilled. Since this is true for every $B$-component $\mU$ for every proper subset $B$ of $A$ this implies that $\mC_{\mH_k}$
is admissible for
$\pss{A}$-coset extension
(with respect to $H_k$). Once more by the inductive
hypothesis (iv), 
every graph in \eqref{eq:small components} is
bridge-free. Then, by Theorem~\ref{thm:forward induction}, the full
$\pss{A}$-coset extension
$\sfCE(H_k,\mC_{\mH_k};\bP_A)$ itself has the cluster property.  As already mentioned, this guarantees that the augmented coset extensions $\sfCE(H_k,\mC_{\mH_k}; \bP_A)\circv \mH_k[B]$ of Definition~\ref{def:Z_k} (2) below are well defined.  We therefore can now define the components of the graph $\mZ_k$.
\begin{Def}\label{def:Z_k} The graph $\mZ_k$ is the disjoint union of
\begin{enumerate}
\item  all augmented $A$-clusters
\[\sfCL(H_k[A],\bP)\circv\mH_k[B]\]
for $A\subseteq E$ with $|A|=k+1$, $\bP$ a set of proper subsets of $A$, $v$ a vertex of $\sfCL(H_k[A],\bP)$ and $B\subsetneq A$;
\item all augmented full
  $\pss{A}$-coset extensions
\[\sfCE(H_k, \mC_{\mH_k};\bP_A)\circv \mH_k[B]\]
for $A\subseteq E$ with $|A|=k+1$, $\mC$ a connected component of $\mA=\langle A\rangle$, $\mC_{\mH_k}$ an $\mH_k$-cover of $\mC$, $\bP_A$ the set of all proper subsets of $A$,  $v$ a vertex of $\sfCE(H_k,\mC_{\mH_k};\bP_A)$ and $B\subsetneq A$.
\end{enumerate}
\end{Def}
We note that the augmented clusters and augmented 
coset extensions contain, for $B=\varnothing$, all ``plain'' clusters and coset extensions. Recall that $\mH_k = \cT({\mathcal Y}_{k})$ and $G_{k+1} =\cT({\mathcal X}_{k+1}) = \cT({\mathcal Y}_{k}\sqcup \ol{\mZ_k})$; see \eqref{eq: mY_k} for $\mY_k$.

\begin{Prop}\label{prop:k-stabilityH_k--->} The expansion $H_k\twoheadleftarrow G_{k+1}$ is $k$-stable and hence $G_{k+1}$ is $(k+1)$-retractable.
\end{Prop}

\begin{proof} We need to prove $k$-stability, the second assertion then follows from Theorem~\ref{thm:first basic} by inductive hypothesis (i) and the definition of $H_k$. Let $C\subseteq E$ with $|C|=k$, let $p\in \til{C}^*$ and assume that $[p]_{G_{k+1}}\ne 1$; we need to show that $[p]_{H_k}\ne 1$. There exists a component $\mL$ of $\mY_k$ or of $\ol{\mZ_k}$ witnessing the inequality
$[p]_{G_{k+1}}\ne 1$.
That is, in this component there is a vertex $v$ such that $v\cdot
p\ne v$. If the witnessing component $\mL$ belongs to $\mY_k$, then we
are done since then $[p]_{H_k}\ne 1$ immediately follows from
$H_k = \cT({\mathcal Y}_k)$.
If $\mL$ is a component of $\ol{\mZ_k}$, then $\mL=\ol{\mM}$ where
$\mM$ is of the form (1) or (2) in Definition~\ref{def:Z_k},
and the path $p\colon v\longrightarrow v\cdot p$ runs in the
$C$-component $v\ol{\mM}[C]$. Recall that $v\ol{\mM}[C]$ denotes the
$C$-component of $v$ in the graph $\ol{\mM}$ while $\ol{v\mM[C]}$ is
the trivial completion of $v\mM[C]$, that is, the trivial completion
of the $C$-component of $v$ in $\mM$.
Obviously
\[v\mM[C]\subseteq v\ol{\mM}[C]\subseteq \ol{v\mM[C]},
\]
and the latter two graphs differ only in loop edges having labels
not in $C$.
Hence $C$-paths in $v\ol{\mM}[C]$ and $\ol{v\mM[C]}$
traverse the same edges and meet the same vertices. It is therefore
sufficient to look at $\ol{v\mM[C]}$ instead of $v\ol{\mM}[C]$. From
Corollaries~\ref{cor:Bcomponentscluster},
\ref{cor:CcomponentBexpandedcluster} and 
Proposition~\ref{prop:2acyclicexpandedAext}, and since the 
(plain) coset extensions in Definition~\ref{def:Z_k}~(2)
have the cluster property, it follows that, for the graph $\mM$ in
question,  the $C$-component $v\mM[C]$ must be isomorphic with one of the following:
\begin{romanenumerate}
\item
  a full $C$-coset $\mH_k[C]$, or
\item
a $C$-cluster $\sfCL(H_k[C],\bP)$ for some set $\bP$ of proper subsets
of $C$ (this includes, for $\bP=\{B\}$, also $B$-cosets $\mH_k[B]$ for
$B\subsetneq C$), or
\item
a $D$-augmented
$C$-cluster $\sfCL(H_k[C],\bP)\circu\mH_k[D]$ for
some set $\bP$ of proper subsets of $C$, some vertex $u$ of
$\sfCL(H_k[C],\bP)$ and  some proper subset $D$ of $C$. 
\end{romanenumerate}
In case~(i),  $v\mM[C]\cong\mH_k[C]$, so the claim $[p]_{H_k}\ne 1$ again
follows immediately. 
In case~(ii) we get
\[ v\mM[C] \cong \sfCL(H_k[C],\bP)\cong \sfCL(G_k[C],\bP)\cong
    \sfCL(H_{k-1}[C],\bP)\]
where the second isomorphism is obvious since $H_k[C]\cong G_k[C]$ by $k$-stability of $H_k\twoheadrightarrow G_k$ while the third isomorphism follows  from Lemma~\ref{lem:small cosets 1}. In case~(iii) we get
\begin{align*}
v\mM[C] &\cong \sfCL(H_k[C],\bP)\circu \mH_k[D]\\
&\cong
                              \sfCL(G_k[C],\bP)\circt \mG_k[D]
    \cong \sfCL(H_{k-1}[C],\bP)\circs \mH_{k-1}[D]
\end{align*}
where $t$ and $s$ are the images of $u$ under the canonical morphisms $H_k\twoheadrightarrow G_k$ and $H_k\twoheadrightarrow H_{k-1}$, respectively, and, again, the second isomorphism is obvious since $H_k[C]\cong G_k[C]$ and $\mH_k[D]\cong \mG_k[D]$ by $k$-stability of $H_k\twoheadrightarrow G_k$ while  the third isomorphism follows from Lemma~\ref{lem:small cosets 2}.
Hence, in cases~(ii) and (iii),
$\ol{v\mM[C]}$  is isomorphic with a component of $\ol{\mZ_{k-1}}$ so that $[p]_{G_k}\ne 1$,
from which again $[p]_{H_k}\ne 1$ follows. 
\end{proof}

From Theorem~\ref{thm:upward induction} it follows that for every set
$A\subseteq E$ with $|A|=k+1$ and every connected component
$\mC$ of $\mA$, every $\mG_{k+1}$-cover $\mC_{\mG_{k+1}}$
(that is, every connected component of $\psi_{k+1}\inv(\mC)$ in
$\mG_{k+1}$ where $\psi_{k+1}\colon \mG_{k+1}\twoheadrightarrow
{\mathcal X}_1$ is a canonical graph morphism) is admissible for
$\pss{A}$-coset extension,
and the full
$\pss{A}$-coset extension
$\sfCE(G_{k+1},\mC_{\mG_{k+1}};\bP_A)$ embeds into
$\mG_{k+1}[A]$ and is bridge-free.  If $|A|=l<k+1$ we have by
induction that, for every connected component $\mC$ of $\mA$,  the full
$\pss{A}$-coset extension
$\sfCE(G_l,\mC_{\mG_l};\bP_A)$ embeds into $\mG_l[A]$. But the
expansion $G_l\twoheadleftarrow G_{k+1}$ is $l$-stable whence
$\sfCE(G_l,\mC_{\mG_l};\bP_A)\cong
\sfCE(G_{k+1},\mC_{\mG_{k+1}};\bP_A)$ and $\mG_l[A]\cong
\mG_{k+1}[A]$.
We have thus maintained Condition~\ref{def:condk} in the passage from $k$ to
$k+1$ by having verified $\mbox{\sc Cond}_{k+1}$:
\begin{enumerate}
\item[(i)] $H_k$ and $G_{k+1}$ are $(k+1)$-retractable and the expansion $G_{k+1}\twoheadrightarrow H_k$ is $k$-stable (by Proposition~\ref{prop:k-stabilityH_k--->})
\end{enumerate}
and, 
for every $A\subseteq E$ with $|A|\le k+1$,  for any $\mG_{k+1}$-cover $\mC_{\mG_{k+1}}$ of every connected component $\mC$ of $\mA=\langle A\rangle$ in $\mE\subseteq \mX_1$, the following hold:
\begin{enumerate}
\item[(ii)] $\mC_{\mG_{k+1}}$ is admissible for
$\pss{A}$-coset extension,
\item[(iii)] the full $\pss{A}$-coset extension $\sfCE(G_{k+1},\mC_{\mG_{k+1}};\bP_A)$ embeds into $\mG_{k+1}[A]$,
\item[(iv)] the {embedded} full $\pss{A}$-coset extension $\sfCE(G_{k+1},\mC_{\mG_{k+1}};\bP_A)$ is bridge-free.
\end{enumerate}
We check that the base case for this inductive procedure,
             $\mbox{\sc Cond}_{2}$ for the pair $(G_2,H_1)$, goes through.
The group $H_1$ is $2$-retractable and so is $G_2$ since
$G_2\twoheadrightarrow H_1$ is $1$-stable (cf.\ Theorem~\ref{thm:first basic}). 
By Proposition~\ref{prop:|A|=2}, for every set
$A\subseteq E$ with $|A|=2$, every $\mH_1$-cover $\mC_{\mH_1}$ of every
component $\mC$ of $\mA$ is admissible for
$\pss{A}$-coset extension (with respect to $H_1$)
and $\sfCE(H_1,\mC_{\mH_1}; \bP_A)$ has the cluster
property. Theorem~\ref{thm:upward induction} then implies that the $\mG_2$-cover $\mC_{\mG_2}$ is admissible for $\pss{A}$-coset extension (with respect to $G_2$) and that
$\sfCE(G_2,\mC_{\mG_2};\bP_A)$ embeds in $\mG_2[A]$ and is bridge-free
(the assertions for $G_2$ can also be checked by direct inspection). In other words, we have shown that conditions $\mbox{\sc Cond}_{2}$ are satisfied by the pair $(G_2,H_1)$.
Altogether the series of expansions
\[G_1\twoheadleftarrow H_1\twoheadleftarrow G_2\twoheadleftarrow\cdots\twoheadleftarrow G_{|E|-1}\twoheadleftarrow H_{|E|-1}\twoheadleftarrow G_{|E|}\] is well defined and $G=G_{|E|}$ is retractable.

\subsection{Properties of $G=G_{|E|}$}
We need to argue that $G$ satisfies 
the requirements of Lemma~\ref{thm:main theorem}. Requirement~(2),
that  $G$ is retractable, and therefore has a content function
by Proposition~\ref{prop:retractablehasconctent},
has already been proved.

We are left with showing requirements~(1) and~(3):
\begin{enumerate}
\item[(1)] that every permutation of $E$ induced by an
automorphism of $\mE$ extends to an automorphism of $G$,
and
\item[(3)] that for every word which forms a path $u\longrightarrow v$ in
$\mE$ there is a $G$-equivalent word which also 
forms a path $u\longrightarrow v$ and uses only edges of the (common)
$G$-content, or $u=v$ in case
of empty content.
\end{enumerate}

We start with item (1);
(3) will then be dealt with
 in Lemma~\ref{lem:contentconnected} and Corollary~\ref{cor:crucial content}.
In the context of~(1), ``an automorphism of $\mE$'' refers to any automorphism
of the \textsl{unlabelled oriented} graph $\mE=(V, \til{E};\alpha,\omega,{}^{-1})$.
Recall from the definition of an
automorphism of an oriented graph that every such automorphism of
$\mE$ is required to induce a permutation on the set $E$ of positive
edges of $\mE$, hence induces a permutation on our labelling alphabet $E$.
Similarly, ``an automorphism of $G$'' means automorphism of the mere
group $G$ (rather than of $G$ as an $E$-group,
which cannot have non-trivial automorphisms).

\begin{Prop}\label{prop:auto extends} Every permutation $E\to E$ induced by an automorphism of the oriented graph $\mE$ extends to an automorphism of $G$.
\end{Prop}
\begin{proof} Let $\gamma$ be a permutation of $E$ induced by an
  automorphism  of $\mE$, also denoted $\gamma$.
We demonstrate the required property for all $G_k$ and $H_k$, 
 by induction on $k$.
First note that  $\gamma$   (uniquely) extends to an automorphism
$\hat{\gamma}$   of $\mathcal{X}_1$ from which the claim follows for
the group $G_1$. Indeed, for every pair of vertices $u,v\in \mX_1$ and
every word $p\in \til{E}^*$, we have $p\colon u\longrightarrow v$ if
and only if ${}^\gamma p\colon {}^{\hat{\gamma}}u\longrightarrow {}^{\hat{\gamma}}v$.
Consequently, for every word $p\in \til{E}^*$, $G_1$ satisfies the relation $p=1$ if and only if it satisfies ${}^{\gamma}p=1$. 
  
  So let $k\ge 1$ and assume inductively that $\gamma$ extends to an automorphism $\hat{\gamma}$ of $\mX_k$ (this means that there is an automorphism $\hat{\gamma}$ of the oriented graph $\mX_k$ such that for every edge $e\in \mX_k$ we have $\ell({}^{\hat{\gamma}}e)={}^{\gamma}\ell(e)$); by the same reasoning as for $k=1$ we see that in this case $\gamma$ extends to an automorphism of $G_k$. From the definition of the graph $\mY_k$ it now follows that $\gamma$ extends to an automorphism $\hat{\gamma}$ of $\mY_k$ which again implies that $\gamma$ extends to an automorphism of $H_k$. From this in turn it follows that $\gamma$ extends to an automorphism of $\mX_{k+1}$ and therefore again to an automorphism of $G_{k+1}$.
\end{proof}

The assertion of the last proposition is essentially a direct
consequence of the fact that the entire process behind our
construction of $G$, on the basis of the given oriented graph $\mE$, is symmetry-preserving. Indeed, none of the intermediate steps involves
any choices that could possibly break symmetries in the input data,
i.e.\ could be incompatible with isomorphisms between oriented input
graphs $\mE$. 
In particular, the inductive construction steps 
reflected in Theorems~\ref{thm:forward induction} and~\ref{thm:upward
  induction} proceed by cardinality of subsets of $E$ and treat all subsets
of the same size uniformly and in parallel.\footnote{This should be
  contrasted e.g.\ with constructions based on some
enumeration of the subsets of $E$, which could well break symmetries.}
Any isomorphism between oriented graphs $\mE\cong\mE'$ 
would successively extend to isomorphisms between the associated graphs 
$\mX_i \cong \mX_i'$ and $\mY_i \cong \mY_i'$ and induced isomorphisms 
between their transition groups  $G_i \cong G_i'$ and $H_i \cong H_i'$.  In this sense, the entire inductive process
underlying the expansion chain~(\ref{eq:series of H and G})
is isomorphism-respecting, hence in particular compatible with
permutations of $E$ stemming from automorphisms of $\mE$.

Finally, we have to deal with requirement (3) of Lemma~\ref{thm:main theorem}.
Recall that for a word $p\in \til{E}^*$, $\co(p)$ is the set of all
letters $a\in E$ for which $a$ or $a\inv$ occurs in $p$. 
The following lemma is crucial for establishing~(3). 
{The reader is 
invited to recall the group $G$ defined in \eqref{eq:series of H and G}, the graphs $\mX_{k+1}:=\mY_k\sqcup \ol{\mZ_k}$ (for $\mZ_k$ see Definition~\ref{def:Z_k}) and the coset extensions $\mathsf{CE}(G,\mK;\bP)$ defined in \eqref{eq:CEgeneral}; the full coset extension $\mathsf{CE}(G,\mK;\bP_A)$ is defined immediately before 
Remark~\ref{rmk: proper A-graph}. 
Also recall that the Cayley graph $\mG$ of $G$ covers, in the sense of Definition~\ref{def:covering relation},
any connected component of any one of the graphs $\mX_k$.}
\begin{Lemma}\label{lem:contentconnected}
Let $p\in \til{E}^*$ be a word that forms a path $u\longrightarrow v$ in $\mE$; let $A=\co(p)$ and suppose that for some letter $a\in A$ and $B=A\setminus\{a\}$ there exists a word $r\in \til{B}^*$ such that $[p]_G=[r]_G$. Then there exists a word $q\in \til{B}^*$ such that $[p]_G=[q]_G$ and, in addition, $q$  forms a path $u\longrightarrow v$ in $\mE$.
\end{Lemma}
\begin{proof}
First recall that every loop edge $e$ of $\mE$ induces the identity permutation on the set $V$ of vertices of $\mX_1$, whence $[e]_{G_1}=1$; then $[e]_G=1$ follows from the fact that the expansion $G\twoheadrightarrow G_1$ is $1$-stable. Hence, if $p$ contains only loop edges then $u=v$,  the path meets only the
vertex $u$  and $[p]_G=1$ so that for $q$ we may choose the empty
word~$1$, which  labels the empty path $u\longrightarrow u$ and
$[p]_G=[1]_G$.

If $e$ is not a loop edge, then no power $e^n$ or $e^{-n}$ for  $n\ge 2$ forms a path; therefore, if $|A|=1$ the only possibilities for $p$ are $f(f^{-1}f)^n$ and $(ff^{-1})^{n+1}$ for some $n\ge 0$ and $f\in \{e,e^{-1}\}$. In these cases the claim is obvious.
 
In the following we use the notation of the series \eqref{eq:series of H and G} and denote the Cayley graphs of $H_k$ and $G$ by $\mH_k$ and $\mG$, respectively.
So, let $|A|=k+1$ for some $k\ge 1$, and let $\mA=\langle A\rangle=\langle p\rangle$ 
be the subgraph of $\mE$ spanned by $A$, which, by definition, is the same as the subgraph of $\mE$ spanned by the path $p$ (which therefore is connected).
Abusing notation, we denote the labelled version of $\mA$ inside $\mX_1$ also by $\mA$ and let $\varphi_u\colon \mH_k\twoheadrightarrow {\mathcal X}_1$ be the canonical morphism mapping $1\in \mH_k$ to $u$; let $\mA_k\subseteq \mH_k$ be the cover of $\mA$ in $\mH_k$ with $1\in \mA_k$ (that is, the connected component of $\varphi_u\inv(\mA)$ which contains the vertex $1$). The path 
$p$ in $\mE$, or, more precisely,
the path $\pi_u^{\mX_1}(p)$ lifts to the path $\pi_1^{\mA_k}(p)$. In particular, in $\mA_k$ there is a $p$-labelled path starting at $1$.
We consider the full
$\pss{A}$-coset extension
$\sfCE(H_k,\mA_k;\bP_A)$ and note that $\sfCE(H_k,\mA_k;B)$ is a subgraph of it.  We also have the path $\pi_1^\mG(p)$ in $\mG$ starting at $1$ and being labelled $p$. The canonical morphism $\psi\colon \mG\twoheadrightarrow \ol{\sfCE(H_k,\mA_k;\bP_A)}$ (mapping $1\in \mG$ to $1\in \mA_k$) maps $\pi_1^\mG(p)$ to $\pi_1^{\ol{\sfCE(H_k,\mA_k;\bP_A)}}(p)$, but this path runs entirely in $\mA_k$, hence coincides with the  path $\pi_1^{\mA_k}(p)$ mentioned earlier. 

By assumption, $[p]_G=[r]_G$ for some word $r\in \til{B}^*$. The paths
$\pi_1^\mG(p)$ and $\pi_1^\mG(r)$ have the same terminal vertex,
namely $[p]_G=[r]_G$. The  path $\pi_1^\mG(r)$ is mapped by $\psi$
onto the path $\pi_1^{\ol{\sfCE(H_k,\mA_k;\bP_A)}}(r)$. But the
$B$-component of $1$ in $\ol{\sfCE(H_k,\mA_k;\bP_A)}$ is the full
$B$-coset $1\mH_k[B], $ which is contained in $\sfCE(H_k,\mA_k;B)$.
So the latter graph contains a path labelled $r$ starting at $1$, and that
path $\pi_1^{\sfCE(H_k,\mA_k;B)}(r)$ actually runs inside $1\mH_k[B]$.
Since the paths $\pi_1^\mG(r)$ and  $\pi_1^\mG(p)$ have the same
terminal vertex, so have the paths
\[\pi_1^{1\mH[B]}(r)=\pi_1^{\sfCE(H_k,\mA_k;B)}(r)\mbox{ and }\pi_1^{\mA_k}(p).\]
It follows that the terminal vertex $v'$ of $\pi_1^{\mA_k}(p)$ is in $\mA_k\cap 1\mH_k[B]$. But $\mA_k\cap 1\mH_k[B]$ is just the $B$-component of $1$ in $\mA_k$, which is a connected $B$-graph. Altogether, there exists a path $\pi\colon 1\longrightarrow v'$ running in $\mA_k\cap 1\mH_k[B]$; let $q\in \til{B}^*$ be the label of that path. By construction, $[q]_{H_k}=[r]_{H_k}$, hence $[q]_G=[r]_G$ since the expansion $H_k\twoheadleftarrow G$ is $k$-stable, and therefore also $[q]_G=[p]_G$. Finally, the canonical morphism $\varphi_u\colon \mH_k\twoheadrightarrow {\mathcal X}_1$ (restricted to $1\mH_k[B]$) maps $\pi=\pi_1^{\mA_k\cap 1\mH_k}(q)$ to the path $
\pi_u^{\mX_1}(q)$ with initial vertex $u=\varphi_u(1)$ and terminal vertex $v=\varphi_u(v')$ and label $q$. 
If we ignore the labelling then the latter path is the sequence $q$ of edges in $\mE$ which forms a path $u\longrightarrow v$.
Altogether, $q$ forms a path $u\longrightarrow v$ in $\mE$.
\end{proof}
This proof sheds some light on the r\^oles that the components of
$\mZ_k$ play in the transition $H_k\leadsto G_{k+1}$.  If there is a word $p$ with $\co(p)=A$ and $|A|=k+1$ such that
$p$ forms a path $u\longrightarrow v$ in $\mE$,  and some letter $a\in
A$ does not belong to the $H_k[A]$-content of $p$ then the subgraph
$\sfCE(H_k,\mA_k;B)$ of $\ol{\sfCE(H_k,\mA_k;\bP_A)}$ (for
$\mA=\langle A\rangle$ and $B=A\setminus \{a\}$) guarantees that the
next group $G_{k+1}$ avoids 
\textsl{every} relation $p=r$ for any $r\in \til{B}^*$
(compare Remark~\ref{rmk:unfolding}) unless there exists a word
$q\in \til{B}^*$ such that $[p]_{H_k}=[q]_{H_k}$ and $q$ forms a path
$u\longrightarrow v$ in $\mE$. From this point of view, 
namely to avoid all relations 
that would obstruct Lemma~\ref{lem:contentconnected}, 
it would be sufficient to let $\mZ_k$ be comprised of all
graphs $\sfCE(H_k,\mA_k;B)$ of the mentioned kind  (after making them
weakly complete by extending edges to $2$-cycles whenever needed).{\footnote{{This means that only ``basic'' $B$-coset extensions of type $\mathsf{CE}(G,\mK;B)$ as in~\eqref{eq:CE(K,B)} would be sufficient for proving Lemma~\ref{lem:contentconnected}}.}} 
However, when
attempting this approach, namely letting $\mZ_k$ be comprised
of just all graphs of the mentioned form, the authors failed to prove $k$-stability of
the expansion $H_k\twoheadleftarrow G_{k+1}$, and it is not clear
whether or not $k$-stability can be achieved by this procedure.
Hence, except for the graphs $\sfCE(H_k,\mA_k;B)$, which appear as
subgraphs of the full coset extensions $\sfCE(H_k,\mA_k;\bP_A)$, all
the machinery used to set up the graph $\mZ_k$ --- (augmented) clusters, (augmented)
full coset extensions, all of Section~\ref{sec:2results} --- serves to achieve $k$-stability of the transition $H_k\leadsto G_{k+1}$.

\medskip
If, in Lemma~\ref{lem:contentconnected}, $[p]_G=1$ then necessarily
$u=v$ since in this case the path $\pi^\mG_1(p)$ is closed and the
canonical morphism $\varphi_u\colon \mG\twoheadrightarrow \mX_1$ maps
this path onto the closed path 
$\pi_u^{\mX_1}(p)$. The path $p$ in $\mE$ obtained by
ignoring the labelling
then clearly is also closed. Iterated
application of Lemma~\ref{lem:contentconnected} leads to the following;
for the definition of a content function $\mathrm{C}$ the reader should recall Definition~\ref{def:G-content}.
\begin{Cor}\label{cor:crucial content} Let $p\in \til{E}^*$ be a word which forms a path $u\longrightarrow v$ in $\mE$; then there exists a word $q\in \til{E}^*$ which uses only letters (i.e.\ edges) from the content $\mathrm{C}([p]_G)$ (and/or their inverses) such that $[p]_G=[q]_G$ and $q$ forms a path $u\longrightarrow v$ in $\mE$. If $\mathrm{C}([p]_G)=\varnothing$, then $u=v$ and $q$ is  the empty word. If $\mathrm{C}([p]_G)\ne\varnothing$, then the graph $\langle \mathrm{C}([p]_G)\rangle=\langle \co(q)\rangle$ is connected and contains the vertices $u$ and $v$. 
\end{Cor}

\subsection*{Acknowledgement} We thank the anonymous referee for comments, corrections 
and suggestions for improvements in our presentation.

\end{document}